\documentclass[a4paper]{amsart}
\usepackage[utf8]{inputenc}
\usepackage[T1]{fontenc}
\usepackage{lmodern}
\usepackage{amsmath, amsthm, amssymb, amsfonts}
\usepackage{amsxtra} %for sptilde
\usepackage[all]{xy}
\usepackage{nicefrac,mathtools}
\usepackage[shortlabels, inline]{enumitem}
\usepackage{microtype}
\usepackage{hyperref}
\usepackage{graphicx}
\usepackage[all]{xy}
\usepackage{xcolor}
\usepackage{tikz-cd}
\usepackage[lite]{amsrefs}
\usepackage{thmtools}

\numberwithin{equation}{section}
\theoremstyle{plain}
\newtheorem{theorem}[equation]{Theorem}

\newtheorem{lemma}[equation]{Lemma}

\newtheorem{proposition}[equation]{Proposition}
\newtheorem{deflem}[equation]{Definition and Lemma}
\newtheorem{corollary}[equation]{Corollary}

\theoremstyle{definition}
\newtheorem{definition}[equation]{Definition}

\theoremstyle{remark} \newtheorem{remark}[equation]{Remark}
\newtheorem{remarks}[equation]{Remarks} \newtheorem*{remark*}{Remark}
\newtheorem*{remarks*}{Remark} \newtheorem{example}[equation]{Example}
\newtheorem{obs}[equation]{Observation}

\newcommand*{\Cont}{\mathrm C}%continuous functions
\newcommand*{\Contc}{\mathrm{C_c}}

\newcommand*{\Cc}{\mathrm{C_c}}%continuous functions with compact support
\newcommand{\Contz}{\mathrm{C_0}}

\newcommand{\base}[1][G]{{#1}^{(0)}}
\newcommand{\dd}{\mathrm{d}}
\newcommand{\7}{\backslash}
\newcommand{\Id}{\textup{Id}}
\newcommand{\supp}{{\textup{supp}}}

\newcommand*{\Hilm}[1][H]{\mathcal #1}% Hilbert module
\newcommand*{\Hils}[1][M]{\mathcal #1}% Hilbert module
\newcommand{\Bound}{\mathbb B}
\newcommand*{\Cst}{\mathrm{C}^*}%C*-algebra
\newcommand{\A}{\mathcal{A}}

\newcommand{\KMS}[1][\beta]{\textup{KMS}\textsubscript{\(#1\)}}
\newcommand{\Ind}{\mathrm{Ind}}
\newcommand{\IndP}[1][x]{\mathrm{Ind}_{\Cst(G_{#1}^{#1};\A(#1))}^{\Cst(G;\A)}}

\newcommand*{\inpro}[2]{\langle#1, #2\rangle}% right inner products
\newcommand*{\binpro}[2]{\bigl\langle#1, #2\bigr\rangle}% right inner
\newcommand*{\Binpro}[2]{\Bigl\langle#1, #2\Bigr\rangle}% right inner
\newcommand*{\Linpro}[2]{\langle\!\langle#1, #2\rangle\!\rangle}% right inner products
\newcommand{\abs}[1]{\lvert #1\rvert}
\newcommand{\norm}[1]{\lvert\!\lvert #1\rvert\!\rvert}
\newcommand{\C}{\mathbb{C}}
\newcommand{\R}{\mathbb{R}}

\newcommand{\N}{\mathbb{N}}

\newcommand{\iso}{\simeq}

\newcommand{\e}{\mathrm{e}}
\newcommand{\Mat}{\mathrm{M}}

\newcommand*{\Star}{$^*$}
\newcommand*{\nb}{\nobreakdash}

\newcommand*{\defeq}{\mathrel{\vcentcolon=}}
\newcommand{\inverse}{^{-1}}
\newcommand{\mi}{\mathrm{i}}

\newcommand{\etale}{{\'e}tale}

\title[KMS states on Fell bundle \(\mathrm{C}^*\)-algebras]{KMS states
  on the \(\mathrm{C}^*\)-algebras of Fell bundles over étale
  groupoids}

\author{Rohit Dilip Holkar} \email{rohit.d.holkar@gamil.com}

\author{Md Amir Hossain}
\email{mdamirhossain18@gmail.com}

\address{Department
	of Mathematics, Indian Institute of Science Education and Research
	Bhopal, Bhopal Bypass Road, Bhauri, Bhopal 462 066, Madhya Pradesh,
	India.}

\thanks{\emph{Subject class.} 46L55, 46L30, 47L65, 46B22, 22A22.}

\begin{document}
\begin{abstract}
  Let \(p\colon \mathcal{A} \to G\) be a saturated Fell bundle over a
  locally compact, Hausdorff, second countable, {\'e}tale
  groupoid~\(G\), and let \(\mathrm{C}^*(G;\mathcal{A})\) denote its
  full \(\mathrm{C}^*\)-algebra. We prove an
  integration-disintegration theorem for KMS states on
  \(\mathrm{C}^*(G;\mathcal{A})\) by establishing a one-to-one
  correspondence between such states and fields of measurable states
  on the \(\mathrm{C}^*\)-algebras of the Fell bundles over the
  isotropy groups. This correspondence is established for certain
  states on \(\mathrm{C}^*(G;\mathcal{A})\) also. While proving this
  main result, we construct an induction
  \(\mathrm{C}^*\)-correspondence
  between~\(\mathrm{C}^*(G;\mathcal{A})\) and the
  \(\mathrm{C}^*\)\nb-algebra of an isotropy Fell bundle.  We
  demonstrate our results through many examples such as groupoid
  crossed products, twisted groupoid crossed products, \(G\)-spaces
  and matrix algebras~\(\mathrm{M}_n(\mathrm{C}(X))\otimes A\).  While
  studying the matrix algebra~\(\mathrm{M}_n(\mathrm{C}(X))\), we
  propose a groupoid model for it. While demonstrating our main result
  for this groupoid model, we provide a solution to the Radon--Nikodym
  problem for the groupoid used in this model.
\end{abstract}

\maketitle
\setcounter{tocdepth}{1}
\tableofcontents

\section*{Introduction}
\label{sec:intro}

The KMS, Kubo--Martin--Schwinger, state is the \(\Cst\)\nb-algebraic
formulation of Gibbs equilibrium
state~\cite{Bratteli-Robinson1981Oper-alg-Quan-sta-mech-part-2} which,
mathematically, makes sense for any real \(\Cst\)\nb-dynamics whether
the dynamics induced from a physical system or not. In recent times,
Mathematicians seem interested in describing the KMS states because
KMS states encode information about the dynamics, and they are useful
while studying the traces on the \(\Cst\)\nb-algebras, for
example~\cite{Neshveyev2010Traces-on-crossed-product}~\cite{Christensen-thesis},\cite{Christensen-Klaus2016Finite-digraphs-and-KMS-states},\cite{Christensen2018Symmetries-of-KMS-symplex}.

In his thesis~\cite{Renault1980Gpd-Cst-Alg}, Renault described the KMS
state of the \(\Cst\)\nb-algebras of a \emph{principal} \etale\
groupoid \(G\) with certain quasi-invariant probability measure on
\(G^{(0)}\) where the dynamics comes from a one-cocycle \(c\). Quite
later, Neshveyev~\cite{Neshveyev2013KMS-states} noticed that for a
non-principal \etale\ groupoid \(G\), not all KMS states could be
described by the quasi-invariant probability measure on \(G^{(0)}\).
Neshveyev~\cite{Neshveyev2013KMS-states} generalised the result of
Renault as follows: there is a one-to-one correspondence between KMS
state on \(\Cst(G)\) for an \etale\ groupoid~\(G\) with pairs
\((\mu,[\phi])\) where \(\mu\) is a quasi-invariant measure
on~\(G^{(0)}\) and \(\phi=\{\phi_u\}_{u\in\base[G]}\) is a
\(\mu\)\nb-measurable field of states on the \(\Cst\)\nb-algebras of
the isotropy groups. The KMS state and the modular function of~\(G\)
for the quasi-invariant measure~\(\mu\) are related. Using this
characterisation, Neshveyes~\cite{Neshveyev2013KMS-states} studies the
KMS states on Toeplitz algebra of \(ax+b\) semigroup and tracial
states on the \(\Cst\)\nb-algebras of transformation groupoid.

In~\cite{Afsar-Sims2021KMS-states-on-Fell-bundle-Cst-alg}, Afsar and
Sims generalise the result of Neshveyev to a class of Fell bundles
over \etale\ groupoid. They consider the Fell bundles whose fibres are
\emph{singly generated}; briefly speaking, each fibre of such a Fell
bundle, including the Hilbert modules, is \emph{unital}, in an
appropriate sense. They use their result to study the KMS states of
the \(\Cst\)\nb-algebra of a twisted
\(k\)\nb-graph~\cite{Kumjian-Pask_Sims2015twisted-higher-rank-graph-alg}. Afsar
and Sims also established Neshveyev's theorem for the twisted
\(\Cst\)\nb-algebra of an \etale\ groupoid.

In present article, we generalise the above results of Neshveyev or
Afsar and Sims to \emph{general}, saturated Fell bundles over \etale\
groupoids. We also gives many illustrations and examples of this
result. Our proof closely follows the chief subtext idea in
Neshveyev's~\cite{Neshveyev2013KMS-states} article and some techniques
of Afsar and
Sims'~\cite{Afsar-Sims2021KMS-states-on-Fell-bundle-Cst-alg}. To chalk
out the details, assume that \(G\) is an \etale\ groupoid and
\(u\in \base[G]\). One of the chief idea of Neshveyev is to induce a
state of the \(\Cst\)\nb-algebra~\(\Cst(G_u^u)\) of the
isotropy~\(G_u^u\) group to a state of the groupoid
\(\Cst\)\nb-algebra. And the other main idea is to use Renault's
disintegration theorem for getting a state of~\(\Cst(G_u^u)\) from a
state of the groupoid \(\Cst\)\nb-algebra. Afsar and Sims, and we too,
adopt these ideas. However, we formalise the process of inducing
states by using a \(\Cst\)\nb-correspondences---we call this
correspondence \emph{the induction correspondence} see
Section~\ref{sec:inducing-states}. This correspondence appears in
Example~3.14 of~\cite{Holkar2017Construction-of-Corr} in an
undecorated basic form; and can be traced back back to Mackey's
induction process. The mathematical techniques for constructing
induction correspondence are inspired from GNS construction and the
induction of representation proposed by
Rieffel~\cite{Rieffel1974Induced-rep}. The bundle over groupoid, in
other word the coffecients in \(\Cc(G)\), make this processes
nontrivial.

Next, we achieve the generalisation of the result of Afsar and Sims by
\emph{removing} the condition of singly generated Fell bundle. One of
our main task was to solve the many technical issues raised after
removing this condition.

Once the induction and `restriction' of states is established
(Theorem~\ref{prop:bij-betwen-states-and-fields-of-states}), the
relation for \KMS~states become clear. For our main result,
Theorem~\ref{thm:KMS-state}, we consider a Fell bundle
\(p\colon \A \to G\) over an \etale\ groupoid~\(G\) and the
\KMS~states on \(\Cst(G;\A)\) in which the dynamics is induces from a
real-valued 1-cocycle on~\(G\). Then our Theorem~\ref{thm:KMS-state}
establishes a one-to-one correspondence between \KMS~states on
\(\Cst(G,\A)\) and pairs \(([\mu],\phi)\) as in Neshveyev's
theorem~\cite{Neshveyev2013KMS-states}.

For the inverse temperature \(\beta =0\) the
Theorem~\ref{thm:KMS-state} gives a description of the invariant
traces of \(\Cst(G;\A)\) under the dynamics, which has particular
importance in Elliott classification program.

A noteworthy observation here is Condition~\eqref{Cond:II} of
Theorem~\ref{thm:KMS-state}. This condition may seem technical,
however, if the Fell bundle \(p\colon \A\to G\) is trivial Fell bundle
over a group, this condition is essentially a trace condition, namely,
\(\tau(yx^*x) = \tau(xyx^*)\). In fact, even if \(\A\) is not the
trivial bundle, still this condition is a trace condition for a
state~\(\tau\) on~\(\Cst(G^x_x;\A(x))\) where \(x\in \base\). As this
condition shows up naturally in the computation, it should be regarded
as a generalisation of the trace condition.

On a different note, Cuntz and Vershik mention in the proof of
Lemma~2.5 of~\cite{Cuntz-Vershik2013Cst-Agl-asso-with-endomorphisms}
that for trivial stabilisers, invariant measures on the space of units
of a transformation groupoid extend uniquely to a trace on the
\(\Cst\)\nb-algebra of the groupoid.  In their words,
\begin{quotation}
  ``The fact that an invariant measure on \(K\) extends uniquely to a
  trace on the crossed product, if all the stabiliser groups are
  trivial, is well-known but not easy to pin down in the literature''.
\end{quotation}
They provide a quick proof of this fact in this Lemma~2.5
of~\cite{Cuntz-Vershik2013Cst-Agl-asso-with-endomorphisms}. The above
mentioned works of Renault, Neshveyev, Afsar and Sims are general
answers to this question. Our Theorem~\ref{thm:KMS-state} offers an
answer to this fact in a much more general situation.

An immediate application of our result is that: we characterise the
\KMS~states and traces on the groupoid (twisted) crossed product
(Corollary~\ref{kms-state-gpd:cros}) for an \etale\ groupoid.

One more result that we prove is
Proposition~\ref{int-dis-state-upp-semi}, wherein we characterise the
states and traces on a \(\Contz(X)\)\nb-algebra. In this case, we
remove the \emph{centraliser hypothesis} and hence this result is not
a particular case of
Theorem~\ref{prop:bij-betwen-states-and-fields-of-states}. Using
Proposition~\ref{int-dis-state-upp-semi}, we give a characterisation
of states and traces on the \(\Cst\)\nb-algebras of locally compact,
Hausdorff, second countable group bundle and the \(\Cst\)\nb-algebras
of Fell bundle over locally compact, Hausdorff, second countable group
bundle; this is essentially a generalisation of Neshveyev's theorem
for \emph{non-\etale} group bundle.

It is well-known that the groupoid model for the matrix algebra
\(\Mat_n(\C)\) is the groupoid of the trivial equivalence on
\(\{1,2,3,\cdots, n\}\) for \(n\in \N\). A physical interpretation of
this groupoid is that it represents the jumps of electron in the
hydrogen atom. For a compact, Hausdorff space~\(X\), we model the
\(\Cont(X)\)\nb-algebra \(\Mat_n(\Cont(X))\) as a groupoid. This
interesting model can be thought of as atomic orbits `extended' by the
space~\emph{X}, see Section~\ref{sec:kms-stat-text}. Using this model
and our earlier result, we illustrate the \KMS~states on the
\(\Cst\)\nb-algebra \(\Mat_n(\Cont(X)) \otimes A\). While studying the
KMS states, we investigate the quasi-invariant measures on the unit
space of the underlying groupoid of \(\Mat_n(\Cont(X))\), and we offer
an answer to the Radon--Nikodym
problem~\cite{Renault2005The-R-N-problem-for-appx-proper-equiv-relation,
  Renault2001AF-equiv-rela-and-cocycle} for this groupoid.

Using a locally compact, Hausdorff \(G\)\nb-space \(X\) one can
associate a Fell bundle \(\A\) over \(G\) and this assignment is
funtorial. Using this identification we can think the isotropy
\(\Cst\)\nb-algebra of \(\Cst(X\rtimes G)\) as a \(\Cst\)\nb-algebra
of certain transformation group. Hence, using our
Theorem~\ref{thm:KMS-state} and Neshveyev's Theorem~1.1 of
\cite{Neshveyev2013KMS-states}, we obtain a get a relationship between
\KMS~states on \(\Cst(X\rtimes G)\) and a triple satisfying some
conditions (see Proposition~\ref{Prop:kms-state:G-space}).
 
\medskip

\paragraph{\itshape Structure of the article:}

Section~\ref{sec:prelim} contains a background of Fell bundles and
it's \(\Cst\)\nb-algebras, some results about approximate identity
which we shall use in Section~\ref{sec:an-example-corr}. In
Section~\ref{sec:an-example-corr}, we construct an induction
correspondence from the \(\Cst\)\nb-algebra of a Fell bundle to the
\(\Cst\)\nb-algebras of Fell bundle over isotropy groups, and we also
show that this induction correspondence possesses \emph{a unit cyclic
  vector}. Section~\ref{sec:int-dis-thm-state} contains the
integration-disintegration theorm for certain states on \(\Cst(G;\A)\)
(Theorem~\ref{prop:bij-betwen-states-and-fields-of-states}). Section~\ref{sec:main-result-KMS-state}
contains the main result of this article: the
integration-disintegration theorem for KMS states on \(\Cst(G;\A)\)
(Theorem~\ref{thm:KMS-state}).  Section~\ref{sec:gpd-crossed:product}
contains an application of our result in groupoid crossed product and
groupoid twisted crossed product.

In section~\ref{sec:state-upp-semi-con-bundle}, we prove the
integration-disintegration theorems for states and traces on the
section algebra of an upper semicontinuous bundle of
\(\Cst\)\nb-algebra and few applications.

In Sections~\ref{sec:kms-stat-text} and~\ref{sec:kms-state-G-space},
we study \(\KMS\) states on \(\Mat_n(\mathrm{C}(X))\otimes A\) and
\(G\)\nb-spaces, respectively.
\section{Preliminaries}
\label{sec:prelim}

We follow~\cite{Renault1980Gpd-Cst-Alg,Muhly1999Coordinates} for the
basic theory of groupoids. The groupoids and spaces considered in this
article are Hausdorff, locally compact and second countable in the
sense of Munkres~\cite{Munkress1975Topology-book}. Our references for
\(\Cst\)\nb-algebras are~Davidson's~\cite{Davidson1996Cst-by-Examples}
and Murhpy's~\cite{Murphy-Book} books; for Hilbert modules we refer to
Lance's monograph~\cite{Lance1995Hilbert-modules}; finally, Bratteli
and Robinson~\cite{Bratteli-Robinson1981Oper-alg-Quan-sta-mech-part-2}
is the reference for KMS states. We assume that the reader is familiar
with basics of groupoids and Haar systems on them, to be precise, the
standard material in Chapter~I in~\cite{Renault1980Gpd-Cst-Alg}.

For a groupoid topological~\(G\), we denote its space of units by
\(G^{(0)}\); the range and the source maps by~\(r\) and~\(s\),
respectively; and the set of composable pairs,
\(\{(\gamma,\eta)\colon G\times G : s(\gamma)=r(\eta)\}\), by
\(G^{(2)}\). For \(\gamma\in G\), \(\gamma\inverse\) denotes its
inverse. Let \(U\subseteq G\). Then for subsets
\(A,B\subseteq \base\), \(U^A,U_B\) and \(U^A_B\) have the standard
meanings, namely,
\(U^A=\{\gamma\in U : r(\gamma)\in A\}, U_B=\{\gamma\in U :
s(\gamma)\in B\}\) and \(U^A_B=U^A\cap U_B\). If \(A=\{x\}, B=\{y\}\)
are singletons, then we simply write \(U^x,U_y\) and \(U^x_y\). For
\(x\in \base\), the group \(G^x_x\subseteq G\) is called the isotropy
at \(x\). Due to the continuity of the range and source maps, the
bundle of groups \(\cup_{x\in \base} G_x^x\subseteq G\) is a closed
subgroupoid of~\(G\)---this bundle is the inverse image of the
diagonal in \(\base\times\base\) under the mapping
\(G \to \base\times \base, \gamma \mapsto (r(\gamma),s(\gamma))\).

\begin{remark}\label{rem:obs-of-bise-out-iso}
  Let \(G\) be a locally compact, Hausdorff, second countable, \etale\
  groupoid and \(\textup{Ib}(G) = \cup_{x\in \base} G_x^x\) denotes
  the isotropy bundle of \(G\). Then the set
  \(G\setminus \textup{Ib}(G)\) is open in \(G\) and has a basis of
  bisections \(U\subseteq G\) such that \(r(U)\cap s(U)
  =\emptyset\). Moreover, for such \(U\), \(U\cap \base =\emptyset\).
\end{remark}

If the range (or, equivalently, source) map of a topological groupoid
is a local homeomorphism, then we call the groupoid \emph{\etale.} Let
\(G\) be an {\'e}tale groupoid. An open subset \(U\subseteq G\) is
called a bisection if \(r(U), s(U) \subset \base\) are open and
\(r|_U\colon U\to r(U)\) and \(s|_U\colon U\to s(U)\) are
homeomorphisms; here note that since \(r\) (or \(s\)) are local
homeomorphisms, \(r(U)\subseteq \base[G]\)
(\(s(U)\subseteq \base[G]\), respectively) is an open set. An \etale\
groupoid has a basis consisting of bisections.

Continuing the last discussion, since \(r\) is a local homeomorphism,
it follows that \(G^x\) is discrete for all \(x\in G^{(0)}\). Let
\(\lambda_x\) be a counting measure on \(G^x\). The family of measure
\(\{\lambda_x\}_{x\in G^{0}}\) constitutes a Haar system for the
{\'e}tale groupoid \(G\),
see~\cite[Defintion~I.2.2]{Renault1980Gpd-Cst-Alg}.

For~\(G\) as above, let \(\mu\) be a measure on the sapce of units
of~\(G\). Then~\(\mu\) induces measures~\(\nu\) and~\(\nu\inverse\)
on~\(G\) as follows:
\[
  \nu(f) = \int_{G^{(0)}} \biggl(\sum_{\gamma\in G^x} f(\gamma)
  \biggr) \dd\mu(x) \quad\text{ and }\quad\nu^{-1}(f) = \int_{G^{(0)}}
  \biggl( \sum_{\gamma\in G_x}f(\gamma) \biggr)\mathrm{d}\mu
\]
where \(f\in \Contc(G)\).  The measure \(\mu\)~on \(\base\) is called
\emph{quasi-invariant} if \(\nu\) and \(\nu\inverse\) are equivalent
measures. The Radon--Nikodym derivative \(\dd\nu/\dd\nu\inverse\) is a
1-cocycle on \(G\) and it's called a modular function of
\(G\)~(cf.\cite[Page~23.]{Renault1980Gpd-Cst-Alg}). If the modular
function \(\dd\nu/\dd\nu\inverse\) is \(1\), then the measure \(\mu\)
is called invariant.

Following is an alternate and handy way to define a quasi-invariant
measure on an \etale\ groupoid: let \(U\subseteq G\) be a
bisection. For \(x\in s(U)\), let \(u_x\in U\) denote the unique arrow
starting at \(x\).  Define the homeomorphism
\(T_U\colon r(U)\to s(U)\) by \(T_U = s|_U \circ r|_U\inverse\). A
measure~\(\mu\) on~\(\base[G]\) is quasi-invariant \emph{iff} there is
a measurable function \(\Delta_\mu\colon G\to (0,\infty)\) such that
\begin{equation}\label{eq:quasi-inv-etale}
  \int_{r(U)} k(T_U(x))\mathrm{d}\mu(x) =\int_{s(U)}
  k(x)\Delta_{\mu}(u_x) \mathrm{d} \mu(x)
\end{equation}
for all bisection \(U\) and for all functions \(k\in \Contc(G)\) which
are supported in~\(s(U)\).

\begin{remark}\label{remk:quasi-inv-positive-fun}
  Given \(k\in \Cc(G)\) can be written as a linear combination of four
  non-negative compactly supported functions on \(G\). Thus, if we can
  show that the measure~\(\mu\) satisfies
  Equation~\eqref{eq:quasi-inv-etale} for all non-negative compactly
  supported function on \(G\), then by linearity of the integral,
  \(\mu\) is a quasi-invariant with modular function \(\Delta_{\mu}\).
\end{remark}

\begin{definition}[Inner product \(A\)-module]
  Let \(A\) be a \(\Cst\)\nb-algebra.  A (right) inner product
  \(A\)\nb-module is a (right) \(A\)\nb-module \(X\) with a map
  \(\inpro{\cdot}{\cdot}_A \colon X \times X \to A\) satisfying the
  following properties:
  \begin{enumerate}
  \item
    \(\inpro{x}{\lambda y+\mu z}_A = \lambda \inpro{x}{y}_A + \mu
    \inpro{x}{z}_A\),
  \item \( \inpro{x}{y\cdot a}_A = \inpro{x}{y}_A\; a\) ,
  \item \(\inpro{x}{y}^*_A =\inpro{y}{x}_A\),
  \item \(\inpro{x}{x}_A \geq 0 \),
  \item \(\inpro{x}{x}_A =0\) if and only if \(x=0\)
  \end{enumerate}
  for given vectors \(x,y,z \in X\), scalars \(\lambda, \mu \in \C\)
  and~\(a \in A\).
\end{definition}

For a vector \(x\) in an inner product \(A\)\nb-module \(X\),
\begin{equation}
  \norm{x} = \norm{\inpro{x}{x}_A}^{\frac{1}{2}}\label{eq:Hilb-mod-norm-general}
\end{equation}
defines a norm making \(X\) a normed linear \(\C\)\nb-vector space. We
often ignore the subscript of~\(\inpro{\cdot}{\cdot}_A\), unless
required, and simply write~\(\inpro{\cdot}{\cdot}\). Moreover, instead
of writing the pair \((A, \inpro{\cdot}{\cdot})\) as an inner product
module, we abuse the language and call the module~\(A\) itself the
inner product module.

\begin{definition}[Hilbert module] A Hilbert module over a
  \(\Cst\)\nb-algebra \(A\) is an inner product \(A\)\nb-module which
  is complete in the norm defined by
  Equation~\eqref{eq:Hilb-mod-norm-general}.
\end{definition}

\noindent A left Hilbert \(A\)\nb-module can be defined analogously;
the inner product in this case is linear in the first variable and
conjugate linear in the second one; we denote \emph{a left inner
  product} by \(\Linpro{\cdot}{\cdot}\). We reserve the phrase
``Hilbert module'' for a right Hilbert module where as explicitly
mention if the module is a \emph{left} module. A (left or right)
Hilbert \(A\)\nb-module is called \emph{full} if the ideal
\(I \defeq \mathrm{span}\{\inpro{x}{y} : x, y \in X\}\) dense in
\(A\). All useful Hilbert modules in this article are full.  We denote
\(\Bound(X)\) the set of all adjointable operator on the Hilbert
module \(X\).
\begin{definition}[Hilbert bimodule]\label{def:Hilbert-bimod}
  Let \(A\) and \(B\) be two \(\Cst\)\nb-algebras. A Hilbert
  \(A\)-\(B\)\nb-bimodule is an \(A\)-\(B\)\nb-bimodule~\(X\) equipped
  with an \(A\)\nb-valued inner
  product~\(\prescript{}{A}{\Linpro{\cdot}{\cdot}}\) and a
  \(B\)\nb-valued inner product~\(\inpro{\cdot}{\cdot}_B\) such that:
  \begin{enumerate}
  \item \((X, \prescript{}{A}{\Linpro{\cdot}{\cdot}})\) and
    \((X, \inpro{\cdot}{\cdot}_B)\) are left and right Hilbert modules
    over \(A\) and \(B\), respectively;
  \item \(\inpro{a\cdot x}{y}_B =\inpro{x}{a^*\cdot y}_B\);
  \item
    \( \prescript{}{A}{\Linpro{x\cdot b}{y}} =
    \prescript{}{A}{\Linpro{x}{y\cdot b^*}}\);
  \item
    \(\prescript{}{A}{\Linpro{y}{z}} \cdot x = z \cdot \inpro{x}{y}_B
    \)
  \end{enumerate}
  for \(x,y \in A\), \(a\in A\) and \(b\in B\).
\end{definition}

An \emph{imprimitivity bimodule} is a Hilbert bimodule that is full on
both sides.

\begin{definition}[\(\Cst\)-correspondence]
  A \(\Cst\)\nb-correspondence from a \(\Cst\)\nb-algebra \(A\) to
  another one, \(B\), is a pair \((\Hilm, \phi)\) where \(\Hilm\) is a
  Hilbert \(B\)\nb-module and \(\phi\colon A\to \Bound(\Hilm)\) is a
  nondegenerate representation.
\end{definition}

\noindent We often skip writing the representation~\(\phi\) in a
\(\Cst\)\nb-correspondence and simply refer the Hilbert module itself
as a \(\Cst\)\nb-correspondence. A vector \(\xi\in \Hilm\) is called
\emph{cyclic} for the representation~\(\phi\) if the closed linear
span of \(\{\phi(a)\xi:a\in A\}\) equals~\(\Hilm\). In
Section~\ref{sec:an-example-corr}, we show that a Fell bundle over
\etale\ groupoid has a canonical \(\Cst\)\nb-correspondence associated
with it called \emph{the induction correspondence}; in
Section~\ref{sec:inducing-states} we prove that this induction
correspondence possesses a unit cyclic vector.

\subsection{Some results about approximate identities}
\label{sec:some-results-about}

In this section, we collect or prove some discrete results involving
approximate identities and Hilbert modules.

\begin{lemma}
  \label{lem:appx-unit-for-full-Hilbert-mod}
  Assume that \(X\) is a Hilbert \(A\)\nb-module (not necessarily
  full), \(I\subseteq A\) is the closed ideal generated by
  \(\{\inpro{x}{y} \in A : x,y\in X\}\), and \((u_\lambda)_{\lambda}\)
  is an approximate identity of~\(I\). Then
  \(\norm{ x u_\lambda- x }\to 0\) for \(x\in X\).
\end{lemma}
\begin{proof}
  For any \(x\in X\), we have
  \begin{multline*}
    \norm{xu_\lambda-x}^2=\norm{\inpro{xu_\lambda-x}{xu_\lambda-x}} =
    \norm{\inpro{xu_\lambda}{x}u_\lambda-\inpro{xu_\lambda}{x}-\inpro{x}{x}u_\lambda-\inpro{x}{x}}
    = \\
    \norm{(\inpro{x}{x}u_\lambda)^*u_\lambda-(\inpro{x}{x}u_\lambda)^*-\inpro{x}{x}u_\lambda-\inpro{x}{x}}
    =
    \norm{u_\lambda\inpro{x}{x}u_\lambda-u_\lambda\inpro{x}{x}-\inpro{x}{x}u_\lambda-\inpro{x}{x}}.
  \end{multline*}
  Since \((u_\lambda)_\lambda\) is an approximate identity of \(I\)
  and the norm is a continuous function, the last term of the above
  equation \(\to 0\) as \(\lambda\to \infty\).
\end{proof}

\begin{lemma}\label{lem-squ-of app-idenetity}
  Let \(A\) be a \(\Cst\)\nb-algebra and \((u_\lambda)_\lambda\) an
  approximate identity of \(A\).
  \begin{enumerate}[leftmargin=*]
  \item Then \(\lim_{\lambda}au^2_\lambda=a\) and
    \(\lim_{\lambda}u^2_\lambda a=a\) for all \(a\in A\).
  \item The iterated limits,
    \(\lim_{\alpha}\lim_{\lambda} a u_\lambda u_\alpha = a\) and
    \( \lim_{\lambda} \lim_{\alpha} u_\alpha u_\lambda a=a\) for all
    \(a\in A\).
  \end{enumerate}
\end{lemma}
\begin{proof}
  \noindent (1): The trangle inequality implies that
  \begin{multline*}
    \norm{a-au^2_\lambda}\leq
    \norm{a-au_\lambda}+\norm{au_\lambda-au^2_\lambda}=\norm{a-au_\lambda}+
    \norm{a-au_\lambda} \norm{u_\lambda} \leq 2\norm{a-au_\lambda}.
  \end{multline*}
  Since \((u_\lambda)\) is an approximate identity, the last term
  converges to \(0\). The second claim follows from a similar
  computation.

  \noindent (2): Let \(\epsilon>0\) be given. Similar to the last
  computation, we have
  \begin{multline*} \norm{a-au_\lambda u_\alpha}\leq
    \norm{a-au_\alpha}+\norm{au_\alpha-au_\lambda u_\alpha}\leq
    \norm{a-au_\alpha}+ \norm{a-au_\lambda} \norm{u_\alpha}\\ \leq
    \norm{a-au_\alpha}+ \norm{a-au_\lambda}.
  \end{multline*}
  Choose an index \(i\) such that \(\alpha,\lambda\geq i\) implies
  \(\norm{a-au_\alpha}< \epsilon/2\) and
  \(\norm{a-au_\lambda}<\epsilon/2\) so that the last term is
  \(<\epsilon\).
\end{proof}

\begin{lemma}\label{lem:convergence-to-id-in-nondgen-rep}
  Consider a nondegenerate representation
  \(\pi\colon A\to \Bound(\Hils)\) of a \(\Cst\)\nb-algebra \(A\) on a
  Hilbert space \(\Hils\).  Let \(\{u_\lambda\}_\lambda\) be a net in
  the unit ball of \(A\) such that \(\lim_\lambda (u_\lambda a)= a\)
  for all \(a\in A\). Then \(\pi(u_\lambda)x\to x\) in norm for each
  \(x\in \Hils\). In other words, \(\pi(u_\lambda)\to \Id_{\Hils}\) in
  the strong operator topology.
\end{lemma}

\begin{proof}
  Let \(\pi(A)\Hils\) denote the linear span of
  \(\{\pi(a)x:a\in A, x\in\Hils\}\). Choose a vector, say
  \(\sum_{i=1}^n\pi(a_i)x_i\), in~\(\pi(A)\Hils\) where
  \(a_i\in A, x_i\in \Hils\) and \(n\in \N\). Now compute the
  following limit in~\(\Hils\):
  \begin{multline*}
    \lim_\lambda \biggr(\pi(u_\lambda)
    \biggr(\sum_{i=1}^n\pi(a_i)x_i\biggr) \biggr) =
    \sum_{i=1}^n\lim_\lambda\pi(u_\lambda)\pi(a_i)x_i \\=
    \sum_{i=1}^n\pi(\lim_\lambda u_\lambda a_i)x_i
    =\sum_{i=1}^n\pi(a_i)x_i.
  \end{multline*}
  Thus the required claim about limit holds for vectors
  in~\(\pi(A)\Hils\).  Next the claim can be extended to
  \(x\in \Hils\) by an \(\epsilon/3\)\nb-argument.
  % we extend it whole of \(\Hils\) using the
  % hypothesis that \(\pi(A)\Hils\subseteq \Hils\) dense. Let
  % \(x\in \Hils\) be a given vector. Consider arbitrary
  % \(\epsilon>0\).
  % Choose \(y\in \pi(A)\Hils\) such that
  % \(\norm{x-y}<\epsilon/3\). Then
  % \begin{multline*}
  %   \norm{\pi(u_\lambda)x-x}\leq
  %   \norm{\pi(u_\lambda)x-\pi(\delta_\lambda)y}+
  %   \norm{\pi(u_\lambda)y-y}+\norm{y-x}\\
  %   \leq \norm{\pi(u_\lambda)} \norm{x-y}+
  %   \norm{\pi(u_\lambda)y-y}+\norm{y-x}\leq \norm{x-y}+
  %   \norm{\pi(u_\lambda)y-y}+\norm{y-x}.
  % \end{multline*}
  % We can choose \(\lambda_0\) sufficiently large so that the middle
  % summand in the last term is~\(<\epsilon/3\) for
  % all~\(\lambda\geq\lambda_0\); in turn, the right hand side of the
  % above equation \(< \epsilon\).
\end{proof}

\begin{corollary}\label{cor:cros-app-id}
  Let \(A\) be a \(\Cst\)\nb-algebra and \((u_\lambda)_\lambda\) an
  approximate identity of it. For any nondegenerate representation
  \(\pi\colon A\to \Bound(\Hils)\) on a Hilbert space~\(\Hils\),
  \(\pi(u_\lambda^2)\to \Id_{\Hils}\) and the iterated limits
  \(\lim_{\lambda}\lim_{\alpha}\pi(u_\lambda)\pi(u_\alpha)=
  \Id_{\Hils}= \lim_{\alpha} \lim_{\lambda}\pi(u_\lambda)
  \pi(u_\alpha)\) in the strong operator topology.
\end{corollary}
\begin{proof}
  Since \(\pi(u_\lambda u_\alpha) = \pi(u_\lambda)\pi(u_\alpha)\), all
  claims follow from Lemma~\ref{lem:convergence-to-id-in-nondgen-rep}.
\end{proof}

\subsection{Bundles of \(\Cst\)-algebras}
\label{sec:upper-semic-bundl}

We refer the reader to Appendix~C
of~\cite{Williams2007Crossed-product-Cst-Alg} for the standard
material about the bundles of Banach spaces and upper semicontinuous
bundles of~\(\Cst\)\nb-algebras over locally compact Hausdorff
spaces. Still, we shall keep recalling some results whenever
required. Following~\cite[Defintion
I.13.4]{Fell1988Representation-of-Star-Alg-Banach-bundles}, an
\emph{upper semicontinuous bundle of Banach spaces} is an open mapping
of spaces \(p\colon \mathcal{B} \to X\) where \(\mathcal{B}\) is the
\emph{total space}, \(X\) is the \emph{base space} and~\(p\) is the
\emph{bundle map}, and following conditions hold: each \emph{fibre}
\(\mathcal{B}_x \defeq p\inverse(x)\) over \(x\in X\) is a Banach
space so that the \textrm{norm} \(b\mapsto \norm{b}\), for
\(b\in \mathcal{B}\), is an upper semicontinuous function
\(\mathcal{B}\to \R\); the scalar
multiplication~\(\C \times \mathcal{B} \to \mathcal{B}\) is also a
continuous function. Moreover, the addition is continuous in following
sense: let \(\mathcal{B}^{(2)}\) denote the fibre product
\(\{(b,b')\in \mathcal{B}\times \mathcal{B} : p(b) = p(b')\}\). Then
addition is a continuous function
\(\mathcal{B}^{(2)} \to \mathcal{B}\) where the domain has subspace
topology of \(\mathcal{B} \times \mathcal{B}\). Finally, a condition
relating the topology on~\(X\) and the norms on fibres is fulfilled.

\begin{definition}[\cite{Williams2007Crossed-product-Cst-Alg}]
  An upper semicontinuous \(\Cst\)\nb-bundle (or upper semicontinuous
  bundle of \(\Cst\)\nb-algebras) is an upper semicontinuous Banach
  bundle \(p\colon \A\to X \) such that each fibre is a
  \(\Cst\)\nb-algebra and the following conditions hold:
  \begin{enumerate}
  \item the fibrewise multiplication is a continuous mapping
    \(\A^{(2)} \to \A \);
  \item the fibrewise involution, \(a \mapsto a^* \) for \(a\in \A\),
    defines a continuous mapping \(\A \to \A \).
  \end{enumerate}
\end{definition}

\noindent We consider \emph{only} those upper semicontinuous
\(\Cst\)\nb-bundles \(p \colon \A \to X\) for which the base
space~\(X\) is locally compact, Hausdorff and second countable, and
each fibre~\(\A_x\) over \(x\in X\) is a separable
\(\Cst\)\nb-algebra. As a consequence of this, the section algebra
\(\Contz(X;\A)\) is also separable. In case we need to emphasize this
assumption, we call the bundle \emph{a separable bundle}.

For a bundle as above, a continuous (or measurable) section has the
standard meaning. However, we shall simply call a continuous section a
\emph{section} unless the continuity needs an emphasize.

The bundle in last definition is said to have \emph{enough sections}
if for given~\(x \in X\) and~\(a\in \A_x\), there is a section~\(s\)
with \(s(x)=a\). If the base space~\(X\) is locally compact and
Hausdorff, then the bundle \(p\colon \A\to X\) has enough
section~\cite{Hofmann1977Bundles-and-sheaves-equi-in-Cat-Ban}.

For a bundle as in last definition, \(\Contz(X,\A)\) denotes the
\(\Cst\)\nb-algebra of sections vanishing at infinity.
\(\Contc(X; \A)\) denotes the dense \({}^*\)-subalgebra
of~\(\Contz(X,\A)\) consisting of continuous sections with compact
support.

It is customary to write either~\(p\) or~\(\A\) for a bundle
\(p\colon \A\to X\) when the base space and the bundle map are
clear. Recall the definition of \(\Contz(X)\)\nb-algebras
from~\cite[Appendix C]{Williams2007Crossed-product-Cst-Alg}. A
\(\Contz(X)\)\nb-algebra \(A\) can be thaught as a bundle of upper
semicontinuous \(\Cst\)\nb-algrebras over \(X\) and conversely, the
section algebras of a bundle of upper semicontinuous
\(\Cst\)\nb-algbra is a \(\Contz(X)\)\nb-algebra (see~\cite[Theorem
C.26]{Williams2007Crossed-product-Cst-Alg}). We denote the fibre of a
\(\Contz(X)\)\nb-algebra \(A\) at \(x\in X\) by \(A(x)\).

\begin{lemma}\label{lem:fiberwisw-appr-iden}
  Let \(p\colon \A \to X \) be an upper semicontinuous
  \(\Cst\)\nb-bundle. If \((u_n)_{n\in \N}\) is an approximate
  identity of \(\Contz(X;\A)\), then for each \(x\in X\) the net
  \((u_n(x))_{n\in \N}\) is an approximate identity of the
  \(\Cst\)\nb-algebra \(\A_x\).
\end{lemma}

\begin{proof}
  For each \(x\in X\), \((u_n(x))_{n\in \N}\) is a sequence of the
  fibre \(\A_x\) with
  \(\norm{u_n(x)} \leq \norm{u_n}_\infty \leq 1 \).  For natural
  numbers \(n \leq m\), we have \(u_n \leq u_m\). By positivity
  of~\(u_m-u_n\in \Contz(X;\A)\), there is \(f \in \Contz(X; \A)\)
  with \(u_m - u_n = f^*\cdot f\). Thus
  \(u_m(x) - u_n(x) =f(x)^*f(x)\) is a positive element in~\(\A_x\).

  Finally, let \(a \in \A_x\). As~\(p\) has enough sections, get a
  section \(f \in \Contz (X;\A)\) such that \(f(x) =a\). As the
  uniform convergence in~\(\Contz(X;\A)\) implies pointwise
  convergences, we get \(\lim_n u_n(x) a = \lim_n a u_n =a\).
\end{proof}

\subsection{Fell bundles}

Fell bundles are natural genaralisation of \(\Cst\)\nb-dynamical
system and is an established objects in Operator Algebras. We shall
\emph{briefly} describe them
from~\cite{Muhly-Williams2008FellBundle-ME,
  Kumjian1998Fell-bundles-over-gpd}. Many of our notation and
elementary discussion of Fell bundles is adopted
from~\cite{Muhly-Williams2008FellBundle-ME}.  Suppose
\(p\colon \A \to G\) is an upper semicontinuous bundle of Banach
spaces over a locally compact, Hausdorff groupoid~\(G\). Let
\( \A^{(2)} \defeq\{ (a,b) \in \A\times \A : (p(a),p(b))\in
G^{(2)}\}.\)
\begin{definition}[Fell bundle~\cite{Muhly-Williams2008FellBundle-ME}]
  \label{def:fell-bundle}
  A Fell bundle over a locally compact, Hausdorff groupoid \(G\) is an
  upper semicontinuous bundle of Banach spaces \(p\colon \A\to G\)
  equipped with a continuous `multiplication' map
  \(m\colon \A^{(2)}\to \A\) and an involution map
  \(\A\to \A, a\mapsto a^*\) for \(a\in \A\) which satisfy the
  following axioms:
  \begin{enumerate}
  \item\label{item:Fell-1} \(p(m(a,b)) =p(a)p(b)\) for all
    \((a,b) \in \A^{(2)}\);
  \item\label{item:Fell-2} the induced map from
    \(\A_{\gamma_1}\times \A_{\gamma_2} \to \A_{\gamma_1 \gamma_2}\),
    \((a,b) \mapsto m(a,b)\) is bilinear for all
    \((\gamma_1, \gamma_2)\in G^{(2)}, a\in \A_{\gamma_1}\) and
    \(b\in \A_{\gamma_2}\);
  \item\label{item:Fell-3} \(m(m(a,b),c) =m(a, m(b,c))\) for
    \((a,b),(b,c)\in \A^{(2)}\);
  \item\label{item:Fell-4} \(\norm{m(a,b)} \leq \norm{a}\norm{b}\) for
    all \((a,b) \in \A^{(2)}\);
  \item\label{item:Fell-5} \(p(a^*)=p(a)^{-1}\) for all \(a\in \A\);
  \item\label{item:Fell-6} the induced map
    \(\A_\gamma \to \A_{\gamma^{-1}}\), \(a\mapsto a^*\) is conjugate
    linear for all \(\gamma \in G\);
  \item\label{item:Fell-7} \((a^*)^*=a\) for all \(a\in \A\);
  \item\label{item:Fell-8} \(m(a,b)^*=m(b^*,a^*)\) for all
    \((a,b)\in \A^{(2)}\);
  \item\label{item:Fell-9}
    \(\norm{m(a^*,a)}=\norm{m(a,a^*)}=\norm{a}^2\) for all
    \(a\in \A\);
  \item\label{item:Fell-10} \(m(a^*,a) \geq 0\) for all \(a\in \A\).
  \end{enumerate}
\end{definition}

\noindent Given \((a,b) \in \A^{(2)}\), we use the shorthand \(ab\)
instead of \(m(a, b)\).  A Fell bundle~\(\A\) is called saturated if
the set
\( \A_{\gamma_1} \cdot \A_{\gamma_2} = \textup{span} \{ab: a\in
\A_{\gamma_1}, b\in \A_{\gamma_2}\}\) is dense in
\(\A_{\gamma_1 \gamma_2}\) for all
\((\gamma_1,\gamma_2) \in G^{(2)}\).  For \(x\in G^{(0)}\),
Conditions~\eqref{item:Fell-2}--\eqref{item:Fell-9} in
Definition~\ref{def:fell-bundle} make the fibre \(\A_x\) a
\(\Cst\)\nb-algebra with the obvious norm, multiplication and
involution.  Further, if \(\gamma \in G\) and \(a,b\in \A_\gamma\),
then \(\A_\gamma\) is a right Hilbert \(\A_{s(\gamma)}\)\nb-module
when equipped with the inner product
\(\inpro{a}{b}_{\A_{s(\gamma)}}\defeq a^*b\); and~\(\A_\gamma\) is a
left Hilbert \(\A_{r(\gamma)}\)\nb-module if the inner product is
defined by \({}_{\A_{r(\gamma)}}\Linpro{a}{b} \defeq ab^*\). Those
claims follow from Conditions~\eqref{item:Fell-2},
\eqref{item:Fell-8}, \eqref{item:Fell-9} and~\eqref{item:Fell-10} in
Definition~\ref{def:fell-bundle}. Thus, the Fell bundle
\(p\colon \A \to G\) is saturated if and only if~\(\A_\gamma\) is an
\(\A_{r(\gamma)}\)-\(\A_{s(\gamma)}\)\nb-imprimitivity bimodule for
each~\(\gamma \in G\). In this article, by a Fell bundle we shall mean
a \emph{saturated} Fell bundles.  For a subset ~\(X\subseteq G\), by
\(\A|_X\) we denote the restriction of~\(\A\) to~\(X\) but we still
denote the bundle map by~\(p\colon \A|_X\to X\).

\subsection{Representation of Fell bundles}

Let \(G\) be a locally compact, Hausdorff, second countable, \etale\
groupoid, and \(p\colon \A\to G\) a saturated Fell bundle over
it. Denote the complex vector space of compactly supported continuous
sections of the Fell bundle by \(\Contc(G;\A)\).  This vector space
becomes a {\Star}-algebra when equipped with the following operations:
\[
  f*g(\gamma)= \sum_{\alpha\beta=\gamma}f(\alpha)g(\beta) \quad
  \textup{and} \quad f^*(\gamma)=f(\gamma^{-1})^*
\]
where \(f,g\in \Contc(G;\A)\).  Define the \emph{\(I\)-norm} on
\(\Contc(G;\A)\) by
\[
  \norm{f}_I= \max \biggl\{\sup_{x\in G^{(0)}}\sum_{\gamma\in G^x}
  \norm{f(\gamma)},\sup_{x\in G^{(0)}}\sum_{\gamma\in G_x}
  \norm{f(\gamma)}\biggr\}.
\]

A representation of the {\Star}-algebra \(\Contc(G; \A)\) on a Hilbert
space \(\Hils\) is a non-degenerate \(^*\)-homomorphism
\(L\colon \Contc(G;\A)\to \Bound(\Hils)\). The representation \(L\) is
called \(I\)-norm decreasing if \(\norm{L(f)}\leq \norm{f}_I\) for all
\(f\in \Contc(G;\A)\).  The universal \(\Cst\)-norm on
\(\Contc(G;\A)\) is defined by
\[
  \norm{f}= \sup\{\norm{L(f)} : L \text{ is an } I\text{-norm
    decrising representation}\};
\]
the completion of \(\Contc(G;\A)\) with respect to the universal norm,
denoted by \(\Cst(G;\A)\), is called the \emph{full
  \(\Cst\)\nb-algebra} of the Fell bundle \(p\colon \A\to G\).

Muhly and Williams prove the disintegration theorem for Fell bundles
over Hausdorff, locally compact, second countable groupoids equipped
with Haar systems in~\cite{Muhly-Williams2008FellBundle-ME}. This
non-trivial theorem is in the same zeal as Renault's disintegration
theorem~\cite[Theorem~II.1.21]{Renault1980Gpd-Cst-Alg} for groupoids
with Haar system.

Muhly and Williams define a unitary representation of a Fell
bundle~\(p\colon \A\to G\), and integrate it to a representation of
the convolution section algebra \(\Contc(G;\A)\). Conversely, a
representation of \(\Contc(G;\A)\) on a Hilbert space disintegrates to
a unitary representation of \(p\colon \A\to G\). Moreover, there is a
certain one-to-one correspondence between the representation of
\(\Cc(G;\A)\) and the unitary representation of \(p\colon \A \to
G\). We need the disintegration theorem~\cite[Theorem
4.13]{Muhly-Williams2008FellBundle-ME} for proving
Proposition~\ref{prop:integration-of-states}.  \medskip

Let \(p\colon \A\to G\) be a Fell bundle over a locally compact,
Hausdorff, second countable groupoid \(G\). Let \(G^{(0)}*\Hils\) be a
Borel Hilbert bundle over \(G^{(0)}\). Let
\(\textup{End}(G^{(0)}*\Hils)\) denote the set
\[
  \textup{End}(G^{(0)}*\Hils) \defeq \big\{(x,T,y): x,y \in G^{(0)},
  T\in \Bound(\Hils_y,\Hils_x)\big\}.
\]
% Consider \(\textup{End}(G^{(0)}*\Hils)\) with the smallest Borel
% structure on it such that the map
% \(\textup{End}(G^{(0)}*\Hils) \to \C\) given by
% \((x,T,y) \mapsto \binpro{g(y)}{Tf(x)}\) is measurable for any two
% Borel section \(f,g \in L^2(\base*\Hils, \mu)\).

\begin{definition}[{\cite[Definition~4.5]{Muhly-Williams2008FellBundle-ME}}]\label{def:star-funtor}
  A map \(\hat{\pi}\colon \A \to \textup{End}(G^{(0)}*\Hils)\) is
  called \(*\)-functor if for \(a\in \A\),
  \(\hat{\pi}(a)= \big(r(p(a)), \pi(a), s(p(a))\big)\) for a bounded
  linear operator
  \(\pi(a) \colon \Hils_{s(p(a))} \to \Hils_{r(p(a))}\) such that
  \begin{enumerate}
  \item \(\pi(\lambda a+b)=\lambda \pi(a)+\pi(b)\) if
    \(a,b \in \A_\gamma\) for \(\gamma \in G\);
  \item \(\pi(ab)=\pi(a)\pi(b)\) if \((a,b)\in \A^{(2)}\);
  \item \(\pi(a^*)=\pi(a)^*\) for \(a\in \A\).
  \end{enumerate}
\end{definition}
\noindent For \(x\in \base[G]\),
\(\pi|_{\A_x}\colon \A_x \to \Bound(\Hils_x)\) is a nondegenarate
\(^*\)\nb-homomorphism.  A strict representation of a Fell bundle
\(\A\) is a triple \((\mu, \base*\Hils, \hat{\pi})\) consisting of a
quasi-invariant probability measure \(\mu\) on \(G^{(0)}\), a Borel
Hilbert bundle \(G^{(0)}*\Hils\) and a \(*\)-funtor \(\hat{\pi}\). We
write \(L^2(G^{(0)}*\Hils,\mu)\) is the completion of the set of all
Borel sections \(f\colon G^{(0)}\to G^{(0)}*\mathcal{H}\) with
\(\int_{G^{(0)}}\inpro{f(x)}{f(x)}_{\Hils_x}\mathrm{d}\mu (x)<\infty\)
with respect to the inner product
\(\inpro{f}{g}_{L^2(G^{(0)}*\Hils,\mu)} = \int_{G^{(0)}}
\inpro{f(x)}{g(x)}_{\Hils_x}\mathrm{d}\mu(x).\) Let \(\Delta_\mu\) be
the Radon--Nikodym derivative for the quasi-invariant measure
\(\mu\). Given a strict representation
\((\mu, G^{(0)}*\Hils,
\hat{\pi})\)~\cite[Proposition~4.10]{Muhly-Williams2008FellBundle-ME}
gives an \(I\)-norm bounded representation of \(\Cst(G;\A)\) on the
Hilbert space \(L^2(G^{(0)}*\Hils,\mu)\) such that
\begin{equation}\label{equ-int-dis-rep}
  \inpro{\eta}{L(f)\xi} = \int_{G^{(0)}} \biggr(\sum_{\gamma\in G^x}
  \inpro{\eta(r(\gamma))}{\pi(f(\gamma))\xi(s(\gamma))}_{\Hils_{r(\gamma)}}
  \Delta_\mu(\gamma)^{-\frac{1}{2}} \biggr) \mathrm{d} \mu(x),
\end{equation}
where \(f\in \Cc(G;\A)\) and \(\xi,\eta \in L^2(G^{(0)}*\Hils,\mu)\).
Muhly and Williams called this representation \(L\) is the integrated
form of \(\pi\). The disintegration theorem~\cite[Theorem
4.13]{Muhly-Williams2008FellBundle-ME} for Fell bundle is a strong
converse of ~\cite[Proposition~4.10]{Muhly-Williams2008FellBundle-ME}.

	\begin{lemma}\label{lem:convo-func:supp:bise}
          Let \(U\) and \(V\) be bisections of \(G\).
          \begin{enumerate}
          \item
            \label{lem:convulu-supp} If \(f \in \Cc(U;\A|_{U})\) and
            \(g\in \Cc(V;\A|_{V})\), then
            \(f*g \in \Cc(UV;\A|_{UV})\).
          \item
            \label{lem:involu-supp} If \(f \in \Cc(U;\A|_{U})\), then
            \(f^* \in \Cc(U^{-1}; \A|_{U^{-1}})\).
          \item \label{cond-3-bise-lemma} If \(f\in \Cc(U;\A|_{U})\),
            then \(f*f^*, f^**f \in \Cc(\base; \A|_{\base})\). In
            particular \(f*f^* \in \Cc(r(U); \A|_{r(U)})\) and
            \(f^**f \in \Cc(s(U); \A|_{s(U)})\).
          \item \label{lem:uni-norm-bisec} If \(f\in \Cc(U;\A|_{U})\),
            then
            \(\norm{f*f^*}_\infty = \norm{f^**f}_\infty =
            \norm{f}_\infty^2\).
          \end{enumerate}
        \end{lemma}
        \begin{proof}
          The proof is similar to Proposition~3.11
          of~\cite{Exel2008Inverse-semigroup-combinotorial-Cst-alg}.
        \end{proof}

        \begin{remark}
          \label{rem:app-id-in-com-supp}
          Let \(p\colon \A \to G\) be a Fell bundle. Then
          \(\Contz(\base; \A|_{\base})\) is the section algebra of the
          upper semicontinuous bundle of \(\Cst\)\nb-algebra
          \(p\colon \A|_{\base} \to \base\).
          Then~\(\Contz(\base; \A|_{\base})\) has an approximate
          identity consisting of continuous sections having compact
          support. To see this, let~\((a_n)_{n\in \N}\) be an
          approximate identity of~\(\Contz(\base;
          \A|_{\base})\). Since~\(\base\) locally compact, Hausdorff,
          second countable \(\Contz(\base)\) has an approximate
          identity~\((s_n)_{n\in \N}\) where
          each~\(s_n\in \Contc(\base)\). Then~\((s_na_n)_{n \in \N}\)
          is an approximate identity of~\(\Contz(\base; \A|_{\base})\)
          consisting of continuous compactly supported sections.
        \end{remark}

\begin{proposition}\label{prop:exist-app:unit}
  Let \(p\colon \A \to G\) be a Fell bundle over a locally compact,
  Hausdorff, second countable, \etale\ groupoid.  Let
  \((u_n)_{n\in \N}\) be an approximate identity of
  \(\Contz(\base; \A|_{\base})\) in \(\Cc(\base;\A|_{\base})\) (by
  Remark~\ref{rem:app-id-in-com-supp} such approximate identity always
  exist) with respect to \(\norm{\cdot }_{\infty}\), then
  \((u_n)_{n\in \N}\) is an approximate identity for \(\Cst(G;\A)\).
\end{proposition}
\begin{proof}
  Since for an \etale\ groupoid \(\base\) is a clopen subset of \(G\),
  we have \(u_n\in \Cc(G;\A)\) for all \(n\in \N\). We show that, for
  any \(f\in \Cc(G;\A)\) we have to show
  \begin{equation}\label{equ:req-app:uint}
    \lim_{n} \norm{f*u_n-f} = 0.
  \end{equation}
  Since the collections of bisections form a base for \(G\), it is
  suffices to prove Equation~\eqref{equ:req-app:uint} for
  \(f\in \Cc(U;\A|_{U})\) where \(U\) is an open bisection. We have
  \begin{multline*}
    \norm{u_n*f-f}^2 = \norm{(u_n*f-f)*(u_n*f-f)^*}
    = \norm{(u_n*f-f)*(f^**u_n-f^*)}\\
    = \norm{ u_n*f*f^**u_n- u_n*f*f^*-f*f^**u_n+f*f^*} \\
    = \norm{ u_n*(f*f^*)*u_n- u_n*(f*f^*)-(f*f^*)*u_n+f*f^*}.
  \end{multline*}
  The first equality follows from \(\Cst\)\nb-identity and the second
  follows because \(u_n\) is self-adjoint. Since \(u_n, f*f^*, f^**f\)
  are all supported on \(\base\), then using
  Lemma~\ref{lem:convo-func:supp:bise} the above term can be rewitten
  as
  \begin{multline*}
    \norm{ u_n\cdot (f*f^*) \cdot u_n- u_n\cdot (f*f^*)-(f*f^*)\cdot u_n+f*f^*}_{\infty} \\
    \leq \norm{u_n\cdot (f*f^*) \cdot u_n- u_n\cdot (f*f^*)}_{\infty} + \norm{(f*f^*)\cdot u_n-f*f^*}_{\infty} \\
    \leq \norm{u_n}_{\infty}\norm{ (f*f^*) \cdot u_n- (f*f^*)}_{\infty} +\norm{(f*f^*)\cdot u_n-f*f^*}_{\infty}\\
    \leq 2 \norm{(f*f^*)\cdot u_n-f*f^*}_{\infty} \to 0 \quad \text{
      as } n \to \infty.
  \end{multline*}
\end{proof}

\section{The induction correspondence}
\label{sec:an-example-corr}

Let \(G\) be a locally compact, Hausdorff, \etale\ groupoid and
\(p\colon \A \to G\) a Fell bundle over it. Fix \(x\in \base[G]\).
Let \(\A|_{G_x}\) be the restriction of \(\A\) to fibre
\(G_x\subseteq G\) and \(\A(x)\) the restriction of \(\A\) to
\(G^x_x\). The restrictions of~\(\A\) to \(G_x^x\) is a Fell bundle
over isotropy group. And the restrictions of~\(\A\) to \(G_x\) is an
upper semicontinuous bundle of Banach spaces such that
\(\Contc(G_x^x;\A(x)) \subseteq \Contc(G_x;\A|_{G_x})\) in an obvious
sense.

For \(\gamma\in G\), \(\delta_\gamma\) denotes the characteristic
function of the singletone \(\{\gamma\}\). In general, a section of
the restricted group Fell bundle \(\A(x)\to G_x^x\) is written as a
sum \(\sum_{\gamma\in G_x^x} a_\gamma \delta_\gamma\) where
\(a_\gamma\in \A_\gamma\). If necessary, we write
\(a_\gamma \cdot \delta_\gamma\) instead of \(a_\gamma\delta_\gamma\).
\smallskip

For \(f\in \Contc(G_x^x;\A(x)), g\in \Contc(G; \A)\) and
\(\xi,\zeta\in \Contc(G_x; \A|_{G_x})\) define
\begin{align}
  g*'\xi(z) &= \sum_{\tau\in G^{r(z)}} g(\tau) \xi(\tau\inverse z), \label{eq:ind-corr-1}\\
  \xi*''f(z) &= \sum_{\upsilon\in G^x_x} \xi(z\upsilon) f(\upsilon\inverse), \label{eq:ind-corr-2}\\
  \inpro{\xi}{\zeta} (\gamma) &= \sum_{t\in G_{x}} \xi(t)^* \zeta(t\gamma) \label{eq:ind-corr-3}
\end{align}
where \(z\in G_x\) and \(\gamma\in G_x^x\). Our immediate goal is to
show that Equations~\eqref{eq:ind-corr-1} and~\eqref{eq:ind-corr-2}
make~\(\Contc(G_x;\A|_{G_x})\) a bimodule over the
pre-\(\Cst\)\nb-algebras \(\Contc(G;\A)\) and \(\Contc(G_x^x;\A(x))\),
respectively; and Equation~\eqref{eq:ind-corr-3} defines a
\(\Contc(G_x^x;\A(x))\)\nb-valued inner product on this
bimodule. Before we start proving this, in the next lemma we note an
observation that shall prove useful in future computation.

For \(a\in \Contc(G_x^x;\A(x))\), let
\(\tilde{a}\in \Contc(G_x;\A|_{G_x})\) denote the section obtained by
extending~\(a\) by zero outside \(G_x^x\). Note that
\(\widetilde{a^*}=(\tilde{a})^*\).

\begin{lemma}\label{lem:inpro-conv-equivalence}
  For \(a\in \Contc(G_x^x;\A(x))\), and
  \(b\in \Contc(G_x;\A|_{G_x})\),
  \[
    \inpro{\tilde{a}}{b} = \widetilde{(a^*)} *'' (b|_{G_x^x}) =
    a^**(b|_{G_x^x})
  \]
  where the last term is the convolution in \(\Contc(G_x^x;\A(x))\).
\end{lemma}
\begin{proof}
  For given \(a\) and \(b\), and \(\gamma\in G_x^x\),
  \[
    \inpro{\tilde{a}}{b}(\gamma) = \sum_{t\in
      G_x}\tilde{a}(t)^*b(t\gamma) = \sum_{t\in
      G_x}\widetilde{a^*}(t\inverse)b(t\gamma).
  \]
  Now we note that \(\widetilde{a^*}(t\inverse)=0\) for
  \(t\inverse\in G_x\7 G_x^x\), and \(\gamma\in G_x^x\), therefore,
  the last sum equals
  \begin{equation*}
    \sum_{t\in G^x_x}\widetilde{a^*}(t^{-1})b(t\gamma)  = \sum_{t\in G^x_x}\widetilde{a^*}(t^{-1})b|_{G^x_x}(t\gamma) = \widetilde{a^*}*'' (b|_{G_x^x})(\gamma).
  \end{equation*}
  And the last equality
  \( \widetilde{(a^*)} *'' (b|_{G_x^x}) = a^**(b|_{G_x^x})\) follows
  from definitions directly.
\end{proof}
If \(a\in \Cc(G^x_x;\A(x))\), then we abuse the notation by removing
the tilde and consider \(a\in \Cc(G_x;\A|_{G_x})\).  Now we start
proving the claim made before Lemma~\ref{lem:inpro-conv-equivalence}.

\begin{lemma}\label{lem:ind-corr-ops}
  In Equations~\eqref{eq:ind-corr-1}--\eqref{eq:ind-corr-3},
  \(g*'\xi, \xi *'' f\in \Contc(G_x; \A|_{G_x})\) and
  \(\inpro{\xi}{\zeta}\in \Contc(G_x^x;\A(x))\).
\end{lemma}
\begin{proof}
  
  Since \(G_x, G^x_x\) are discrete spaces, we only have to show that
  the sections \(g*'\xi, \xi *'' f\) and \(\inpro{\xi}{\zeta}\) are
  zero everywhere but on finite sets. The sum on the right of
  Equation~\eqref{eq:ind-corr-1} is zero when
  \(\tau \notin G^{r(z)}\cap \supp(g) \) which is a finite set.  Thus
  \(g*'\xi(z)\) is a well-defined element of \(\A_z\). Next for
  \(z\notin \supp(g)\supp(\xi)\), we have \(g*'\xi (z) =0\). Which
  implies that \(\supp(g*'\xi)\subseteq \supp(g)\supp(\xi)\) is a
  finite set. It follows that \(g*'\xi \in \Cc(G_x;\A|_{G_x})\). A
  similar argument shows that
  \( \xi *'' f\in \Contc(G_x; \A|_{G_x})\). The sum
  in~Equation~\eqref{eq:ind-corr-3} is well-defined, because it is a
  finite sum; an argument similar to that of
  Equation~\eqref{eq:ind-corr-1} proves this.  From the
  Equation~\ref{eq:ind-corr-3}, it is clear that
  \(\inpro{\xi}{\zeta}\) is a section. If \(\zeta(t\gamma) = 0\) for
  some \(t\in G_x\), then \(\inpro{\xi}{\zeta}(\gamma) = 0\). Since
  \(\zeta\) is finitely supported, \(\inpro{\xi}{\zeta}\) is also
  finitely supported.
  
\end{proof}

\begin{lemma}\label{lem:ind-corr-bimod}
  Suppose sections \(f,l\in \Cc(G^x_x;\A(x))\), sections
  \(\xi,\zeta, \zeta' \in \Cc(G_x;\A|_{G_x})\) and sections
  \(g,k\in \Cc(G;\A)\). Then
  \begin{align}
    (\xi+\zeta)*''f &=\xi*''f+\zeta*''f \label{eq:ind-corr-bimod-1}\\
    \xi*'' (f+l) &= \xi *'' f+ \xi *'' l \label{eq:ind-corr-bimod-2}\\
    \xi*''(f*l) &=(\xi*''f)*''l \label{eq:ind-corr-bimod-3}\\
    g*' (\xi+\zeta) &= g*' \xi+ g*' \zeta \label{eq:ind-corr-bimod-4}\\
    (g+k)*'\xi &= g*'\xi + k*'\xi \label{eq:ind-corr-bimod-5}\\
    (g*k)*'\xi &= g*'(k*'\xi) \label{eq:ind-corr-bimod-6}\\
    (g*'\xi)*''f &= g*' (\xi*'' f) \label{eq:ind-corr-bimod-7}\\
    \inpro{\xi}{\zeta+\zeta'} &=\inpro{\xi}{\zeta}+\inpro{\xi}{\zeta'}
                                \label{eq:ind-corr-bimod-8}\\
    \inpro{\xi}{\zeta*''f}
                    &=\inpro{\xi}{\zeta}*f \label{eq:ind-corr-bimod-9}\\
    \inpro{\xi}{\zeta}^* &=\inpro{\zeta}{\xi} \label{eq:ind-corr-bimod-10}\\
    \inpro{g*'\xi}{\zeta} &=\inpro{\xi}{g^**'\zeta} \label{eq:ind-corr-bimod-11}
  \end{align}
\end{lemma}
Equations~\eqref{eq:ind-corr-bimod-1}--\eqref{eq:ind-corr-bimod-7}
suggest that the vector space \(\Contc(G_x;\A|_{G_x})\) is a left and
right bimodule over the pre-\(\Cst\)\nb-algebras \(\Contc(G;\A)\) and
\(\Contc(G_x^x;\A(x))\),
respectively. Equations~\eqref{eq:ind-corr-bimod-8}--\eqref{eq:ind-corr-bimod-10}
show that \(\inpro{\cdot}{\cdot}\) is a
\(\Contc(G_x^x;\A(x))\)\nb-valued conjugate linear form on
\(\Contc(G_x;\A|_{G_x})\). And Equation~\eqref{eq:ind-corr-bimod-11}
show that of \(\Contc(G;\A)\) on this bimodule by adjointable
operators.
\begin{proof}[Proof of Lemma~\ref{lem:ind-corr-bimod}]
  The proofs of most of the equations in this lemma are similar to
  Lemma~2.7 in~\cite{Holkar2017Construction-of-Corr}.
  Equation~\eqref{eq:ind-corr-bimod-1},
  Equation~\eqref{eq:ind-corr-bimod-2},
  Equation~\eqref{eq:ind-corr-bimod-4} and
  Equation~\eqref{eq:ind-corr-bimod-5} follows directly from
  defintions.
  % For
  % Equation~\eqref{eq:ind-corr-bimod-1}:
  % \begin{multline*}
  %   (\xi+\zeta)*''f(z) = \sum_{v \in G^x_x}(\xi+\zeta)(zv)f(v^{-1})
  %   =
  %   \sum_{v \in G^x_x}\xi(zv)f(v^{-1})+\sum_{v\in
  %   G^x_x}\zeta(zv)f(v^{-1})\\ = (\xi*''f+\zeta*''f)(z).
  % \end{multline*}
  % Equation~\eqref{eq:ind-corr-bimod-2},
  % Equation~\eqref{eq:ind-corr-bimod-4} and
  % Equation~\eqref{eq:ind-corr-bimod-5} follows from a similar
  % computation.
 
  \noindent The proof of Equation~\eqref{eq:ind-corr-bimod-3},
  Equation~\eqref{eq:ind-corr-bimod-6} and
  Equation~\eqref{eq:ind-corr-bimod-7} are similar. We shall only
  prove Equation~\eqref{eq:ind-corr-bimod-7}~only.

  % Equation~\eqref{eq:ind-corr-bimod-3}:
  % \begin{equation*}
  %   \xi*''(f*l)(z) = \sum_{v\in G^x_x}\xi(zv)(f*l)(v^{-1}) = \sum_{v
  %   \in G^x_x }\xi(zv)\sum_{\gamma \in
  %   G^x_x}f(\gamma)l(\gamma^{-1}v^{-1}).
  % \end{equation*}
  % Change the variable \(\gamma \mapsto v^{-1}\gamma\); then the last
  % term equals
  % \[
  %   \sum_{v\in G^x_x}\sum_{\gamma \in G^x_x} \xi (zv)
  %   f(v^{-1}\gamma) l(\gamma^{-1}).
  % \]
  % Changing the order of the finite sum above (finiteness allows to
  % use Fubini's theorem) gives us
  % \[
  %   \sum_{\gamma\in G^x_x}\sum_{v \in G^x_x} \xi (zv)
  %   f(v^{-1}\gamma) l(\gamma^{-1}).
  % \]
  
  % \noindent Applying the change of variable \(v\mapsto \gamma v\)
  % the above term equals
  % \[
  %   \sum_{\gamma\in G^x_x}\sum_{v \in G^x_x}\xi(z\gamma
  %   v)f(v^{-1})l(\gamma^{-1}) = \sum_{\gamma\in
  %   G^x_x}\xi*''f(z\gamma)l(\gamma^{-1}) = (\xi*''f)*''l(z).
  % \]
 
  % \noindent Equation~\eqref{eq:ind-corr-bimod-6}:
  % \[
  %   (g*k)*'\xi(z) = \sum_{\tau \in G^{r(z)}}g*k(\tau)\xi(\tau^{-1}z)
  %   = \sum_{\tau \in G^{r(z)}} \sum_{\gamma \in
  %   G^{r(\tau)}}g(\gamma)k(\gamma^{-1}\tau)\xi(\tau^{-1}z).
  % \]
  % Change the order of the sum and the variable
  % \(\tau \to \gamma \tau\) the last term equals
  % \[
  %   \sum_{\gamma\in G^{r(z)}}\sum_{\tau \in G^{s(\gamma)}} g(\gamma)
  %   k(\tau)\xi(\tau^{-1}\gamma^{-1}z) = \sum_{\gamma \in
  %   G^{r(z)}}g(\gamma) (k*'\xi)(\gamma^{-1}z) = g*'(k*'\xi)(z).
  % \]

  \noindent Equation~\eqref{eq:ind-corr-bimod-7}: we have
  \[
    (g*'\xi)*''f(z) = \sum_{t\in G^x_x}g*'\xi(zt)f(t^{-1}) = \sum_{t
      \in G^x_x} \sum_{\tau \in G^{r(x)}} g(\tau)
    \xi(\tau^{-1}zt)f(t^{-1}).
  \]
  Change the order of the sums; then last term equals
  \[
    \sum_{\tau \in G^{r(z)}} g(\tau) (\xi*''f)(\tau^{-1}z) =
    g*'(\xi*''f)(z).
  \]
  \noindent Equation~\eqref{eq:ind-corr-bimod-8} is direct
  computation.
 
  \noindent Equation~\eqref{eq:ind-corr-bimod-9}:
  \[
    \inpro{\xi}{\zeta*''f}(\gamma) = \sum_{t\in G_x}
    \xi(t)^*(\zeta*''f)(t\gamma) = \sum_{t\in G_x} \xi(t)^* \sum_{v\in
      G^x_x} \zeta(t\gamma v) f(v^{-1}).
  \]
  By changing the order of the sums in the last term, we get
  \[
    \sum_{v\in G^x_x} \inpro{\xi}{\zeta}(\gamma v)f(v^{-1}) =
    \inpro{\xi}{\zeta}*f(\gamma).
  \]
 
  \noindent Equation~\eqref{eq:ind-corr-bimod-10}:
  \[
    \inpro{\xi}{\zeta}^*(z) = \big(\inpro{\xi}{\zeta}(z^{-1})\big)^* =
    \biggr(\sum_{t\in G_x} \xi(t)^*\zeta(tz^{-1}) \biggr)^* =
    \sum_{t\in G_x} \zeta(tz^{-1})^*\xi(t).
  \]
  Since \(G_x\) is a discrete space and \(t\mapsto tz\) is a bijection
  on \(G_x\), using the variable \(t\mapsto tz\), we get
  \( \sum_{t\in G_x}\zeta(t)^*\xi(tz) = \inpro{\zeta}{\xi}(z).\)
 
  \noindent Equation \eqref{eq:ind-corr-bimod-11}: let
  \(g\in \Cc(G;\A)\) and \(\xi,\zeta \in \Cc(G_x;\A|_{G_x})\), then
  for \(\gamma \in G^x_x\), we have
  \begin{multline*}
    \binpro{g*'\xi}{\zeta}(\gamma) = \sum_{t\in
      G_x}\big(g*'\xi(t)\big)^*\zeta(t\gamma) = \sum_{t\in G_x}
    \sum_{\tau \in
      G^{r(t)}} \big(g(\tau)\xi(\tau ^{-1}t)\big)^*\zeta(t\gamma)\\
    = \sum_{t\in G_x} \sum_{\tau \in G^{r(t)}} \xi(\tau ^{-1}t)^*
    g(\tau)^*\zeta(t\gamma).
  \end{multline*}
  By applying the change of variable \(t \mapsto \tau t\), the last
  term above becomes
  \[
    \sum_{t\in G_x}\sum_{\tau \in G_{r(t)}}
    \xi(t)^*g(\tau)^*\zeta(\tau t\gamma).
  \]
  Since \(G^{r(t)}, G_{r(t)}\) both are discrete and
  \(\tau \mapsto \tau^{-1}\) is a bijection from
  \(G^{r(t)} \to G_{r(t)}\), using the change of variable
  \(\tau \mapsto \tau^{-1}\), we get
  \[
    \sum_{t\in G_x}\sum_{\tau \in G^{r(t)}}
    \xi(t)^*g(\tau^{-1})^*\zeta(\tau^{-1}t\gamma) = \sum_{t\in
      G_x}\xi(t)^*g^**'\zeta(t\gamma)=\binpro{\xi}{g^**'\zeta}(\gamma).
  \]
\end{proof}

Before we proceed, fix the notation: let \((u^x_n)_{n\in \N}\) be an
approximate identity for the fibre \(\Cst\)-algebra \(\A_x\); write
\(d^n_x\) for the section \(u^x_n\cdot \delta_x\) of the bundles
\(\A|_{G_x}\to G_x\) or \(\A(x)\to G_x^x\).

\begin{lemma}
  \label{lemma:apprx-id-of-bimod}
  Let \(x\in G^{(0)}\) be a unit. Then the sequence of sections
  \((d_x^n)_{n\in \N}\) defined above is an approximate identity for
  the \(\Cst\)-algebra \(\Cst(G^x_x;\A(x))\).
\end{lemma}

\begin{proof} Since \(u^x_n\) is positive, there exists
  \(v^x_n \in \A_x\) such that \((v^x_n)^*v^x_n =u^x_n\). It follows
  that
  \( d_x^n=u^x_n\cdot\delta_x =(v^x_n)^*v^x_n \cdot \delta_x \delta_x
  = (v^x_n\cdot \delta_x)^* *(v^x_n\cdot \delta_x).  \) Therefore,
  \(d_x^n\) is positive in \(\Cst(G^x_x;\A(x))\). Clearly,
  \((d_n)_{n\in \N}\) is an increasing sequence in
  \(\Cst(G^x_x;\A(x))\). It is sufficient to show that for any section
  \(f\in \Contc(G^x_x; \A(x))\), \(d^n_x*f\to f\) in the inductive
  limit topology of \(\Contc(G^x_x; \A(x))\).  Firstly, note that
  \( \supp (d^n_x*f)\subseteq \supp(d^n_x)\supp(f) \subseteq \supp(f)
  \) for all \(n\).  In fact \(d^n_x*f\) converges uniformly to \(f\)
  on the finite set \(\supp(f)\).  To see this, compute
  \[
    d^n_x*f(\gamma)=\sum_{\zeta\in G^x_x}d^n_x(\gamma
    \zeta^{-1})f(\zeta) =d^n_x(x)f(\gamma) =u^x_if(\gamma).
  \]
  Since \(\A_\gamma\) is full left (and right) Hilbert
  \(\A_x\)\nb-module, by
  Lemma~\ref{lem:appx-unit-for-full-Hilbert-mod},
  \(\norm{u_n^x\cdot f(\gamma) - f(\gamma)}\to 0\) in \(\A_\gamma\)
  for every \(\gamma\in \supp(f)\). Since \(\supp(f)\subseteq G_x^x\)
  is a finite set,
  \[
    \norm{d^n_x*f(y)-f(y)}_\infty = \max_{y\in
      \supp(f)}\big\{\norm{u^x_n\cdot f(y)-f(y)} \big\} \to 0 \quad
    \textup{as \(n\to \infty\).}
  \]
\end{proof}

\begin{lemma}\label{lem:for-cont-left-action}
  \begin{enumerate}[leftmargin=*]
  \item If \(\xi_n \to 0\) in \(\Cc(G_x;\A|_{G_x})\) in the inductive
    limit topology, then \(\inpro{\xi_n}{\xi_n} \to 0\) in
    \(\Cc(G^x_x;\A(x))\) in the inductive limit topology, hence in
    \(\Cst(G^x_x;\A(x))\).
		
  \item Let \(\xi\in \Cc(G_x;\A|_{G_x})\). If \(f_n \to f\) in
    \(\Cc(G;\A)\) in the inductive limit topology, then
    \(f_n*'\xi \to f*'\xi\) in \(\Cc(G_x;\A|_{G_x})\) in the inductive
    limit topology.
		
  \end{enumerate}
\end{lemma}
\begin{proof}
  \noindent (1): Let \(\xi_n \to 0\) uniformly and \(\supp(\xi_n)\) is
  contained in a fixed finite set \(K\). Then the support of
  \(\inpro{\xi_n}{\xi_n}\) is also contained in the fixed finite set
  \(K^{-1}\cdot K\). And
  \begin{multline*}
    \norm{\inpro{\xi_n}{\xi_n}}_\infty = \sup_{\gamma\in G^x_x} \norm{\inpro{\xi_n}{\xi_n}(\gamma)}  = \sup_{\gamma\in G^x_x} \norm{\sum_{t\in G_x} \xi_n(t)^*\xi_n (t\gamma)}\\
    \leq \sup_{\gamma\in G^x_x} \sum_{t\in G_x} \norm{\xi_n(t)}
    \norm{\xi_n(t\gamma)} \leq \#(K) \norm{\xi_n}^2_{\infty} \to 0
    \quad \textup{as } n\to \infty.
  \end{multline*}
  Therefore, \(\inpro{\xi_n}{\xi_n} \to 0\) in the inductive limit
  topology.
	
  \noindent (2): Let \(f_n \to f\) uniformly and
  \(\supp(f_n), \supp(f)\) contained in a fixed finite set \(L\). Then
  \(\supp(f_{n}*' \xi) \) and \(\supp(f*'\xi)\) contained in a fixed
  finite set of \(G_x\). Now,
  \begin{multline*}
    \norm{f_n*'\xi -f*'\xi}_\infty = \norm{(f_n-f)*'\xi}_\infty = \sup_{\gamma\in G_x}\norm{(f_n-f)*'\xi(\gamma)} \\
    = \sup_{\gamma\in G_x}\norm{\sum_{\tau \in G^{r(\gamma)}}
      (f_n-f)(\tau) \xi(\tau^{-1} \gamma)} \leq \sup_{\gamma\in
      G_x}\sum_{\tau \in G^{r(\gamma)}} \norm{(f_n-f)(\tau)}
    \norm{\xi(\tau^{-1} \gamma)}\\ \leq \norm{f_n -f}_\infty
    \sup_{\gamma\in G_x} \sum_{\tau \in G^{r(\gamma)}}
    \norm{\xi(\tau^{-1} \gamma)} \to 0 \quad \textup{as } n\to \infty
  \end{multline*}
  The first equality of the above computation follows from
  Equation~\eqref{eq:ind-corr-bimod-5} of
  Lemma~\ref{lem:ind-corr-bimod}.
\end{proof}

\begin{lemma}\label{lem:cont-left-action}
  Let \(\xi,\zeta \in \Cc(G_x;\A|_{G_x})\). Then the map
  \(f\mapsto \binpro{\xi}{f*'\zeta}\) from \(\Cc(G;\A)\) to
  \(\Cst(G^x_x;\A(x))\) is continuous in the inductive limit topology.
\end{lemma}
\begin{proof}
  Let \(K\) be a compact subset of \(G\) and \((f_n)_{n\in \N}\) a
  sequence of \(\Cc(G;\A)\) that converges uniformly to
  \(f\in \Cc(G;\A)\) on \(K\). Then the
  \(\supp\big(\inpro{\xi}{f_n*'\zeta}\big)\) and
  \(\supp\big(\inpro{\xi}{f*'\zeta}\big)\) are also contained in a
  fixed compact set.  Using Equation~\eqref{eq:ind-corr-bimod-5} of
  Lemma~\ref{lem:ind-corr-bimod}. and the Cauchy--Schwarz inequality,
  we have
  \begin{multline*}
    \norm{\binpro{\xi}{f_n*'\zeta} - \binpro{\xi}{f*'\zeta}}^2 = \norm{\binpro{\xi}{(f_n-f)*'\zeta}}^2 = \norm{\binpro{\xi}{(f_n-f)*'\zeta}^*\binpro{\xi}{(f_n-f)*'\zeta}}\\
    \leq
    \norm{\inpro{\xi}{\xi}}\norm{\binpro{(f_n-f)*'\zeta}{(f_n-f)*'\zeta}}.
  \end{multline*}
  Lemma~\ref{lem:for-cont-left-action}(2) says that
  \(f_n*'\zeta \to f*'\zeta\) in \(\Cc(G_x;\A|_{G_x})\) in the
  inductive limit topology. From
  Lemma~\ref{lem:for-cont-left-action}(1) we have
  \(\binpro{(f_n-f)*'\zeta}{(f_n-f)*'\zeta} \to 0\) in
  \(\Cst(G^x_x;\A(x))\). Therefore, the right hand side of above
  inequality converges to zero as \(n\to \infty\).
\end{proof}

\begin{lemma}\label{lem:inpro-positive}
  \begin{enumerate}[leftmargin=*]
  \item For \(\xi\in \Contc(G_x;\A|_{G_x})\),
    \(\inpro{\xi}{\xi}\geq 0\) in the pre-\(\Cst\)\nb-algebra
    \(\Contc(G_x^x;\A(x))\).
  \item The set
    \(\{\inpro{\xi}{\zeta}: \xi,\zeta\in \Contc(G_x; \A|_{G_x})\}\) is
    dense in \(\Contc(G_x^x; \A(x))\) in the inductive limit topology.
  \item The set
    \(\{g*'\xi : g\in \Contc(G;\A), \xi\in \Contc(G_x;\A|_{G_x})\}\)
    is dense in \(\Contc(G_x;\A|_{G_x})\) in the inductive limit
    topology.
  \end{enumerate}

\end{lemma}
\begin{proof}
  \noindent (1): Let \(a\cdot \delta_\gamma \in
  \Cc(G_x;\A|_{G_x})\). Then for \(\eta \in G^x_x\), we have
  \[
    \binpro{a\cdot \delta_\gamma}{a\cdot \delta_\gamma}(\eta) =
    \begin{cases}
      0 & \quad \textup{ if } \eta \neq x,\\
      a^*a & \quad \textup{ if } \eta = x.
    \end{cases}
  \]
  Since for a vector \(a\in \A_\gamma\), \(a^*a\) is a positive
  element of \(\A_x\), there is an element \(b\in \A_x\) such that
  \(a^*a =b^*b\) and
  \(\binpro{a \cdot \delta_\gamma}{a \cdot \delta_\gamma} = (b \cdot
  \delta_x)^**(b\cdot \delta_x)\). Therefore,
  \(\binpro{a\cdot \delta_\gamma}{a \cdot \delta_\gamma} \geq 0\) in
  \(\Cc(G^x_x;\A(x))\).  Let
  \(a_1 \cdot \delta_{\gamma_1} , a_2 \cdot \delta_{\gamma_2} \in
  \Cc(G_x;\A|_{G_x})\). Then for \(i,j\in \{1,2\},\) and
  \(\eta \in G^x_x\), we have
  \[
    \binpro{ a_i\cdot \delta_{\gamma_i}}{ a_j\cdot
      \delta_{\gamma_j}}(\eta) =
    \begin{cases}
      a_i^*a_j & \quad \textup{ if } r(\gamma_i) = r(\gamma_j) \text{ and } \eta = \gamma^{-1}_i\gamma_j,\\
      0 & \quad \textup{ elsewhere. }
    \end{cases}.
  \]

  \noindent Let
  \(a_1 \cdot \delta_{\gamma_1} + a_2 \cdot \delta_{\gamma_2} \in
  \Cc(G_x;\A|_{G_x})\). We shall prove
  \(\binpro{a_1 \cdot \delta_{\gamma_1} + a_2 \cdot
    \delta_{\gamma_2}}{a_1 \cdot \delta_{\gamma_1} + a_2 \cdot
    \delta_{\gamma_2}} \geq 0\). The computation are done for two
  cases, namely, \(r(\gamma_1) \neq r(\gamma_2)\) and
  \(r(\gamma_1) = r(\gamma_2)\).

  \noindent Case I: Let \(r(\gamma_1) \neq r(\gamma_2)\). Then
  \[
    \binpro{a_1 \cdot \delta_{\gamma_1} + a_2 \cdot
      \delta_{\gamma_2}}{a_1 \cdot \delta_{\gamma_1} + a_2 \cdot
      \delta_{\gamma_2}} = a_1^*a_1 \cdot \delta_x + a^*_2a_2 \cdot
    \delta_x = \big(a_1^*a_1 + a^*_2a_2\big) \cdot \delta_x.
  \]
  Since \(a_1^*a_1+a_2^*a_2\) is a positive elements of \(\A_x\),
  there is \(b\in \A_x\) such that \((a_1^*a_1+a_2^*a_2) =b^*b
  \). Therefore, we have
  \[
    \binpro{a_1 \cdot \delta_{\gamma_1} + a_2 \cdot
      \delta_{\gamma_2}}{a_1 \cdot \delta_{\gamma_1} + a_2 \cdot
      \delta_{\gamma_2}} = b^*b \cdot \delta_x = (b\cdot
    \delta_x)^**(b\cdot \delta_x)
  \]
  which is a positive element in \(\Cc(G_x^x;\A(x))\).

  \noindent Case II: Let \(r(\gamma_1) = r(\gamma_2)\). We have

\[
  \binpro{a_1 \cdot \delta_{\gamma_1} + a_2 \cdot
    \delta_{\gamma_2}}{a_1 \cdot \delta_{\gamma_1} + a_2 \cdot
    \delta_{\gamma_2}} = a_1^*a_1 \cdot \delta_x + a^*_1a_2 \cdot
  \delta_{\gamma_1^{-1}\gamma_2} + a^*_2a_1 \cdot
  \delta_{\gamma_2^{-1}\gamma_1} + a^*_2a_2 \cdot \delta_x.
\]
Since the Fell bundle \(\A\) is saturated, \(\A_{\gamma_2}\) is a full
(left) Hilbert \(\A_{r(\gamma_2)}\)\nb-module and
by~\cite[Lemma~7.3]{Lance1995Hilbert-modules} there are
\(\eta^i_\alpha \in \A_{\gamma_2}\) such that
\(\big(\sum_{i=1}^{n_\alpha}
\Linpro{\eta^i_\alpha}{\eta^i_\alpha}\big)_\alpha\) is an approximate
identity of \(\A_{r(\gamma_2)} = \A_{r(\gamma_1)}\) where
\(n_\alpha \in \N\).  Then the above sum can be written as

\begin{multline*}
  \lim_{\alpha} \big\{ a_1^* \sum_{i=1}^{n_\alpha} \Linpro{\eta^i_\alpha}{\eta^i_\alpha}a_1 \cdot \delta_x + a^*_1 \sum_{i=1}^{n_\alpha} \Linpro{\eta^i_\alpha}{\eta^i_\alpha} a_2 \cdot \delta_{\gamma_1^{-1}\gamma_2} + a^*_2 \sum_{i=1}^{n_\alpha} \Linpro{\eta^i_\alpha}{\eta^i_\alpha}a_1 \cdot \delta_{\gamma_2^{-1}\gamma_1}\\ + a^*_2 \sum_{i=1}^{n_\alpha} \Linpro{\eta^i_\alpha}{\eta^i_\alpha} a_2 \cdot \delta_x \big\}\\
  = \lim_{\alpha}\sum_{i=1}^{n_\alpha} \big\{a_1^*\eta^i_\alpha (\eta^i_\alpha)^*a_1 \cdot \delta_x + a^*_1 \eta^i_\alpha (\eta^i_\alpha)^* a_2 \cdot \delta_{\gamma_1^{-1}\gamma_2} + a^*_2 \eta^i_\alpha (\eta^i_\alpha)^*a_1 \cdot \delta_{\gamma_2^{-1}\gamma_1} + a^*_2 \eta^i_\alpha (\eta^i_\alpha)^*a_2 \cdot \delta_x\big\}\\
  = \lim_{\alpha} \sum_{i=1}^{n_\alpha} \big(a_1^*\eta^i_\alpha \cdot
  \delta_{\gamma_1^{-1}\gamma_2} + a_2^*\eta^i_\alpha \cdot
  \delta_x\big) \big(a_1^*\eta^i_\alpha \cdot
  \delta_{\gamma_1^{-1}\gamma_2} + a_2^*\eta^i_\alpha \cdot
  \delta_x\big)^*
\end{multline*}
which is a positive element in \(\Cc(G^x_x;\A(x))\).  Therefore,
\((a_1 \cdot \delta_1 + a_2 \cdot \delta_2)\) is a positive
element. Proceeding by induction, we can show that
\(\inpro{\xi}{\xi}\geq 0\) any
\(\xi =\sum_{i=1}^{n} a_i\cdot \delta_{\gamma_i} \in
\Cc(G_x;\A|_{G_x}) \).

\noindent (2): Consider a generating element \(a\cdot \delta_\eta\) of
\(\Cc(G^x_x;\A(x))\) where \(a\in \A_\eta\) and \(\eta \in
G^x_x\). Let \((u^n_x)_{n\in \N}\) be an approximate identity of
\(\A_x\).  Then
\(\binpro{u^n_x \cdot \delta_x}{a\cdot
  \delta_\eta}=u^n_xa\cdot\delta_\eta = (u^n_x \cdot \delta_x)*
(a\cdot\delta_\eta) \). Lemma~\ref{lemma:apprx-id-of-bimod} says that
\(( u^n_x \cdot \delta_x)_n\) forms an approximate identity for
\(\Cst(G^x_x; \A(x))\) and we have
\[
  \binpro{u^n_x \cdot \delta_x}{a\cdot \delta_\eta}= (u^n_x \cdot
  \delta_x) *(a\cdot\delta_\eta) \to a \cdot \delta_{\eta}
\]
in the inductive limit toplogy on \(\Cc(G^x_x;\A(x))\) as
\(n \to \infty\).

\noindent (3): Follows from the fact \(\Cst(G;\A)\) has an approximate
identity of elements in \(\Cc(G;\A)\).
  
\end{proof}

Lemma~\ref{lem:inpro-positive}(1) shows that the
\(\Contc(G_x^x;A(x))\)\nb-valued conjugate bilinear form
\(\inpro{\cdot}{\cdot}\) on \(\Contc(G_x;\A|_{G_x})\) is also
positive. Let \(Y(x)\) be the Hilbert \(\Cst(G_x^x;\A(x))\)\nb-module
obtained by completing the pre-Hilbert \(\Cst(G_x^x;\A(x))\)\nb-module
\(\Contc(G_x; \A|_{G_x})\). Lemma~\ref{lem:inpro-positive}(2) implies
that \(Y(x)\) is a full Hilbert module.

Our next task is to prove that the action of \(\Contc(G;\A)\) in
Equation~\eqref{eq:ind-corr-1} extends to a nondegenerate
representation of \(\Cst(G;\A)\) on \(Y(x)\) by adjointable operators.
Equation~(\ref{eq:ind-corr-bimod-11}) says that the dense subalgebra
\(\Cc(G;\A)\) acts \(\Cc(G_x;\A|_{G_x})\) by adjointable operator. To
show this action is bounded we use a similar argument as used in
\cite[Lemma~2.9]{Holkar2017Construction-of-Corr}. Let \(\psi\) be a
state on \(\Cst(G^x_x;\A(x))\). Then \(Y(x)_\psi\) became a Hilbert
space with the inner product given by
\(\inpro{a}{b}_{\psi}\defeq \psi(\inpro{a}{b})\) where
\(a,b \in Y(x)\). Consider a vector subspace \(V\) of the Hilbert
space \(Y(x)_\psi\) generated by
\(\{f*'\xi:f\in \Cc(G;\A),\xi \in \Cc(G_x;\A|_{G_x})\}\). Define a
representation \(\pi\) of \(\Cc(G;\A)\) on \(V\) by
\(\pi(f)\xi=f*'\xi\).  The idea is to show that \(\pi\) is a
pre-representation (see~\cite[Definition
4.1]{Muhly-Williams2008FellBundle-ME}) of \(\Cc(G;\A)\) on
\(V \subseteq Y(x)_\psi\). We notics the following facts:
\begin{enumerate}
\item The representation \(\pi\) of \(\Cc(G;\A)\) is nondegenerate on
  \(V\) because \(\Cc(G;\A)*\Cc(G;\A) \subseteq \Cc(G;\A)\) is dense.
	
\item From Lemma~\ref{lem:cont-left-action}, for any
  \(\xi,\zeta \in \Cc(G_x;\A|_{G_x})\), the map
  \(f\mapsto \inpro{\xi}{\pi(f)\zeta}_{\psi}\) is a continuous
  function on \(\Cc(G;\A)\) when \(\Cc(G;\A)\) has inductive limit
  topology.
	
\item Equation~\eqref{eq:ind-corr-bimod-11} says that \(\pi\) is
  involutive, that is:
  \( \inpro{f*'\xi}{\zeta}_{\psi} = \inpro{\xi}{f^**'\zeta}_{\psi}.
  \)
\end{enumerate}
Muhly--Williams' disintegration theorem for Fell
bundles~\cite[Theorem~4.13]{Muhly-Williams2008FellBundle-ME} says that
\(\pi\) is equivalent to the integrated form of a representation of
the Fell bundle \(p\colon \A \to G\). Therefore \(\pi\) is bounded
with respect to the norm of \(\Cst(G;\A)\). Thus
\(\psi(\inpro{f*'\xi}{f*'\xi}) \leq
\norm{f}^2_{\Cst(G;\A)}\psi(\inpro{\xi}{\xi})\) for all
\(f\in \Cc(G;\A)\) and \(\xi \in \Cc(G_x;\A|_{G_x})\). Since \(\psi\)
was an arbitrary state, for all \(f\in \Cc(G;\A)\) and
\(\xi \in \Cc(G_x;\A|_{G_x})\) we have
\[
  \inpro{f*'\xi}{f*'\xi}\leq \norm{f}^2_{\Cst(G;\A)} \inpro{\xi}{\xi}.
\]
Hence, the action can be extended to \(\Cst(G;\A)\).
\begin{theorem}\label{thm:ind-cst-corr}
  Let \(p\colon \A \to G\) be a Fell bundle over a Hausdorff, locally
  compact, second countable, {\'e}tale groupoid \(G\). Let
  \(x\in G^{(0)}\). Then the \(\Contc(G;
  \A)\)-\(\Contc(G_x^x; \A(x))\)\nb-bimodule \(\Contc(G_x;\A|_{G_x})\)
  completes to a \(\Cst\)\nb-correspondence
  \(Y(x)\colon \Cst(G;\A)\to \Cst(G^x_x; \A(x))\).
\end{theorem}
We call~\(Y(x)\) above the induction correspondence from
\(\Cst(G;\A)\) to \(\Cst(G^x_x; \A(x))\). This idea is same as that of
the induction correspondence described in Example~3.14
in~\cite{Holkar2017Construction-of-Corr}.
\begin{proof}[Proof of Theorem~\ref{thm:ind-cst-corr}]
  The proof follows from Lemma~\ref{lem:inpro-positive} and the above
  discussion.
\end{proof}

 \begin{corollary}\label{cor:ind-rep-associated-with-ind-corr}
   Let \(p\colon \A \to G\) be a Fell bundle over a Hausdorff, locally
   compact, second countable, {\'e}tale groupoid \(G\). Let
   \(x\in G^{(0)}\). Then a representation
   \(\pi\colon \Cst(G^x_x; \A(x))\to B(\Hilm)\) induces a
   representation \(\Ind_{\Cst(G^x_x; \A(x))}^{\Cst(G; \A)}(\pi)\) of
   \(\Cst(G;\A)\).
 \end{corollary}
 \begin{proof}
   Follows from Theorem~\ref{thm:ind-cst-corr}
   and~\cite[Proposition~2.66]{Raeburn-Williams1998ME-book}.
 \end{proof}

 \subsection{Induction of states}
 \label{sec:inducing-states}

 As in the earlier section, let \(p\colon \A \to G\) be a Fell bundle
 over a locally compact, Hausdorff, second countable, \etale\ groupoid
 and \(x\in \base[G]\).  In light of
 Corollary~\ref{cor:ind-rep-associated-with-ind-corr}, we raise the
 question if a state on \(\Cst(G_x^x;\A(x))\) induces a state of
 \(\Cst(G;\A)\). In other words, consider a cyclic representation of
 \(\Cst(G_x^x;\A(x))\); then does the induction correspondence induce
 it to a cyclic representation of \(\Cst(G;\A)\)?  The answer to this
 is affirmative. We need a cyclic vector of this induced
 representation the construction of which we discuss in this section.

 Assume that \(\phi\) is a state on \(\Cst(G_x^x;\A(x))\) and
 \((\Hils, L, \xi)\) the corresponding GNS triple. Then the induction
 correspondence \(Y(x)\colon \Cst(G;\A)\to \Cst(G_x^x;\A(x))\) induces
 the representation~\(L\) to the representation
 \(\IndP(L)\colon \Cst(G;\A)\to \Bound(Y(x)\otimes \Hils)\) of
 \(\Cst(G;\A)\).  We show that the induced representation of
 \(\Cst(G;\A)\) on \(Y(x)\otimes \Hils\) has a \emph{cyclic vector}
 \(a(x)\). Recall from~\cite[Page~33,
 34]{Raeburn-Williams1998ME-book}, that the inner product of
 \(Y(x) \otimes \Hils\) is given by
 \(\inpro{a\otimes \xi}{b \otimes \zeta} =
 \inpro{\xi}{L\inpro{a}{b}\zeta}\) for \(a,b\in Y(x)\) and
 \(\xi,\zeta \in \Hils\).
 
 \begin{lemma}\label{lem:compression-by-app-unit}
   Let \(x\in \base\) and \(d^n_x\) as defined in
   Lemma~\ref{lemma:apprx-id-of-bimod} for \(n\in \N\). Then for
   \(f\in \Cc(G;\A)\), \( \binpro{d^n_x}{f*'d^m_x} \to f|_{G^x_x}\)
   pointwise as \(m,n \to \infty\) where \(\inpro{\cdot}{\cdot}\) is a
   \(\Cc(G^x_x;\A(x))\)\nb-valued inner product defined in
   Equation~\eqref{eq:ind-corr-3}. By \(m,n \to \infty,\) we mean one
   of the itterated limits \(m\to \infty\) then \(n\to \infty\) or
   \(n\to \infty\) then \(m\to \infty\).
 \end{lemma}
 
 \begin{proof}
   Lemma~\ref{lem:inpro-conv-equivalence} and the positivity of
   \(d^n_x\) gives us
   \(\binpro{d^n_x}{f*'d^m_x} = (d^n_x)^**(f*'d^m_x)|_{G^x_x} =
   d^n_x*(f*'d^m_x)|_{G^x_x}\).  Recall the definition of \(*'\) from
   Equation~\eqref{eq:ind-corr-1}. For \(\gamma \in G^x_x\) ,
   \begin{multline*}
     \binpro{d^n_x}{f*'d^m_x} (\gamma) = d^n_x*(f*'d^m_x)|_{G^x_x} (\gamma) = \sum_{\alpha \in G^x_x} d^n_x(\alpha) (f*'d^m_x)|_{G^x_x}(\alpha^{-1}\gamma) \\
     = d^n_x(x) (f*'d^m_x)|_{G^x_x}(\gamma) = u^x_n \sum_{\tau \in
       G^x} f(\tau) d^m_x(\tau^{-1}\gamma) = u^x_nf(\gamma) d^m_x (x)=
     u^x_n f(\gamma) u^x_m.
   \end{multline*}
   Again \((u^x_x)_{n\in \N}\) is a approximate identity of \(\A_x\)
   and \(\A_\gamma\) is a \(\A_x\)-\(\A_x\)\nb-Hilbert bimodule for
   \(\gamma \in G^x_x\). Using
   Lemma~\ref{lem:appx-unit-for-full-Hilbert-mod} and taking limit in
   the above computation, we have
   \[
     \lim_{m,n} \binpro{d^n_x}{f*'d^m_x} (\gamma) = \lim_{m,n} u^x_n
     f(\gamma) u^x_m = f(\gamma)
   \]
   for \(\gamma\in G^x_x\).
 \end{proof}

 \begin{deflem}\label{lem:appx-unit-for-ind-corr}
   Let \(\phi\) be a state on \(\Cst(G^x_x;\A(x))\) with GNS triple
   \((\Hils, L,\xi)\). Then the sequence
   \((d^n_x\otimes\xi)_{n\in \N}\) in the Hilbert space
   \(Y(x)\otimes \Hils\) is convergent. Denote this limit
   by~\(a(x)\). Then \(a(x)\) is a unit vector.
 \end{deflem}
 \begin{proof}
   To establish the convergence, it is sufficient to show that the
   sequence is Cauchy.  Consider
   \begin{multline*}
     \norm{d^n_x\otimes\xi-d^m_x\otimes\xi}^2 =
     \inpro{d^n_x\otimes\xi-d^m_x\otimes\xi}{d^n_x\otimes\xi
       -d^m_x\otimes\xi} \\ =
     \inpro{d^n_x\otimes\xi}{d^n_x\otimes\xi}-\inpro{d^n_x\otimes\xi}{d^m_x\otimes\xi}
     -\inpro{d^m_x\otimes\xi}{d^n_x\otimes\xi}+\inpro{d^m_x\otimes\xi}{d^m_x\otimes\xi}
     \\
     =\inpro{\xi}{L(\inpro{d^n_x}{d^n_x})\xi}-\inpro{\xi}{L(\inpro{d^n_x}{d^m_x})\xi}
     -\inpro{\xi}{L(\inpro{d^m_x}{d^n_x})\xi}+\inpro{\xi}{L(\inpro{d^m_x}{d^m_x})\xi}.
   \end{multline*}
   Due to Lemma~\ref{lem:inpro-conv-equivalence} and positive natures
   of \(d^n_x\) and \(d^m_x\), we identify
   \(\inpro{d^n_x}{d^m_x} = d^n_x*d^m_x\) and similarly for the other
   terms. Therefore, the last term above equals
   \[
     \inpro{\xi}{L((d^n_x)^2)\xi}-\inpro{\xi}{L(d^n_x*d^m_x)\xi}
     -\inpro{\xi}{L(d^m_x*d^n_x)\xi}+\inpro{\xi}{L((d^m_x)^2)\xi}.
   \]
   Since, for all \(p,q\in \{n,m\}\),
   \(L(d^p_x*d^q_x)\to \Id_{\Hils}\) in the SOT, (see
   Corollary~\ref{cor:cros-app-id}) the above term converges to
   \(\inpro{\xi}{\xi}-\inpro{\xi}{\xi}-\inpro{\xi}{\xi}+\inpro{\xi}{\xi}=0\)
   as \(n,m \to \infty\).  Thus \(\{d^n_x\otimes \xi\}_{n\in \N}\) is
   a Cauchy sequence in the Hilbert space \(Y(x)\otimes \Hils\); let
   \(a(x)\) be the limit.
 
   To show that \(a(x)\) is an unit, consider
   \begin{multline*}
     \norm{a(x)}^2 = \norm{\lim_{n}d^n_x\otimes\xi}^2=
     \lim_{n}\norm{d^n_x\otimes\xi}^2=\lim_{n} \inpro{d^n_x\otimes
       \xi}{d^n_x\otimes\xi}=\lim_{n}\inpro{\xi}{L(\inpro{d^n_x}{d^n_x})\xi}
     \\=\lim_{n}\inpro{\xi}{L((d^n_x)^2)
       \xi}=\inpro{\xi}{\xi}=\norm{\xi}^2=1.
   \end{multline*}

 \end{proof}

 Towards the end, we need the following small lemma to prove
 Proposition~\ref{prop:cyclic-vect-for-ind-rep} which will come
 immediately after the lemma.

\begin{lemma}\label{lem:for-cyclic-vect}
  Suppose \(\phi\) be a state on \(\Cst(G^x_x;\A(x))\) with GNS triple
  \((\Hils, L, \xi)\). Let \(f\in \Contc(G;\A)\).
  \begin{enumerate}[leftmargin=*]
  \item \(f*' d_x^n \to f|_{G_x}\) in \(\Contc(G_x;\A|_{G_x})\) in the
    inductive limit topology. Hence \(f*'d^n_x\to f|_{G_x}\) in the
    norm topology in \(\Cc(G_x;\A|_{G_x}) \subseteq Y(x)\).
  \item \(\IndP(L) (f) (a(x)) = f|_{G_x} \otimes \xi\) in
    \(Y(x)\otimes \Hils\).
  \item \(\binpro{a(x)}{\IndP(L)(f) (a(x))} = \phi(f|_{G^x_x})\).
		
  \end{enumerate}
\end{lemma}

\begin{proof}
  (1) For given \(f\) and \(\gamma\in G_x\),
  \[
    f*' d_x^n(\gamma) = \sum_{\tau\in G^{r(\gamma)}} f(\tau)
    d_x^n(\tau\inverse \gamma) = f(\gamma)\cdot u^x_n.
  \]
  Therefore, \(\supp (f*'d^n_x) \subseteq \supp(f) \cap G_x\) is a
  fixed finite set. Also we have
  \[
    \norm{f*' d_x^n - f|_{G_x}}_{\infty} = \max_{\gamma \in \supp(f)
      \cap G_x} \norm{f(\gamma)\cdot u^x_n -f(\gamma)} \to 0 \text{ as
    } n\to \infty
  \]
  as \(\A(\gamma)\) is a full Hilbert \(\A_x\)\nb-module, and
  \((u^x_n)_n\) is an approximate identity of \(\A_x\).
	
  \noindent (2) Using the definition of \(a(x)\),
  \[
    \IndP(L)(f)(a(x)) = \lim_n \IndP(L)(f)(d_x^n\otimes \xi) =\lim_n
    (f*' d_x^n)\otimes \xi.
  \]
  The first equality of above follows from the continuity of \(\IndP\)
  and the definition of \(a(x)\).  Using
  Lemma~\ref{lem:for-cont-left-action}(2), we get
  \(\lim_n (f*' d_x^n)\otimes \xi = f|_{G_x} \otimes \xi\).  Thus
  \[
    \IndP(L)(f)(a(x)) = f|_{G_x}\otimes \xi.
  \]
	
  \noindent (3) Let \(f\in \Cc(G;\A)\).  Using the continuity of
  \(\IndP\) and the definiton of \(a(x)\), we get
  \begin{multline*}
    \binpro{a(x)}{\IndP(L)(f) (a(x))} =  \lim_{n,m}\binpro{d^n_x\otimes \xi}{\IndP (L)(f) (d^m_x\otimes \xi)}\\
    = \lim_{n,m}\binpro{d^n_x\otimes \xi}{(f*'d^m_x)\otimes \xi}.
  \end{multline*}
  By \(\lim_{n,m}\) we mean one of the iterated limits \(m\to \infty\)
  then \(n\to \infty\) or \(n\to \infty\) then \(m\to \infty\). Recall
  the defintion of \(*'\) from Equation~\eqref{eq:ind-corr-1}.  The
  definition of the inner product on the interior tensor product of
  Hilbert modules says that the above term is equal to
  \[
    \lim_{n,m}\binpro{\xi}{L\big(\inpro {d^n_x}{f*'d^m_x}\big) \xi} =
    \binpro{\xi}{L\big(\lim_{n,m} \inpro {d^n_x}{f*'d^m_x}\big) \xi}.
  \]
  The above computation together with
  Lemma~\ref{lem:compression-by-app-unit} gives us
  \[
    \binpro{a(x)}{\IndP(L)(f) (a(x))} = \binpro{\xi}{L( f|_{G_x^x})
      \xi} = \phi(f|_{G_x^x})
  \]
  for \(f\in \Cc(G;\A)\).
\end{proof}

\begin{proposition}
  \label{prop:cyclic-vect-for-ind-rep}
  Suppose that \(G\) is a locally compact, Hausdorff, second
  countable, \etale\ groupoid, \(x\in \base[G]\), and that
  \(p\colon \A \to G\) is a saturated Fell bundle. Let
  \(Y(x)\colon \Cst(G;A) \to \Cst(G_x^x;\A(x))\) be the induction
  correspondence. Assume that a state \(\phi\) on
  \(\Cst(G_x^x;\A(x))\) and its GNS triple \((\Hils, L, \xi)\) is
  given. Then \(a(x)\) is a unit cyclic vector for \(\IndP(L)\)
  where~\(a(x)\) is given by Lemma~\ref{lem:appx-unit-for-ind-corr}.
\end{proposition}
\begin{proof}
  Let \(A\subseteq Y(x)\otimes \Hils\) denote the linear span of the
  set
  \[
    \big\{\IndP(L)(f) (a(x)) : f\in \Contc(G;\A)\big\}.
  \]
  We need to show is that \(a(x)\) is a cyclic vector for the induced
  representation, that is, \(A\) is dense in \(Y(x)\otimes \Hils\).
  
  Let \(B\) denote the vector space spanned by the elementary tensors
  \(g\otimes L(h)\xi\), for \(g\in \Contc(G_x;\A|_{G_x})\) and
  \(h\in \Contc(G_x^x;\A(x))\), in the interior tensor product
  \(Y(x)\otimes_{\Cst(G_x^x;\A(x))}\Hils\). Since \(\xi\) is a cyclic
  vector for the representation \(L\) of \(\Cst(G_x^x;\A(x))\)
  on~\(\Hils\), \(B\) is a dense subspace of
  \(Y(x)\otimes_{\Cst(G_x^x;\A(x))}\Hils\). We show that
  \(B\subseteq A\) which will prove the proposition.
  
  For this purpose, consider \(g\otimes L(h)\xi\in B\) where
  \(g\in \Cc(G_x;\A|_{G_x})\) and \(h \in \Cc(G^x_x;\A(x))\). Since
  the tensor product \(\otimes\) balanced over \(\Cst(G^x_x;\A(x))\),
  \(g\otimes L(h)\xi = g*''h\otimes \xi\) where we know that
  \(g*''h\in \Contc(G_x;\A|_{G_x})\). For time being, assume that
  there is \(f\in \Contc(G;\A)\) such that \(f|_{G_x} = g*''h\). Then
  Lemma~\ref{lem:for-cyclic-vect}(2) tell us that
  \[
    (g*'' h)\otimes \xi = f|_{G_x}\otimes \xi = \IndP(L)(f)(a(x)) \in
    A.
  \]
  Thus, we need to prove the existence of \(f\in \Cc(G;\A)\) such that
  \(f|_{G_x} = g*''h\) to conclude the proof.

  Instead of proving the existence of \(f\), we can prove a general
  assertion, namely, given \(k\otimes \xi\in B\) where
  \(k\in \Contc(G_x;\A|_{G_x})\), there is
  \(\tilde{k}\in \Contc(G;\A)\) such that \(\tilde{k}|_{G_x} = k\).
  
  The construction of \(\tilde{k}\) is as follows: let
  \(\supp(k) = \{\gamma_1,\gamma_2,\cdots,\gamma_n\}\) where \(n\) is
  a natural number. Since the Fell bundle \(p\colon \A\to G \) has
  enough sections choose \(s_i\in \Cc(G;\A)\) such that
  \(s_i(\gamma_i)= k(\gamma_i)\) for \(i=1,2,\cdots,n\). Choose
  \(b_1,b_2,\cdots,b_n\in \Cc(G)\) such that
  \[
    b_i(\gamma_j)=~\begin{cases}
      1 & \quad \text{ when } i=j,\\
      0 & \quad \text{ when } j\in \{1,\dots,n\}\7 \{j\}.
    \end{cases}
  \]
  Let \(\tilde{k} = \sum_{i=1}^{n}b_is_i\). Then clearly
  \(\tilde{k}|_{G_x} = k.\)
\end{proof}

\section{An integration-disintegration theorem for states}
\label{sec:int-dis-thm-state}

As in earlier section, we fix a Fell bundle \(p\colon \A\to G\) over a
locally compact, Hausdorff, \etale\ groupoid~\(G\). In this section,
in Proposition~\ref{prop:integration-of-states}, we prove that a state
on \(\Cst(G;\A)\) disintegrates into a probability measure on
\(\base[G]\) and a \(\mu\)\nb-measurable field of states. The
converse, namely, the integration of a \(\mu\)\nb-measurable field of
states into a state on \(\Cst(G;\A)\) is described in
Proposition~\ref{thm:disint-states}. Finally,
Theorem~\ref{prop:bij-betwen-states-and-fields-of-states} establishes
the one-to-one correspondence of these two processes.

Let \(\phi\) be a state on a \(\Cst\)-algebra \(A\). The centraliser
of \(\phi\) is the subalgebra of \(A\) consisting elements \(a\in A\)
with the proeprty that \(\phi(ab)=\phi(ba)\) for all \(b\in A\).

\begin{definition}[{A \(\mu\)-measurable field of states
    (\cite[Page~515.]{Neshveyev2013KMS-states})}]
  \label{def:mbl-field-state}
  Let \(\mu\) be a probability measure on \(G^{(0)}\).
  \begin{enumerate}
  \item A field of state \(\{\phi_x\}_{x\in G^{(0)}}\), where
    \(\phi_x\) is a state on the \(\Cst\)-algebra
    \(\Cst(G^x_x;\A(x))\), is called a \(\mu\)-measurable field if for
    every \(f\in \Contc(G;\A)\) the map
    \[
      \base \to \C, \quad x\mapsto \sum_{\gamma\in
        G^x_x}\phi_x(f(\gamma)\delta_\gamma)
    \]
    is \(\mu\)-measurable.
  \item Given two \(\mu\)\nb-measurable field of states
    \(\{\psi_x\}_{x\in G^{(0)}}\) and \(\{\phi_x\}_{x\in\base[G]}\) on
    \(\base[G]\) are called equivalent if \(\psi_x=\phi_x\)
    \(\mu\)\nb-almost everywhere on \(\base[G]\).
  \end{enumerate}
\end{definition}

\noindent Clearly, being equivalent is an equivalence relation on the
collection of \(\mu\)\nb-measurable field of states. For such a field
\(\Phi = \{\phi_x\}_{x\in \base[G]}\), we denote its equivalence class
by~\([\Phi]\).

\begin{proposition}
  \label{prop:integration-of-states}
  Let \(p\colon \A \to G\) be a Fell bundle over a Hausdroff, locally
  compact, second countable, {\'e}tale groupoid \(G\). Let \(\phi\) be
  a state on \(\Cst(G;\A)\) whose centraliser is contained
  \(\Contz(G^{(0)};\A|_{G^{(0)}})\). Then there is probability measure
  \(\mu\) on the unit space \(G^{(0)}\) and a \(\mu\)\nb-measurable
  field~\(\{\phi_x\}_{x\in G^{(0)}}\) where \(\phi_x\) is a state on
  \(\Cst(G^x_x;\A(x))\) whose centraliser is contained \(\A_x\); and
  the pair \((\mu,\{\phi_x\}_{x\in \base[G]})\) satisfies the
  following equation
  \begin{equation}\label{eq:states-int-disint-reln}
    \phi(f)
    = \int_{G^{(0)}} \biggr( \sum\limits_{\gamma\in G^x_x} \phi_x(f(\gamma)\cdot
    \delta_\gamma) \biggr)  \textup{d}\mu(x)
  \end{equation}
  for \(f\in \Contc(G; \A)\).
\end{proposition}
\begin{proof}
  Consider the GNS triple \((\Hils, L, \xi)\) associated with the
  state \(\phi\), that is, \(\Hils\) is the Hilbert space associated
  with \(\phi\); \(L\) is the GNS representation of \(\Cst(G;\A)\) on
  a Hilbert space~\(\Hils\), and \(\xi\) is the cyclic vector for
  \(L\) with the property that \(\phi(x)= \binpro{\xi}{L(x)\xi}\) for
  \(x\in \Cst(G;\A)\). Let \((\nu, G^{(0)}*\Hilm, \hat{\pi})\) be the
  disintegration of \(L\) given by Muhly--Williams' disintegration
  theorem for Fell bundles (\cite[Theorem
  4.13]{Muhly-Williams2008FellBundle-ME}). Then~\(\nu\) is a
  quasi-invariant measure on \(\base\) and the Hilbert space \(\Hils\)
  is identified with the direct integral \(L^2(G^{(0)}*\Hilm,\nu)\)
  consisting the \(\nu\)\nb-square-integrable sections of the bundle
  of Hilbert spaces \(G^{(0)}*\Hilm\to \base[G]\). Let
  \(\{\xi_x\}_{x\in G^{(0)}}\) be the vector field corresponding to
  the cyclic vector~\(\xi\).  The representation \(L\) is unitarily
  equivalent to the integrated form of the
  \(*\)\nb-functor~\(\hat{\pi}\). Prompted by
  Equation~\eqref{equ-int-dis-rep}, we define the linear functional on
  \(\phi_x\colon \Cc(G^x_x;\A(x))\to \C\) as follows: for
  \(\sum_{i=1}^n a_i\cdot \delta_{\gamma_i} \in \Cc(G^x_x;\A(x))\),
  \[
    \phi_x\biggr(\sum_{i=1}^n a_i\cdot \delta_{\gamma_i}\biggr) =
    \sum_{i=1}^n \norm{\xi_x}^{-2}\binpro{\xi_x}{\pi(a_i)\xi_x}.
  \]
  We next show that the functional \(\phi_x\), in fact, extends to a
  state on \(\Cst(G^x_x;\A(x))\): let
  \(\sum_{i=1}^n a_i\cdot \delta_{\gamma_i} \in
  \Cc(G^x_x,\A(x))\). Then
  % \end{align*}
  \begin{multline*}
    \phi_x \biggr( \big(\sum_{i=1}^n a_i\cdot
    \delta_{\gamma_i}\big)^*\big(\sum_{i=1}^n a_i\cdot
    \delta_{\gamma_i}\big)\biggr) = \phi_x
    \biggr(\sum_{i,j=1}^{n}a_i^*a_j\cdot
    \delta_{\gamma_i^{-1}\gamma_j} \biggr) \\=
    \sum_{i,j=1}^{n}\norm{\xi_x}^{-2}\binpro{\xi_x}{\pi(a_i^*a_j)\xi_x}
    =
    \sum_{i,j=1}^{n}\norm{\xi_x}^{-2}\binpro{\pi(a_i)\xi_x}{\pi(a_j)\xi_x}\\
    =\norm{\xi_x}^{-2}\Binpro{\sum_{i=1}^{n}\pi(a_i)\xi_x}{\sum_{i=1}^{n}\pi(a_i)\xi_x}
    \geq 0.
  \end{multline*}
  Now we show that \(\phi_x\) has a unit norm. Recall from
  Definition~\ref{def:star-funtor} that
  \(\pi|_{\A_x}\colon \A_x\to \Bound(\Hilm_x)\) is a nondegenerate
  representation. Then consider the approximate identity
  \((d_x^n)_{n\in \N}\) of \(\Cst(G^x_x; \A(x))\) as
  in~Lemma~\ref{lemma:apprx-id-of-bimod}; recall from the same lemma
  that \(d_x^n = u^x_n\cdot\delta_x\), where \((u^x_n)_{n\in \N}\) is
  an approximate identity of the
  \(\Cst\)\nb-algebra~\(\A_x\). Since~\(\phi_x\) is a positive linear
  functional,
  \begin{equation*}
    \norm{\phi_x} = \lim_{n} (\phi_x (d_x^n)) = \lim_{n}\norm{\xi_x}^{-2}\binpro{\xi_x}{\pi(u^x_n)\xi_x} =\norm{\xi_x}^{-2}\binpro{\xi_x}{\xi_x}  =1
  \end{equation*}
  where the first equality is due to~\cite[Theorem~3.33]{Murphy-Book}
  and the third one follows from the nondegeneracy of \(\pi\). This
  concludes that \(\phi_x\) is a state on \(\Cst(G_x^x;\A(x))\) for
  \(x\in \base[G]\).
	
  Define the measure~ \(\mu\) on \(G^{(0)}\) desired in
  Proposition~\ref{prop:integration-of-states} as
  \(\mathrm{d}\mu(x)=\norm{\xi_x}^2\mathrm{d}\nu(x)\). As \(L\) is the
  GNS representation corresponding to \(\phi\), using
  Equation~\eqref{equ-int-dis-rep}, we have
  \begin{equation}\label{equ:equality-state:dis-state}
    \phi(f)=\binpro{\xi}{L(f)\xi}= \int_{G^{(0)}}\sum_{\gamma \in G^x}\binpro{\xi_x}{\pi(f(\gamma))\xi_{s(\gamma)}}\Delta_\nu(\gamma)^{-\frac{1}{2}}\mathrm{d}\nu (x).
  \end{equation}
  for \(f\in \Cc(G;\A)\) and \(\Delta_\nu\) is the modular function of
  the quasi-invariant measure \(\nu\) on \(\base\). Choose an
  approximate identity \((f_n)_{n\in \N}\) of \(\Cst(G;\A)\) such that
  \(f_n \in \Cc(\base; \A|_{\base})\)
  (cf.~Proposition~\ref{prop:exist-app:unit}). Equation~\eqref{equ:equality-state:dis-state}
  implies that
  \[
    \phi(f_n) = \int_{\base} \binpro{\xi_x}{\pi(f_n(x)) \xi_x} \dd \nu
    (x).
  \]
  Taking limit as \(n\to \infty\) in above equation, and using
  \cite[Theorem~3.33]{Murphy-Book} and dominated convergence theorem
  we have
  \begin{multline*}
    1 = \norm{\phi} = \lim_{n \to \infty} \phi(f_n) = \lim_{n \to \infty} \int_{\base} \binpro{\xi_x}{\pi(f_n(x)) \xi_x} \dd \nu (x) =\int_{\base} \inpro{\xi_x}{\xi_x} \dd \nu(x)\\
    = \int_{\base} \norm{\xi_x}^2 \dd \nu(x).
  \end{multline*}
  Therefore, \(\dd \mu = \norm{\xi_x}^2 \dd \nu\) is a probability
  measure on \(\base\).  The \(\mu\)\nb-measurablity
  of~\(\{\phi_x\}_{x\in \base}\) is follows from the measuability of
  the vector field \(\{\xi_x\}_{x\in \base}\) corrseponding to the
  unit vector \(\xi\in \Hils\).
	
  To establish Equation~\eqref{eq:states-int-disint-reln}, we need to
  show that
  \[
    \sum_{\gamma \in G^x\setminus
      G^x_x}\binpro{\xi_x}{\pi(f(\gamma))\xi_{s(\gamma)}}\Delta_\nu(\gamma)^{-\frac{1}{2}}=0
  \]
  for \(\nu\)-almost every \(x\in G^{(0)}\) in
  Equation~\eqref{equ:equality-state:dis-state}. Since \(G\) is
  \etale\ and in above equation \(\gamma\) comes from outside
  \(G^x_x\), it is suffices to show that for each bisection
  \(U\subseteq G\setminus \textup{Ib}(G)\) with
  \(r(U) \cap s(U) = \emptyset\) and for \(f\in \Cc(U; \A|_{U})\),
  \begin{equation}\label{equ:dis:thm:rel:state:field}
    \sum_{\gamma \in G^x}\binpro{\xi_x}{\pi(f(\gamma))\xi_{s(\gamma)}} \Delta_\nu(\gamma)^{-\frac{1}{2}} = 0
  \end{equation}
  \(\nu\)-almost every \(x\in G^{(0)}\). Choose \(f\) as above and let
  \((u_n)_{n\in \N}\) be an approximate identity of
  \(\Contz(\base; \A|_{\base})\) contained in
  \(\Cc(\base; \A|_{\base})\)
  (cf. Remark~\ref{rem:app-id-in-com-supp}). Let \(k\in \Cc(r(U))\)
  and set \(h_n = k\cdot u_n\). Since \(r(U) \cap s(U) = \emptyset\)
  and \(\supp(f) \in U\) using Remark~\ref{rem:obs-of-bise-out-iso},
  we have \(f*h_n = 0\). Using the fact that \(h_n\) is contained in
  the centraliser of \(\phi\) we have
  \begin{multline*}
    0 = \phi(f*h_n) = \phi(h_n*f) = \int_{G^{(0)}}\sum_{\gamma \in G^x}\binpro{\xi_x}{\pi(h_n*f(\gamma))\xi_{s(\gamma)}}\Delta_\nu(\gamma)^{-\frac{1}{2}}\mathrm{d}\nu (x)\\
    = \int_{r(U)}\sum_{\gamma \in G^x} \binpro{\xi_x}{\pi(k(x)u_n(x)f(\gamma))\xi_{s(\gamma)}}\Delta_\nu(\gamma)^{-\frac{1}{2}}\mathrm{d}\nu (x)\\
    = \int_{r(U)} k(x) \sum_{\gamma \in G^x}
    \binpro{\xi_x}{\pi(u_n(x)f(\gamma))\xi_{s(\gamma)}}\Delta_\nu(\gamma)^{-\frac{1}{2}}\mathrm{d}\nu
    (x).
  \end{multline*}
  \noindent Since \(k\in \Cc(r(U))\) was arbitrary, we have that the
  function \(r(U) \to \C\) given by
  \[
    x\mapsto \sum_{\gamma \in G^x}
    \binpro{\xi_x}{\pi(u_n(x)f(\gamma))\xi_{s(\gamma)}}\Delta_\nu(\gamma)^{-\frac{1}{2}}
  \]
  is zero \(\nu\)\nb-almost everywhere on \(r(U)\). Taking limit as
  \(n\to \infty\), the last equation becomes
  \[
    \sum_{\gamma \in G^x}
    \binpro{\xi_x}{\pi(f(\gamma))\xi_{s(\gamma)}}\Delta_\nu(\gamma)^{-\frac{1}{2}}
    = 0
  \]
  for \(\nu\)\nb-almost everywhere \(x\in r(U)\). Since we can cover
  \(G\setminus \bigcup_{x\in \base} G^x_x \) by countably many \(U\)
  with \(r(U) \cap s(U) =\emptyset\), we have proved
  Equation~\eqref{equ:dis:thm:rel:state:field}.  Therefore,
  Equation~\eqref{equ:equality-state:dis-state} becomes
  \[
    \phi(f) =\int_{G^{(0)}} \sum_{\gamma\in G^x_x}
    \binpro{\xi_x}{\pi(f(\gamma))\xi_x}\dd \nu(x) =\int_{G^{(0)}}
    \sum_{\gamma\in G^x_x} \phi_x(f(\gamma)\cdot \delta_\gamma)\dd
    \mu(x)
  \]
  which is Equation~\eqref{eq:states-int-disint-reln}.

  Finally, to show \(\A_x\) contained in the centraliser of \(\phi_x\)
  for \(x\in \base\), let \(h\in \Cc(\base)\).  Since
  \(\Contz(\base; \A|_{\base})\) is contained in centraliser of
  \(\phi\), for \(k \in \Contz(\base; \A|_{\base}) \) and
  \(g \in \Cc(G;\A)\), we have
  \begin{equation}\label{equ:centraliser-cont-in-smaller-state}
    \phi(k*g) = \phi(g*k).
  \end{equation}
  Set \(k=h\cdot f \in \Contz(\base;\A|_{\base})\) for some
  \(f\in \Contz(\base; \A|_{\base})\). Plug in this \(k\) in the both
  side of Equation~\eqref{equ:centraliser-cont-in-smaller-state}; and
  use Equation~\eqref{eq:states-int-disint-reln} to see that
  \[
    \int_{G^{(0)}} \biggr( \sum\limits_{\gamma\in G^x_x}
    \phi_x\big((k*g)(\gamma)\cdot \delta_\gamma\big) \biggr) \dd\mu(x)
    = \int_{G^{(0)}} \biggr( \sum\limits_{\gamma\in G^x_x}
    \phi_x\big((g*k)(\gamma)\cdot \delta_\gamma\big) \biggr) \dd
    \mu(x).
  \]
  Using the definition of \(k\) and linearity of \(\phi_x\) in the
  above line, we have
  \[
    \int_{G^{(0)}} \biggr( h(x)\sum\limits_{\gamma\in G^x_x}
    \phi_x\big((f*g)(\gamma)\cdot \delta_\gamma\big) \biggr) \dd\mu(x)
    = \int_{G^{(0)}} \biggr(h(x) \sum\limits_{\gamma\in G^x_x}
    \phi_x\big((g*f)(\gamma)\cdot \delta_\gamma\big) \biggr) \dd
    \mu(x).
  \]
  Since the last equality holds for all \(h\in \Cc(\base)\), we have
  \begin{equation}\label{equ:dis:int:thm:centraliser}
    \sum_{\gamma \in G^x_x} \phi_x\big((f*g)(\gamma) \cdot \delta_\gamma\big) = \sum_{\gamma \in G^x_x} \phi_x\big((g*f)(\gamma) \cdot \delta_\gamma\big)
  \end{equation}
  for \(\mu\)\nb-almost everywhere \(x\in \base\).  Let \(a\in \A_x\),
  \(\gamma \in G^x_x\) and \(b \in \A_{\gamma}\). Since the bundle
  \(p\colon \A \to G\) has enough sections, choose
  \(f\in \Cc(\base; \A|_{\base})\) such that \(f(x) = a\) and
  \(g\in \Cc(V; \A|_{V})\) such that \(g(\gamma) = b\) where \(V\) is
  a bisection containing \(\gamma\). Then
  Equation~\eqref{equ:dis:int:thm:centraliser} becomes
  \( \phi_x\big((a\cdot \delta_x)( b \cdot \delta_\gamma)\big) =
  \phi_x\big((b \cdot \delta_\gamma) (a\cdot \delta_x)\big).\) Since
  \(\{b\cdot \delta_\gamma\}\) is a spanning element of
  \(\Cst(G^x_x;\A(x))\), we have that \(a\cdot \delta_x\) contained in
  the centraliser of \(\phi_x\) for \(x\in \base\). But \(a\in \A_x\)
  was arbitrary, therefore, \(\A_x\) contained in the centraliser of
  \(\phi_x\) for \(x\in \base\).
	
\end{proof}

\begin{proposition}
  \label{prop:integration-of-states-2}
  Assume the same hypothesis as in
  Proposition~\ref{prop:integration-of-states}. Assume that
  \((\mu, \Phi)\) and \((\nu,\Psi)\) are two disintegrations of the
  state~\(\phi\). Then \(\mu=\nu\) and \(\Phi\) and \(\Psi\) are
  equivalent fields of states.
\end{proposition}

\begin{proof}
  Let \(\Phi = \{\phi_x\}_{x \in G^{(0)}}\) and
  \(\Psi = \{\psi_x\}_{x\in G^{(0)}}\). Let \((u_n)_{n\in \N}\) be an
  approximate identity for the \(\Contz(G^{(0)}; \A|_{G^{(0)}})\)
  where \(u_n \in \Cc(\base; \A|_{\base})\). Then by
  Lemma~\ref*{lem:fiberwisw-appr-iden}, \((u_n(x))_{n\in \N}\) is an
  approximate identity for \(\A_x\) where \(x\in G^{(0)}\). For each
  \(f\in \Cc(G^{(0)})\), define a sequence of sections
  \(\tilde{f}_n \in \Cc(G;\A)\) as
  \[
    \tilde{f}_n(\gamma) =
    \begin{cases}
      f(x)u_n(x) & \quad \textup{ if } \gamma = x \in G^{(0)},\\
      0 & \quad \textup{ if } \gamma \notin G^{(0)}.
    \end{cases}
  \]
  Since \((\mu, \Phi)\) and \((\nu, \Psi)\) are two disintegrations of
  \(\phi\), Equation~\eqref{eq:states-int-disint-reln} holds for both
  of them. Hence,
  \[
    \int_{G^{(0)}} \biggr( \sum\limits_{\gamma\in G^x_x}
    \phi_x\big(\tilde{f}_n (\gamma)\cdot \delta_\gamma\big) \biggr)
    \textup{d}\mu(x) = \phi(\tilde{f}_n) = \int_{G^{(0)}} \biggr(
    \sum\limits_{\gamma\in G^x_x} \psi_x\big( \tilde{f}_n(\gamma)\cdot
    \delta_\gamma\big) \biggr) \textup{d}\nu(x).
  \]
  As \(\tilde{f_n}\) is supported on \(\base\), the last equation
  becames
  \[
    \int_{G^{(0)}} \phi_x(\tilde{f_n}(x)) \textup{d}\mu(x)=
    \int_{G^{(0)}}\psi_x(\tilde{f_n}(x)) \textup{d}\nu(x).
  \]
  Using the definiton of \(\tilde{f_n}\) in the above equation, we get
  \[
    \int_{G^{(0)}} f(x)\phi_x(u_n(x)) \textup{d}\mu(x)= \int_{G^{(0)}}
    f(x)\psi_x(u_n(x)) \textup{d}\nu(x).
  \]
  Since \(\phi_x(u_n(x)), \psi_x(u_n(x)) \to 1\) as \(n \to \infty \),
  taking limit as \(n \to \infty \) and using the dominated
  convergence theorem the above equation becames
  \[
    \int_{G^{(0)}} f(x) \textup{d}\mu(x)= \int_{G^{(0)}}
    f(x)\textup{d}\nu(x) \quad \textup{for all } f \in \Cc(\base).
  \]
  Thus \(\mu= \nu\).
	
  Let
  \( \text{Ib}(G) \defeq \{\gamma \in G : s(\gamma) = r(\gamma)\} =
  \bigcup_{x\in \base} G^x_x \) be the bundle of isotropy groups. It
  is known that, \( \text{Ib}(G)\) is a closed subgroupoid of
  \(G\). Let \(f\in \Cc( \text{Ib}(G); \A|_{ \text{Ib}(G)})\) and
  \(g\in \Cc(\base)\). Using vector-valued Tietze extension theorem
  (cf.~\cite[Proposition A.5]{Muhly-Williams2008FellBundle-ME}) we
  extend \(f\) to get a section in \(\Cc(G;\A)\). We shall denote this
  extension by \(f\) it self. Define a section \(k\in \Cc(G; \A)\) as
  \( k(\gamma) = g(r(\gamma)) f(\gamma)\) for \(\gamma \in G\). Since
  \(\mu=\nu\) and \((\mu, \Phi) \) and \((\mu, \Psi)\) are two
  disintegration of \(\phi\), we have
  \[
    \int_{G^{(0)}} \biggr( \sum\limits_{\gamma\in G^x_x} \phi_x(k
    (\gamma)\cdot \delta_\gamma) \biggr) \textup{d}\mu(x) =
    \int_{G^{(0)}} \biggr( \sum\limits_{\gamma\in G^x_x} \psi_x(
    k(\gamma)\cdot \delta_\gamma) \biggr) \textup{d}\mu(x).
  \]
  Using the defintion of \(k\) in the above equation, we have
  \[
    \int_{G^{(0)}} \biggr( \sum\limits_{\gamma\in G^x_x}
    \phi_x\big(g(x)f(\gamma) \cdot \delta_\gamma\big) \biggr)
    \textup{d}\mu(x) = \int_{G^{(0)}} \biggr( \sum\limits_{\gamma\in
      G^x_x} \psi_x\big( g(x)f(\gamma)\cdot \delta_\gamma\big) \biggr)
    \textup{d}\mu(x).
  \]
  Since \(g(x)\) is a complex number, the linearilty of \(\phi_x\) and
  \(\psi_x\) ensure that the above equation became
  \[
    \int_{\base} g(x) \big( \phi_x(f|_{G^x_x}) - \psi_x(f|_{G^x_x})
    \big) \dd \mu(x) = 0.
  \]
  Since the above equality holds for all \(g\in \Cc(\base)\), we have
  \(\phi_x(f|_{G^x_x}) = \psi_x(f|_{G^x_x})\) for \(\mu\)\nb-almost
  everywhere \(x\in \base\), where
  \(f\in \Cc( \text{Ib}(G); \A|_{ \text{Ib}(G)})\). Therefore,
  \([\Phi] = [\Psi]\).
\end{proof}

\begin{proposition}
  \label{thm:disint-states}
  Let \(p\colon \A \to G\) be a Fell bundle over a Hausdorff, locally
  compact, second countable, {\'e}tale groupoid \(G\).  Let \(\mu\) be
  a probability measure on the unit space \(G^{(0)}\) and
  \(\Phi= \{\phi_x\}_{x\in G^{(0)}}\) a \(\mu\)\nb-measurable field of
  states on \(\Cst(G^x_x; \A(x))\) with \(\A_x\) contained in
  centraliser of \(\phi_x\). Then there is a state \(\phi\) on the
  \(\Cst\)\nb-algebra \(\Cst(G;\A)\) with centraliser containing
  \(\Contz(G^{(0)};\A|_{G^{(0)}}) \) satisfying the following relation
  \begin{equation}\label{eq:disint-states}
    \phi(f)
    = \int_{G^{(0)}} \biggr( \sum\limits_{\gamma\in G^x_x} \phi_x(f(\gamma)\cdot
    \delta_\gamma) \biggr) \textup{d}\mu(x)
  \end{equation}
  for all \(f\in \Contc(G; \A).\)
\end{proposition}

\begin{proof}
  The idea of the proof is as follows: let \((\pi_x,\Hils_x,\xi_x)\)
  be the GNS triple corresponding to the state \(\phi_x\) of
  \(\Cst(G_x^x;\A(x))\). Then,
  Theorem~\ref{prop:cyclic-vect-for-ind-rep} tells us that the induced
  representation \(\IndP(\pi_x)\), via the induction correspondence
  \(Y(x)\colon \Cst(G; \A)\to \Cst(G_x^x; \A(x))\), is also cyclic
  with a cyclic unit vector~\(a(x)\). Recall the definition of
  \(a(x)\) from Definition and
  Lemma~\ref{lem:appx-unit-for-ind-corr}. Thus,
  \(\big(\IndP(\pi_x), Y(x)\otimes \Hils_x, a(x)\big)\) is a GNS
  triple corresponding to a state \(\psi_x\) on \(\Cst(G;\A)\).  Since
  \(a(x) \in Y(x)\otimes \Hils_x\) is a unit cyclic vector for
  \(\IndP(\pi_x)\), for \(f\in \Cc(G;\A) \subseteq \Cst(G;\A)\),
  \[
    \psi_x(f) = \binpro{a(x)}{\IndP(\pi_x)(f)( a(x))}
  \]
  defines a state on \(\Cst(G;\A)\). Using
  Lemma~\ref{lem:for-cyclic-vect}(3) for
  \(\phi = \phi_x, L= \pi_x, \Hils = \Hils_x \) and \(\xi= \xi_x\), we
  get \( \psi_x(f) = \phi_x(f|_{G_x^x}) \) for \(f\in
  \Cc(G;\A)\). Since \(\{\phi_x\}_{x\in\base[G]}\) is a
  \(\mu\)\nb-measurable field of states, for a \(f\in \Contc(G;\A)\),
  \(x \mapsto \phi_x(f|_{G_x^x})\) is a \(\mu\)\nb-measurable function
  on \(\base[G]\) and the function is also essentially
  bounded. Therefore, \( x\mapsto \psi_x(f) = \phi_x(f|_{G_x^x}) \) is
  a \(\mu\)\nb-integrable function on \(\base[G]\). Since \(\psi_x\)
  is a state, the above function is essentially bounded by
  \(\norm{f}_{\Cst(G;\A)}\) on \(\base[G]\). Thus defining
  \(\phi(f) \defeq \int_{\base} \psi_x(f) \dd \mu (x) \) is justified;
  this arrangment gives us Equation~\eqref{eq:disint-states}. Since
  \(\mu\) is a probability measure on \(\base\), \(\phi\) is clearly a
  state on \(\Cst(G;\A)\).
	
  To show \(\Contz(\base; \A|_{\base}) \subseteq \) centraliser of
  \(\phi\), let \(f\in \Contz(G^{(0)};\A|_{G^{(0)}}) \) and
  \(g\in \Contc(G;\A)\). Then
  \begin{multline*}
    \psi_x(f*g)= \phi_x\big((f*g)|_{G^x_x}\big) =
    \phi_x\biggr(\sum_{\gamma\in G_x^x} f*g(\gamma) \delta_\gamma
    \biggr) = \sum_{\gamma\in G^x_x}
    \phi_x\big((f*g)(\gamma)\delta_\gamma\big)\\
    = \sum_{\gamma\in G^x_x}
    \phi_x\big(f(x)g(\gamma)\delta_\gamma\big) = \sum_{\gamma\in
      G^x_x} \phi_x\big(f(x)\delta_x \cdot
    g(\gamma)\delta_\gamma\big).
  \end{multline*}
  Since \(\A_x\) is contained in centraliser of \(\phi_x\), the last
  term of the above equation equals
  \[
    \sum_{\gamma\in G^x_x} \phi_x\big(g(\gamma)\delta_\gamma \cdot
    f(x)\delta_x\big) = \sum_{\gamma\in G^x_x}
    \phi_x\big((g*f)(\gamma)\delta_\gamma\big) =
    \phi_x\big((g*f)|_{G^x_x}\big) = \psi_x(g*f).
  \]
  The density of \(\Cc(\base; \A|_{\base}) \) in
  \( \Contz(\base; \A|_{\base})\) implies that
  \( \Contz(\base; \A|_{\base}) \subseteq \) centraliser of
  \(\psi_x\).
	
\end{proof}

\begin{lemma}\label{lem:disint-states}
  Assume the same hypothesis as in
  Proposition~\ref{thm:disint-states}. Let
  \(\Psi=\{\psi_x\}_{x\in \base[G]}\) be a \(\mu\)\nb-measurable field
  of states equivalent to \(\Phi = \{\phi_x\}_{x\in \base}\). Then
  \((\mu,\Phi)\) and \((\mu, \Psi)\) induce the same states on
  \(\Cst(G;\A)\).
\end{lemma}

\begin{proof}
  Since \(\Psi\) and \(\Phi\) are equivalent, we have
  \(\psi_x =\phi_x\) for \(\mu\)\nb-almost everywhere
  \(x\in G^{(0)}\). By Proposition~\ref{thm:disint-states}, let
  \(\psi\) and \(\phi\) be the integrated states corresponding to the
  pairs \((\mu, \Psi)\) and \((\mu, \Phi)\), respectively. Then by
  Equation~\eqref{eq:disint-states}, we have
  \begin{multline*}
    \phi(f) = \int_{G^{(0)}} \biggr( \sum\limits_{\gamma\in G^x_x}
    \phi_x(f(\gamma)\cdot
    \delta_\gamma) \biggr) \textup{d}\mu(x) = \int_{\base} \phi_x(f|_{G^x_x}) \dd \mu(x)\\
    = \int_{\base} \psi_x(f|_{G^x_x}) \dd \mu(x)\ = \int_{G^{(0)}}
    \biggr( \sum\limits_{\gamma\in G^x_x} \psi_x(f(\gamma)\cdot
    \delta_\gamma) \biggr) \textup{d}\mu(x) = \psi(f)
  \end{multline*}
  for \(f \in \Cc(G;\A)\). Since \(\Cc(G;\A)\) is dense in
  \(\Cst(G;\A)\), we have \(\phi = \psi\).
\end{proof}

\begin{theorem}
  \label{prop:bij-betwen-states-and-fields-of-states}
  Let \(p\colon \A \to G\) be a Fell bundle over a Hausdorff, locally
  compact, second countable, {\'e}tale groupoid \(G\).
  \begin{enumerate}
  \item A state \(\psi\) on \(\Cst(G;\A)\), with centraliser
    containing \(\Contz(G^{(0)};\A|_{\base})\), \emph{disintegrates}
    into a pair \((\mu, [\Psi])\) where \(\mu\)\nb-is a probability
    measure; \(\Psi = \{\psi_x\}_{x \in \base[G]}\) is a
    \(\mu\)\nb-measurable field of states on \(\Cst(G^x_x; \A(x))\)
    with \(\A_x \subseteq\) centraliser \(\psi_x\); and \([\Psi]\) is
    the equivalence class of~\(\Psi\). The relation between \(\psi\)
    and the pair \( (\mu, [\Psi])\) is given by
    Equation~\eqref{eq:states-int-disint-reln} where
    \(\Phi\in [\Psi]\).
  \item Conversely, a pair \((\mu, [\Psi])\) consisting of a
    probability measure~\(\mu\); a \(\mu\)\nb-measurable field of
    states \(\Psi = \{\psi_x\}_{x \in \base[G]}\) where \(\psi_x\) is
    a state on~\(\Cst(G^x_x; \A)\) with \(\A_x \subseteq\) centraliser
    \(\psi_x\); and \([\Psi]\) is the equivalence class of~\(\Psi\),
    \emph{integrates} to a state~\(\psi\) on \(\Cst(G;\A)\) whose
    centraliser contains \(\Contz(\base;\A|_{\base[G]})\). The state
    \(\psi\) and the pair \((\mu, [\Psi])\) are related by
    Equation~\eqref{eq:disint-states} for \(\Phi\in [\Psi]\).
  \item Assume that a state \(\psi\) on \(\Cst(G;\A)\) disintegrates
    to \((\mu,[\Psi])\) as in~(1). Then~\(\psi\) is the integrated
    state corresponding to~\((\mu,[\Psi])\).
  \item Conversely, assume that a pair \((\mu, [\Psi])\) integrates to
    the state~\(\psi\) on \(\Cst(G;\A)\) as in~(2) above. Then
    \((\mu,[\Psi])\) is the disintegration of \(\psi\).
  \end{enumerate}
  In short, the processes~(1) and~(2) above are inverses of each
  other.
\end{theorem}

\begin{proof}
  \noindent (1): The proof follows from
  Proposition~\ref{prop:integration-of-states} and
  Proposition~\ref{prop:integration-of-states-2}.
	
  \noindent (2): It follows from Theorem~\ref{thm:disint-states} and
  Lemma~\ref{lem:disint-states}.
	
  \noindent (3): Let a state \(\psi\) on \(\Cst(G;\A)\) disintegrate
  to \((\mu, [\Psi])\), then \(\psi\) and \([\Psi]\) are related by
  the Equation~\eqref{eq:states-int-disint-reln} for
  \(\Phi\in [\Psi]\). Then Lemma~\ref{lem:disint-states} says that
  \(\psi\) is the integrated state of \((\mu, [\Psi])\).
	
  \noindent (4): Let \(\psi\) be the integrated state of a pair
  \((\mu, [\Psi])\), then by Equation~\eqref{eq:disint-states} we have
  \[
    \psi(f) = \int_{G^{(0)}} \biggr( \sum\limits_{\gamma\in G^x_x}
    \psi_x(f(\gamma)\cdot \delta_\gamma) \biggr) \textup{d}\mu(x).
  \]
  Since \(\psi\) is a state on \(\Cst(G;\A)\) by
  Proposition~\ref{prop:integration-of-states}, there is a pair
  \((\mu^{\prime}, \{\psi^{\prime}_x\}_{x\in G^{(0)}})\) satisfying
  Equation~\eqref{eq:states-int-disint-reln}. Then
  Proposition~\ref{prop:integration-of-states-2}, gives us
  \(\mu=\mu^{\prime}\) and
  \(\{\psi^{\prime}_x\}_{x\in G^{(0)}} \in [\Psi]\). Therefore,
  \((\mu, [\Psi])\) is the disintegration of \(\psi\).
\end{proof}

\section{A characterisation of KMS states}
\label{sec:main-result-KMS-state}

\subsection{KMS states}
Let \((A, \mathbb{R}, \tau)\) be a \(\Cst\)-dynamical system. An
element \(a\in A\) is called analytic if the map
\(t\mapsto \tau_t(a)\) is restriction of an analytic function from
\(\mathbb{C}\) to \(A\). By \cite[Proposition
2.5.22]{Bratteli-Robinson1979Oper-alg-Quan-sta-mecha-part-1} says that
the set of all analytic element of \(A\) form a norm dense
\(^*\)\nb-subalgebra of \(A\).
\begin{definition}[{KMS state
    \cite[Definition~5.3.1]{Bratteli-Robinson1981Oper-alg-Quan-sta-mech-part-2}}]
  Let \((A,\mathbb{R}, \tau)\) ba a \(\Cst\)-dynamical system. For
  \(\beta \in \mathbb{R}\), a state \(\phi\) on \(A\) is called
  \(\KMS\) state (or KMS state at inverse temperature \(\beta\)) if
  \(\phi(ab) = \phi(b \tau_{\mathrm{i} \beta} (a))\) for all analytic
  elements \(a, b\) in \(A\).
\end{definition}
By
\cite[Proposition~5.3.3]{Bratteli-Robinson1981Oper-alg-Quan-sta-mech-part-2}
\(\KMS\) states for \(\beta \neq 0\) is \(\tau\) invariant in the
sense \(\phi(\tau_t(a)) = \phi(a)\) for all \(a \in A\) and for all
\(t \in \mathbb{R}\).

Let \(c\colon G\to \R\) be a 1\nb-cocycle, that is, a continuous
groupoid homomorphism. Such a cocycle induces a real dynamical system
\(\sigma \colon \R \to \textup{Aut}(\Cst(G;\A))\) which is defined on
the pre-\(\Cst\)\nb-algebra \(\Contc(G;\A)\) as
\begin{equation}\label{equ:real-dyna}
  \sigma_t(f)(\gamma)= \mathrm{e}^{\mathrm{i}tc(\gamma)}f(\gamma)
\end{equation}
where \(t\in \R, \gamma \in G\) and
\(f \in \Contc(G;\A)
\)~(\cite[Page~233]{Afsar-Sims2021KMS-states-on-Fell-bundle-Cst-alg}).

We fix a notation and mention some observations. Let~\(U\) be a
bisection. For \(x\in G^{(0)}\), let \(u^x\) denote the unique element
in \(r^{-1}(x)\cap U\), and by \(U^x\) we denote the singleton
\(\{u^x\}\); define the singleton set \(U_x= s^{-1}(x)\cap U\)
containing \(u_x\) similarly. The maps \(r(U) \to U, x\mapsto u^x\)
and \(s(U) \to U, x\mapsto u_x\) are homeomorphisms; they are the
inverses of the restrictions of the range and source maps,
respectively, to~\(U\). Let \(T_U\colon r(U)\to s(U)\) be the
homeomorphism \(T_U \colon x \mapsto s(u^x)\). Thus we note that
\(u^x = u_{T_U(x)}\) for \(x\in r(U)\).

\begin{obs}
  \label{obs:relation-betw-isotro}
  It can be easily chicked that for a bisection \(U\) and
  \(x\in r(U)\) the isotropies at \(T_U(x)\) and \(x\) are related as
  \[
    G^{T_U(x)}_{T_U(x)} = (u^x)\inverse G^x_x u^x.
  \]
  %	This is because \((u^x)^{-1} \gamma u^x \in
  % G^{T_U(x)}_{T_U(x)}\) for \(\gamma \in G^x_x\). For the other
  % inclusion if \(\eta \in G^{T_U(x)}_{T_U(x)}\), then
  % \(\eta = (u^x)^{-1} \delta u^x \in (u^x)\inverse G^x_x u^x\) where
  % \(\delta = u^x \eta (u^x)^{-1} \in G^x_x\).
\end{obs}
If \(f\in \Contc(U;\A|_{U})\) and \(g\in \Contc(G; \A)\), then
\begin{eqnarray}\label{eq:kms-thm-comp-1}
  f*g(\gamma) &=&\sum_{\alpha \tau =\gamma} f(\alpha) g(\tau) \nonumber\\
              &=& \begin{cases}
                f(u^{r(\gamma)})g({(u^{r(\gamma)}})^{-1}\gamma) & \quad \textup{ if }  r(\gamma) \in r(U),\\
                0 & \quad \textup{ if } r(\gamma) \notin r(U).
              \end{cases}
\end{eqnarray}
Recall the real dynamics from Equation~\eqref{equ:real-dyna} and a
similar computation as above gives us
\begin{eqnarray}\label{eq:kms-thm-comp-2}
  g*\sigma_{\mi\beta}(f)(\gamma)&=&  \sum_{\alpha \tau = \gamma} g(\alpha) \sigma_{\mi\beta} (f) (\tau) =  \sum_{\alpha \tau =\gamma} g(\alpha)\,
                                    \mathrm{e}^{- \beta c(\tau)} f(\tau)  \nonumber\\
                                &=& \begin{cases}
                                  \e^{- \beta c(u_{s(\gamma)})} g(\gamma u_{s(\gamma)}^{-1})f(u_{s(\gamma)}) & \quad \textup{ if } s(\gamma) \in s(U),\\
                                  0 & \quad \textup{ if } s(\gamma) \not \in s(U).
                                \end{cases}
\end{eqnarray}
If \(\{\phi_x\}_{x\in \base}\) is a \(\mu\)\nb-measurable field of
states, then
\[
  x\mapsto \phi_x \left( g*\sigma_{\mathrm{i}\beta}(f)\right) =
  \phi_x\biggr(\mathrm{e}^{-\beta c(u_x)} \sum_{\gamma\in G_x^x }
  (g*f) (\gamma) \biggr) = \mathrm{e}^{-\beta
    c(u_x)}\phi_x\left((g*f)|_{G_x^x} \right)
\] is a measurable complex valued function on \(\base[G]\)
\emph{supported} in \(s(U)\).

We shall break the proof of Theorem~\ref{thm:KMS-state}, which is the
main result of this section, into succeeding lemmas.  Fix a Fell
bundle \(p\colon \A \to G\) over a Hausdorff, locally compact, second
countable, {\'e}tale groupoid \(G\) and \(\sigma\) a real dynamics on
\(\Cst(G;\A)\) induced by a real-valued \(1\)\nb-cocycle \(c\).
\begin{remark}\label{rem:dis-int-kms-state}
  If \(\phi\) be a KMS states on \((\Cst(G;\A), \R, \sigma)\), then
  the KMS condition forces that the centraliser of \(\phi\) contains
  \(\Contz(\base; \A|_{\base})\). Then by
  Theorem~\ref{prop:bij-betwen-states-and-fields-of-states}(1),
  \(\phi\) disintegrates to a pair \((\mu, [\Phi])\) consisting of a
  probability measure \(\mu\) on \(G^{(0)}\) and a
  \(\mu\)\nb-measurable field \( \Phi =\{\phi_x\}_{x\in G^{(0)}}\) of
  states as in
  Theorem~\ref{prop:bij-betwen-states-and-fields-of-states}. Moreover,
  the \KMS\ state \(\phi\) is given by
  Equation~\eqref{eq:disint-states}.
\end{remark}
	
\begin{lemma}
  \label{lem:kms-state-implies-quasi}
  Let \((\mu, [\Phi])\) be a disintegration of a \KMS\ state \(\phi\)
  on \(\Cst(G;\A)\) as observed in
  Remark~\ref{rem:dis-int-kms-state}. Then the probability measure
  \(\mu\) on \(\base\) is a quasi-invariant with Radon--Nikodym
  derivative \(\mathrm{e}^{-\beta c}\).
\end{lemma}
	
	\begin{proof}
          Fix a bisection \(U\). Let \(a\in \Cc(U;\A|_U)\) and
          \(b\in \Cc(G;\A)\). Since \(\phi\) is a \KMS\ state, we have
          \begin{equation}\label{equ:kms-quasi-inv}
            \phi(a*b)=\phi\big(b*\sigma_{i\beta}(a)\big).
          \end{equation}
          Let \((f_n)_{n\in \N}\) be an approximate identity for
          \(\Cst(G;\A)\) where \(f_n\in \Cc(G;\A)\) for \(n\in \N\)
          (see~Proposition~\ref{prop:exist-app:unit}) and
          \(q \in \Cc(s(U))\). Let \(g_n\) be the square root of
          \(f_n\) for \(n\in \N\). Define the following continuous
          compactly supported sections \(h_n\colon U\to \A \) by
          \(h_n(\gamma) = g_n(\gamma) q(s(\gamma))\).  Let
          \(\Phi = \{\phi_x\}_{x\in \base} \). To compute both sides
          of Equation~\eqref{equ:kms-quasi-inv}, we use
          Equation~\eqref{eq:disint-states}. Using
          Equations~\eqref{eq:kms-thm-comp-1}
          and~\eqref{eq:kms-thm-comp-2}, and substituting
          \(a = h_n, b = h^*_n\) in
          Equation~\eqref{equ:kms-quasi-inv}, we get
          \[
            \int_{r(U)} \phi_x\big(h_n(u^x)h_n^*((u^x)^{-1} )\cdot
            \delta_x\big) \mathrm{d}\mu (x) = \int_{s(U)} \e^{\beta
              c(u_x)} \phi_x\big(h_n^*((u_x)^{-1})h_n(u_x ) \cdot
            \delta_x\big)\mathrm{d}\mu(x).
          \]
          Using the definition of involution in the above equation, we
          get
          \[
            \int_{r(U)} \phi_x\big(h_n(u^x)(h_n(u^x))^*\cdot
            \delta_x\big) \mathrm{d}\mu (x) = \int_{s(U)} \e^{\beta
              c(u_x)} \phi_x\big((h_n((u_x))^*h_n(u_x ) \cdot
            \delta_x\big)\mathrm{d}\mu(x).
          \]
          Putting the value of \(h_n\) in the last equation, we get
          \begin{multline*}
            \int_{r(U)} \phi_x\big(g_n(u^x) q(s(u^x)) \overline{ q(s(u^x))} (g_n((u^x))^* \cdot \delta_x\big) \mathrm{d}\mu (x) \\
            = \int_{s(U)} \e^{\beta c(u_x)}
            \phi_x\big(\overline{q(s(u_x))} (g_n(u_x))^*g_n(u_x )
            q(s(u_x))\cdot \delta_x\big)\mathrm{d}\mu(x).
          \end{multline*}
          As \(s(u_x)=x\) and \(s(u^x) =T_U(x)\), the above equation
          becomes
          \begin{multline*}
            \int_{r(U)} \abs{ q(T_U(x))}^2 \phi_x\big(g_n(u^x) g^*_n((u^x)^{-1} \cdot \delta_x\big) \mathrm{d}\mu (x)\\
            = \int_{s(U)} \abs{q(x)}^2\e^{\beta c(u_x)}
            \phi_x\big(g^*_n((u_x)^{-1})g_n(u_x ) \cdot
            \delta_x\big)\mathrm{d}\mu(x).
          \end{multline*}
          The last equation can be rewritten as
          \[
            \int_{r(U)} \abs{ q(T_U(x))}^2 \phi_x\big(g_n *g^*_n(x)
            \cdot \delta_x\big) \mathrm{d}\mu (x) = \int_{s(U)}
            \abs{q(x)}^2\e^{\beta c(u_x)} \phi_x\big(g^*_n*g_n(x)
            \cdot \delta_x\big)\mathrm{d}\mu(x).
          \]
          Since \(g_n\) is the positive square root of \(f_n\) for
          \(n\in \N\), the above equation is same as
          \[
            \int_{r(U)} \abs{ q(T_U(x))}^2 \phi_x(f_n(x) \cdot \delta_x) \mathrm{d}\mu (x)\\
            = \int_{s(U)} \abs{q(x)}^2\e^{\beta c(u_x)} \phi_x(f_n(x)
            \cdot \delta_x)\mathrm{d}\mu(x).
          \]
          Taking limits on both sides in the above equation, and using
          the dominated convergence theorem and
          \cite[Theorem~3.33]{Murphy-Book}, we have
		
		\[
                  \int_{r(U)} |q(T_U(x))|^2
                  \mathrm{d}\mu(x)=\int_{s(U)} \e^{-\beta c(u_x)}
                  |q(x)|^2 \mathrm{d}\mu(x).
		\]
		Therefore, by Equation~\eqref{eq:quasi-inv-etale} and
                Remark~\ref{remk:quasi-inv-positive-fun}, \(\mu\) is a
                quasi-invariant measure with Radon--Nikodym derivative
                \(\e^{-\beta c(\cdot)}\).
              \end{proof}
	
              \begin{lemma}
		\label{lem:kms-state-implies-trace-cond}
		Let \((\mu, [\Phi])\) be a disintegration of a \KMS\
                state \(\phi\) on \(\Cst(G;\A)\) as observed in
                Remark~\ref{rem:dis-int-kms-state}. Then there is a
                \(\mu\)\nb-conull set~\(K\subseteq \base[G]\) such
                that the field of states
                \(\Phi = \{\phi_x\}_{x\in \base}\) satisfies the
                following condition. Let \(x\in K\). For any
                \(\gamma \in G^x_x\), \(\eta \in G_x\),
                \(a\in \A_\gamma\) and \(\xi \in \A_\eta\) the
                following equation holds:
		\begin{equation*}
                  \phi_{s(\eta)} \big(a \xi^* \xi \cdot
                  \delta_{\gamma}\big) = \phi_{r(\eta)}\big(\xi   a \xi^* \cdot
                  \delta_{\eta \gamma \eta\inverse}\big).
		\end{equation*}
              \end{lemma}
	
	\begin{proof}
          Let \(U, V \subseteq G\) be bisections such that
          \(s(V) \subseteq s(U)\). Choose \(f \in \Cc(G_x;\A|_{G_x})\)
          with \(\supp(f) \in U\) and \(g\in \Cc(V;\A|_V)\). The KMS
          condition on \(\phi\) implies
          \begin{equation}\label{Equ:A}
            \phi\big(g*(f*g^*) \big) = \phi\big((f*g^*)*\sigma_{\mi\beta}(g)\big).
          \end{equation}
          We shall compute both sides of Equation~\eqref{Equ:A} for
          the functions~\(f,g\) above; and obtain an integral that
          shall allow us to conclude the desired result. Let us
          compute the left hand side of this equation
          first. Equation~\eqref{eq:disint-states} says that this left
          hand side is given by
          \begin{equation}\label{eq:A-1}
            \phi\big(g*(f*g^*)\big) = \int_{G^{(0)}} \sum_{w \in
              G^z_z}\phi_z\big(g*(f*g^*)(w)\cdot \delta_w\big) \dd{\mu}(z).
          \end{equation}
		
          \noindent Now we note that \(g*(f*g^*)\) is supported on the
          bisection \(V\cdot (U \cdot V\inverse)\). Thus the integral
          above is to be considered on~\(r(V)\).  Further, as the
          integral above is considered on~\(r(V)\), let \(z\in r(V)\).
          The integrant sum above is considered for the arrows
          \(w\in G^z_z \cap ( V\cdot (U \cdot V\inverse))\); what are
          these arrows? As~\(U, V\) are bisections with
          \(s(V)\subseteq s(U)\), the reader can take it as a small
          exercise that for given \(z\in r(V)\), these arrows are
          exactly \(w = v^z u_{T_V(z)} (v^z)\inverse\). A sketch of
          why this happens is that one should starts with
          \(v^z\in V\), then notice that the unique arrow
          \(\eta\in U\) that is composable with \(v^z\) is
          \(u_{T_V(z)}\). Finally, note that the unique arrow
          in~\(V\inverse\) and starting at~\(z\) that is composable
          with~\(u_{T_V(z)}\) is~\((v^z)\inverse\). Thus the required
          arrow~\(w\) in the isotropy~\(G_z^z\)
          is~\(v^z u_{T_V(z)} (v^z)\inverse\). See the
          footnote~\footnote{At this point, we not that
            \(u_{T_V(z)}\in G_z^z\) as it is conjugated
            by~\(v^z\).}\label{footnote:iso-cond-1}.  Now the last
          term of above equation becomes
          \[
            \int_{r(V)} \phi_z\biggr(\big(g (v^z) f( u_{T_V(z)})
            g^*((v^z)\inverse))\big) \cdot
            \delta_{v^zu_{T_V(z)}(v^z)^{-1}}\biggr) \dd{\mu}(z).
          \]
          Uising the definition~\(g^*((v^z)\inverse) \defeq g(v^z)^*\)
          in the \(^*\)-algebra~\(\Cc(G;\A)\), the last term can be
          written as
          \[
            \int_{r(V)} \phi_z\biggr(\big(g (v^z) f( u_{T_V(z)})
            g(v^z)^* \big) \cdot
            \delta_{v^zu_{T_V(z)}(v^z)^{-1}}\biggr) \dd{\mu}(z).
          \]
          The definitions of \(v^z\) and \(v_{T_V(z)}\) gives us
          \(v^z =v_{T_V(z)}\) for \(z\in r(V)\). Also, the index~\(z\)
          of~\(\phi_z\) above is~\(z=r(v^z) = r(v_{T_V(z)})\). Then
          above term can be rewritten as
          \[
            \int_{r(V)} \phi_{r(v_{T_V(z)})} \biggr(\big(g(v_{T_V(z)})
            f(u_{T_V(z)})g(v_{T_V(z)})^* \big)\cdot
            \delta_{(v_{T_V(z)} u_{T_V(z)})(v_{T_V(z)})^{-1}}\biggr)
            \dd{\mu}(z).
          \]
		
          \noindent Now we want to change the variable
          \(T_V(z) \mapsto
          z\). Lemma~\ref{lem:kms-state-implies-quasi} ensures that
          \(\mu\) is a quasi-invariant measure with Radon--Nikodym
          derivative \(\e^{-\beta c}\). We apply the quasi-invariance
          of~\(\mu\) using Equation~\ref{eq:quasi-inv-etale}: this
          equation suggests to use the homeomorphism
          \(T_V \colon r(V) \to s(V)\) to change the variable
          \(T_V(z) \mapsto T_V\inverse (T_V(z)) = z\)---which changes
          the integral over~\(r(V)\) to~\(s(V)\) and introduces the
          modular function~\(\e^{-\beta c(\cdot)}\). This change of
          variable transforms the last term to
          \[
            \int_{s(V)} \e^{-\beta c(v_z)}\phi_{r(v_z)} \biggr(
            \big(g(v_z)f(u_z)g(v_z)^* \big) \cdot \delta_{v_z u_z
              (v_z)^{-1}}\biggr) \dd{\mu}(z).
          \]
          Thus, Equation~\eqref{eq:A-1} can be rewritten as
          \begin{equation}\label{Equ:left-tr}
            \phi\big(g *(f*g^*)\big) = 
            \int_{s(V)} \e^{-\beta c(v_z)} \phi_{r(v_z)} \biggr( \big(g(v_z)f(u_z)g(v_z)^* \big) \cdot \delta_{v_zu_z(v_z)^{-1}}\biggr) \dd{\mu}(z).
          \end{equation} 
		
          \noindent See the footnote~\footnote{Continuing the
            observation of the last footnote on
            page~\pageref{footnote:iso-cond-1}, notice that
            \(u_{z}\in G_{T_V(z)}^{T_V(z)}\) as it is conjugated
            by~\(v_z\).}.\label{footnote:iso-cond-2}
		
          Next, we compute the right hand side of
          Equation~\eqref{Equ:A}. Like last computation, apply
          Equation~\eqref{eq:disint-states} to see that the this right
          hand side is given by
          \begin{equation}\label{eq:com-rhs-kms-st}
            \phi\big(f*g^**\sigma_{\mi\beta}(g)\big) = \int_{G^{(0)}}  \sum_{w \in G^y_y}\phi_y\biggr(\big(f*g^**\sigma_{\mi\beta}(g)\big)(w)\cdot \delta_w\biggr) \dd{\mu}(y).
          \end{equation}
          \noindent First of all, note that since
          \(\sigma_{\mi \beta}(g)(\gamma) = \e^{-\beta c(\gamma)}
          g(\gamma)\), for \(\gamma\in G\), \(\sigma_{\mi\beta}(g)\)
          has same support as that of~\(g\). Therefore, the function
          \(f*\big(g^**\sigma_{\mi\beta} (g)\big)\) is supported in
          \( U\cdot (V\inverse \cdot V) = U\cdot s(V) = V\) as
          \(s(V)\subseteq s(U)\). Thus the integral in
          Equation~\eqref{eq:A-1} is to be considered on~\(s(V)\).  As
          \(g^**\sigma_{\mi\beta}(g)\) is supported on~\(s(V)\), and
          \(f\) is supported on the bisection~\(U\), for
          \(y\in s(V)\), we can write
          \( f*g^**\sigma_{\mi\beta}(g)(v_y) = f(u_y)
          \big(g^**\sigma_{\mi\beta}(g)\big)(y)\) where \(v_y\in V\)
          is the unique arrow with source~\(y\), and similar is the
          meaning of~\(u_y\).  Thus, the right hand side of
          Equation~\eqref{eq:com-rhs-kms-st} becomes
          \begin{equation*}
            \int_{s(V)} \phi_y \biggr( f(u_y) \big(g^**\sigma_{\mi\beta}(g)\big)(y) \cdot \delta_{u_y} \biggr) \dd{\mu}(y).
          \end{equation*}
          Compute \(g^**\sigma_{\mi\beta}(g)\) using the formula in
          Equation~\eqref{eq:kms-thm-comp-2}. Then the last term
          becomes:
          \[
            \int_{s(V)} \e^{-\beta
              c(v_y)}\phi_y\big(f(u_y)(g^*((v_y)^{-1}) g(v_y)\cdot
            \delta_{u_y}\big) \dd{\mu}(y).
          \]
		
          \noindent Finally, notice
          that~\(g^*(v_y\inverse) \defeq g(v_y)^*\) in the
          \(^*\)-algebra~\(\Cc(G;\A)\); apply this observation to the
          last term and this simplifies
          Equation~\eqref{eq:com-rhs-kms-st} as
          \begin{equation}\label{Equ:right-tr}
            \phi\big(f*g^** \sigma_{\mi\beta}(g) \big) = \int_{s(V)} \e^{-\beta c(v_y)}\phi_y\big(f(u_y)g(v_y)^* g(v_y)\cdot \delta_{u_y}\big) \dd{\mu}(y).
          \end{equation}
		
          \noindent Substituting the right hand sides of
          Equations~\eqref{Equ:left-tr} and~\eqref{Equ:right-tr} in
          Equation~\eqref{Equ:A}, we get
          \begin{multline*}
            \int_{s(V)}  \e^{-\beta c(v_y)} \phi_{r(v_y)} \biggr( \big(g(v_y)f(u_y)g(v_y)^* \big) \cdot \delta_{v_yu_y(v_y)^{-1}}\biggr)  \dd{\mu}(y)\\
            = \int_{s(V)} \e^{-\beta c(v_y)}\phi_y\big(f(u_y) g(v_y)^*
            g(v_y)\cdot \delta_{u_y}\big) \dd{\mu}(y).
          \end{multline*}
		
          \noindent Since the choice of the bisection \(V\) was
          arbitrary, we have
		
		\begin{equation}\label{eq:last-eq-for-trace-cond}
                  \phi_{r(v_y)} \biggr( \big(g(v_y)f(u_y)g(v_y)^* \big) \cdot \delta_{v_yu_y(v_y)^{-1}}\biggr)\\
                  = \phi_y\big(f(u_y) g(v_y)^* g(v_y)\cdot \delta_{u_y}\big)
		\end{equation}
		for \(\mu\)\nb-almost everywhere \(y\in \base\). Let
                \(K\subseteq \base[G]\) be the \(\mu\)\nb-conull
                subset on which the above equality holds.
		
		Before proving the final claim of the lemma involving
                choices of arrows \(\gamma\in G_x^x\) and
                \(\eta\in G_x\) and their coefficients~\(a,b\) we make
                some observations. Note that
                Equation~\eqref{eq:last-eq-for-trace-cond} holds for
                all \(v_y\in V\) and \(u_y\in U\) such that the
                composite arrow~\(v_yu_y(v_y)^{-1}\) is
                well-defined. As in the earlier footnotes on
                pages~\pageref{footnote:iso-cond-1}
                and~\pageref{footnote:iso-cond-2}, the well-defineness
                of this arrow implies that~\(u_y\) is in the
                isotropy~\(G_{y}^y\) for \(y\in r(V)\). Also,
                \(v_y\in G_y\) for \(y\in V\).

		Finally, let \(x\in K\). As in the statement of lemma,
                suppose \(\gamma \in G^x_x, a\in \A_\gamma\),
                \(\eta \in G_x\) and \(\xi\in \A_\eta\) are
                given. Let~\(U\) be a bisection containing~\(\gamma\)
                and~\(V\) be a bisection containing~\(\eta\) such that
                \(s(V) \subseteq s(U)\). Since~\(p\colon \A \to G\)
                has enough sections, choose
                \(f \in \Cc(G_x;\A|_{G_x})\) with \(\supp(f) \in U\)
                and \(g\in \Cc(V;\A|_V)\) with \(f(\gamma) = a,\) and
                \(g(\eta) = \xi.\) The observation in the last
                paragraph implies that
                Equation~\eqref{eq:last-eq-for-trace-cond} holds for
                current~\(f(\gamma) = a\) and~\(g(\eta) = \xi\) as
                well. This is exactly what we wanted in the lemma.
		
              \end{proof}
	
              \begin{proposition}
		\label{thm:conds-implies-kms-state}
		Let \(p\colon \A \to G\) be a Fell bundle over a
                Hausdorff, locally compact, second countable,
                {\'e}tale groupoid \(G\). Let \(\sigma\) be a real
                dynamics on \(\Cst(G;\A)\) induced by a real-valued
                \(1\)\nb-cocycle \(c\). Let \(\beta\in \R\). Let
                \((\mu, [\Phi])\) be a pair consisting of a
                probability measure \(\mu\) on \(G^{(0)}\) and a
                \(\mu\)\nb-measurable field
                \( \Phi =\{\phi_x\}_{x\in G^{(0)}}\) of state
                \(\phi_x\) on \(\Cst(G^x_x;\A(x))\) as in
                Theorem~\ref{prop:bij-betwen-states-and-fields-of-states}(2). Then
                the integrated state \(\phi\) corresponding to
                \((\mu, [\Phi])\) is a \KMS\ states on
                \((\Cst(G;\A), \R, \sigma)\) if the following holds:
		\begin{enumerate}
                \item\label{prop:int-kms-cond-1} \(\mu\) is
                  quasi-invariant with Radon--Nikodym derivative
                  \(\mathrm{e}^{-\beta c}\).
                \item \label{prop:int-kms-cond-2} There is a
                  \(\mu\)\nb-conull set~\(K\subseteq \base[G]\) such
                  that the field of states
                  \(\Phi = \{\phi_x\}_{x\in \base}\) satisfies the
                  following condition. Let \(x\in K\). For any
                  \(\gamma \in G^x_x\), \(\eta \in G_x\),
                  \(a\in \A_\gamma\) and \(\xi \in \A_\eta\) the
                  following equation holds:
                  \begin{equation*}
                    \phi_{s(\eta)} \big(a \xi^* \xi \cdot
                    \delta_{\gamma}\big) = \phi_{r(\eta)}\big(\xi   a \xi^* \cdot
                    \delta_{\eta \gamma \eta\inverse}\big).
                  \end{equation*}
		\end{enumerate}
              \end{proposition}
	
	\begin{proof}
          To prove that \(\phi\) is a \KMS\ state on \(\Cst(G;\A)\)
          for \(\sigma\), it is sufficient to show that
          \[
            \phi(f*g)=\phi\big(g*\sigma_{\mi\beta}(f)\big)
          \]
          for \(f,g\in \Contc(G;\A)\). Since the bisections constitute
          a basis for the topology of~\(G\) and \(\phi\) is linear, we
          may choose \(f\in \Contc(U;\A|_{U})\) where \(U\subseteq G\)
          is a bisection and \(g \in \Cc(G;\A)\).  The definition of
          the state~\(\phi\), Equation~\eqref{eq:disint-states},
          combined with the Computation~\eqref{eq:kms-thm-comp-1} says
          that
          \begin{equation}\label{Equ:kms-LHS}
            \phi(f*g)=\int_{r(U)} \biggr( \sum_{\gamma\in G^x_x} \phi_x\big(f(u^x)g((u^x)^{-1}\gamma)\cdot \delta_\gamma\big)\biggr) \mathrm{d} \mu(x).
          \end{equation}
          On the other hand, using
          Computation~\eqref{eq:kms-thm-comp-2}, we see that
          \begin{multline*}
            \phi\big(g*\sigma_{\mi\beta}(f)\big)
            =\int_{s(U)}\e^{-\beta c(u_x)} \biggr( \sum_{\gamma\in
              G^x_x}\phi_x\big(g(\gamma(u_x)^{-1})f(u_x)\cdot
            \delta_\gamma\big)\biggr) \mathrm{d}\mu(x)\\ = \int_{s(U)}
            \e^{-\beta c(u_x)}\phi_x\big( (g*f)|_{G_x^x}\big)
            \dd\mu(x).
          \end{multline*}
          Now \(\mu\) is quasi-invariant with the Radon--Nikodym
          derivative \(\e^{-\beta c(\cdot)}\). Therefore,
          Equation~\eqref{eq:quasi-inv-etale} implies that the last
          term is
          \begin{multline*}
            \int_{r(U)} \phi_{T_U(x)}\left(
              (g*f)|_{G_{T_U(x)}^{T_U(x)}}\right) \dd\mu(x) \\ =
            \int_{r(U)} \biggr( \sum_{\gamma\in
              G_{T_U(x)}^{T_U(x)}}\phi_{T_U(x)}\big(g(\gamma
            (u_{T_U(x)})^{-1})f(u_{T_U(x)})\cdot
            \delta_\gamma\big)\biggr) \mathrm{d}\mu(x).
          \end{multline*}
          Since \(s(u_{T_U(x)}) = T_U(x)\) and \(u^x\) is the unique
          arrow in the bisection \(U\) with source \(T_U(x)\), we have
          \(u_{T_U(x)} = u^x\). Therefore, the previous term can be
          rewritten as
          \begin{equation}
            \label{equ:prov-kms-state-eq-1}
            \int_{r(U)} \biggr(
            \sum_{\gamma\in
              G_{T_U(x)}^{T_U(x)}}\phi_{T_U(x)}\big(g(\gamma (u^x)^{-1})f(u^x)\cdot
            \delta_\gamma \big) \biggr)
            \mathrm{d}\mu(x).
          \end{equation}
          From the Observation~\ref{obs:relation-betw-isotro}, the
          isotropies at \(T_U(x)\) and \(x\) are related as
          \[
            G^{T_U(x)}_{T_U(x)}(u^x)^{-1}=(u^x)\inverse G^x_x.
          \]
          With this observation, we replace \(\gamma (u^x)\inverse\),
          where \(\gamma\in G^{T_U(x)}_{T_U(x)}\), by
          \((u^x)\inverse \zeta\) where \(\zeta\in G_x^x\).  Thus
          Equation~\eqref{equ:prov-kms-state-eq-1} transforms to
          \begin{equation}\label{Equ:kms-RHS}
            \int_{r(U)} \sum_{\zeta\in
              G^x_x}\phi_{T_U(x)}\big(g((u^x)^{-1}\zeta)f(u^x)\cdot
            \delta_{(u^x)^{-1}\zeta u^x } \big)  \, \mathrm{d}\mu(x).
          \end{equation}
          Let us compute the summand in the last integral. Since
          \(\A_{u^x}\) is a full right Hilbert
          \(\A_{T_U(x)}\)\nb-module in the Fell bundle
          \(\A\),~\cite[Lemma~7.3]{Lance1995Hilbert-modules} assures
          the existence of an approximate identity
          \(\big(\sum_{i=1}^{n_\alpha}
          \binpro{\eta^i_\alpha}{\eta^i_\alpha}\big)_\alpha\) of
          \(\A_{T_U(x)}\) where \(\eta^i_\alpha \in \A_{u^x}\),
          \(n_\alpha \in \N\). Using
          Lemma~\ref{lem:appx-unit-for-full-Hilbert-mod}, for
          \(\zeta \in G^x_x\), we have
          \begin{multline*}
            \phi_{T_U(x)}\big(g((u^x)^{-1}\zeta)f(u^x)\cdot
            \delta_{(u^x)^{-1}\zeta u^x}\big) \\
            = \phi_{T_U(x)}\biggr(g((u^x)^{-1}\zeta)f(u^x)
            \lim_{\alpha}\sum_{i=1}^{n_\alpha}
            \binpro{\eta^i_\alpha}{\eta^i_\alpha}\cdot
            \delta_{(u^x)^{-1}\zeta u^x}\biggr).
          \end{multline*}
          Using the defintion of inner product and the continuity of
          \(\phi_x\), the last term became
          \[
            \lim_{\alpha}\sum_{i=1}^{n_\alpha}\phi_{T_U(x)}\big(g((u^x)^{-1}\zeta)f(u^x)
            (\eta^{i}_\alpha)^* \eta^i_\alpha \cdot
            \delta_{(u^x)^{-1}\zeta u^x}\big).
          \]
          Since \(T_U(x) = s(u^x)\) and \(r(u^x) = x \), using the
          Condition~\eqref{prop:int-kms-cond-2} of the current
          proposition for
          \(a = g((u^x)^{-1}\zeta)f(u^x), \xi = \eta^{i}_\alpha,
          \gamma = (u^x)^{-1}\zeta u^x\) and \(\eta = u^x\), we see
          that the above term is the same as
          \begin{multline*}
            \lim_{\alpha}\sum_{i=1}^{n_\alpha}\phi_{r(u^x)}\big(\eta^i_\alpha
            g((u^x)^{-1}\zeta) f(u^x) (\eta^i_\alpha)^* \cdot
            \delta_{u^x(u^x)^{-1} \zeta u^x(u^x)^{-1}}\big)\\ =
            \lim_{\alpha}\sum_{i=1}^{n_\alpha}\phi_{x}\big(\eta^i_\alpha
            g((u^x)^{-1}\zeta) f(u^x) (\eta^i_\alpha)^* \cdot
            \delta_{\zeta}\big)
          \end{multline*}
          Since
          \(p\big(f(u^x)(\eta^i_\alpha)^*\big) =u^x(u^x)^{-1} = r(u^x)
          = x\), we have \(f(u^x)(\eta^i_\alpha)^* \in \A_x\), which
          is contained in the centraliser of \(\phi_x\). Therefore,
          the last term is equal to
          \begin{multline*}
            \lim_{\alpha}\sum_{i=1}^{n_\alpha}\phi_{x}\big(f(u^x)
            (\eta^i_\alpha)^* \eta^i_\alpha g((u^x)^{-1}\zeta) \cdot
            \delta_{\zeta}\big) = \lim_{\alpha}\phi_{x}\biggr(f(u^x)
            \sum_{i=1}^{n_\alpha}
            \binpro{\eta^i_\alpha}{\eta^i_\alpha}g((u^x)^{-1}\zeta)
            \cdot \delta_{\zeta}\biggr)\\ =
            \phi_x\big(f(u^x)g((u^x)^{-1}\zeta) \cdot
            \delta_{\zeta}\big).
          \end{multline*}
          Replacing the integrand sum in Equation~\eqref{Equ:kms-RHS}
          by the last term above, we get,
          \[
            \phi\big(g*\sigma_{i\beta}(f)\big)=\int_{r(U)} \biggr(
            \sum_{\zeta\in G^x_x} \phi_x\big(f(u^x)g((u^x)^{-1}\zeta)
            \cdot \delta_{\zeta}\big) \biggr) \mathrm{d}\mu(x),
          \]
          which is equal to Equation~\eqref{Equ:kms-LHS}. Therefore,
          \(\phi\) is a \KMS\ state.
		
	\end{proof}

	\begin{theorem}
          \label{thm:KMS-state}
          Let \(p\colon \A \to G\) be a Fell bundle over a Hausdorff,
          locally compact, second countable, {\'e}tale groupoid
          \(G\). Let \(\sigma\) be a real dynamics on \(\Cst(G;\A)\)
          induced by a real-valued \(1\)\nb-cocycle \(c\). Let
          \(\beta\in \R\). Then there is a one-to-one correspondence
          between the \KMS\ states on \(\Cst(G;\A)\) and a pair
          \((\mu, [\Phi])\) consisting of a probability measure
          \(\mu\) on \(G^{(0)}\) and a \(\mu\)\nb-measurable field
          \( \Phi =\{\phi_x\}_{x\in G^{(0)}}\) of state \(\phi_x\) on
          \(\Cst(G^x_x;\A(x))\) with \(\A_x\) contained in centraliser
          \((\phi_x)\) such that the following holds:
          \begin{enumerate}
          \item\label{Cond:I} \(\mu\) is quasi-invariant with
            Radon--Nikodym derivative \(\mathrm{e}^{-\beta c}\).
          \item \label{Cond:II} There is a \(\mu\)\nb-conull
            set~\(K\subseteq \base[G]\) such that the field of states
            \(\Phi = \{\phi_x\}_{x\in \base}\) satisfies the following
            condition. Let \(x\in K\). For any \(\gamma \in G^x_x\),
            \(\eta \in G_x\), \(a\in \A_\gamma\) and
            \(\xi \in \A_\eta\) the following equation holds:
            \begin{equation*}
              \phi_{s(\eta)} \big(a \xi^* \xi \cdot
              \delta_{\gamma}\big) = \phi_{r(\eta)}\big(\xi   a \xi^* \cdot
              \delta_{\eta \gamma \eta\inverse}\big).
            \end{equation*}
          \end{enumerate}
	\end{theorem}
	
	\begin{proof}
          Suppose \(\phi\) is a \KMS\ state on \(\Cst(G;\A)\) for the
          real dynamics \(\sigma\). Remark~\ref{rem:dis-int-kms-state}
          says that there is a pair \((\mu, [\Phi])\), consisting a
          probability measure \(\mu\) on \(\base\) and the field of
          states \(\Phi= \{\phi_x\}_{x\in \base}\) is
          \(\mu\)\nb-measurable with \(\A_x\subseteq \) centraliser of
          \(\phi_x\) for \(x\in
          \base\). Lemmas~\ref{lem:kms-state-implies-quasi}
          and~\ref{lem:kms-state-implies-trace-cond} ensure that, the
          pair \((\mu, [\Phi])\) satsfies Conditions~\ref{Cond:I}
          and~\ref{Cond:II} of current theorem.
		
          The converse part follows from
          Proposition~\ref{thm:conds-implies-kms-state}.
	\end{proof}

		\begin{remark}
                  Condition~\ref{Cond:II} of
                  Theorem~\ref{thm:KMS-state} implies that each
                  \(\phi_x\) is a trace on the \(\Cst\)\nb-algebra
                  \(\Cst(G^x_x;\A(x))\). The reason is, if
                  \(\gamma, \eta \in G^x_x\), then the
                  Condition~\ref{Cond:II} become
                  \begin{equation*}
                    \phi_{x} \big((a\cdot \delta_\gamma) (\xi \cdot \delta_\eta)^* (\xi \cdot
                    \delta_{\eta})\big) = \phi_{x}\big((\xi\cdot \delta_\eta) (a \cdot \delta_\gamma )(\xi \cdot
                    \delta_{\eta})^*\big) \quad \mu \text{ almost everywhere } x\in \base,
                  \end{equation*}
                  for \(a\in \A_\gamma\) and \(\xi \in \A_\eta\);
                  which is trace condition
                  (see~\cite[Proposition~8.1.1]{Kadison-Ringrose1997-Vol-II}). Therefore,
                  it makes sense to call Condition~\ref{Cond:II} in
                  Theorem~\ref{thm:KMS-state} ``the trace condition
                  for Fell bundle''.
		\end{remark}
		
		\noindent Applying Theorem~\ref{thm:KMS-state} for
                inverse temperature \( \beta =0\), we get the
                following corollary:
		\begin{corollary}
                  Let \(p\colon \A \to G\) be a Fell bundle over a
                  Hausdorff, locally compact, second countable,
                  {\'e}tale groupoid \(G\). Let \(\sigma\) be a real
                  dynamics on \(\Cst(G;\A)\) induced by a real-valued
                  \(1\)\nb-cocycle \(c\). Let \(\beta\in \R\). Then
                  there is a one-to-one correspondence between the
                  \(\sigma\)\nb-invariant tracial states on
                  \(\Cst(G;\A)\) and a pair \((\mu, [\Phi])\)
                  consisting of a probability measure \(\mu\) on
                  \(G^{(0)}\) and a \(\mu\)\nb-measurable field
                  \( \Phi =\{\phi_x\}_{x\in G^{(0)}}\) of state
                  \(\phi_x\) on \(\Cst(G^x_x;\A(x))\) with \(\A_x\)
                  contained in centraliser \((\phi_x)\) such that the
                  following holds:
                  \begin{enumerate}
                  \item \(\mu\) is an invariant measure.
                  \item There is a \(\mu\)\nb-conull
                    set~\(K\subseteq \base[G]\) such that the field of
                    states \(\Phi = \{\phi_x\}_{x\in \base}\)
                    satisfies the following condition. Let \(x\in
                    K\). For any \(\gamma \in G^x_x\),
                    \(\eta \in G_x\), \(a\in \A_\gamma\) and
                    \(\xi \in \A_\eta\) the following equation holds:
                    \begin{equation*}
                      \phi_{s(\eta)} \big(a \xi^* \xi \cdot
                      \delta_{\gamma}\big) = \phi_{r(\eta)}\big(\xi   a \xi^* \cdot
                      \delta_{\eta \gamma \eta\inverse}\big).
                    \end{equation*}
                  \end{enumerate}
		\end{corollary}

                \section{KMS states on the groupoid (twisted) crossed
                  product
                  \(\Cst\)-algebras}\label{sec:gpd-crossed:product}
  
                In this section, we shall give a quick summary of
                groupoid crossed product; for detailed treatment, the
                reader can
                see~\cite{Muhly-Williams2008Renaults-equivalence-thm-for-gpd-cros-prd}
                and~\cite{Goehle2009phd-thesis}.  The groupoid crossed
                product \(\Cst\)\nb-algebra can be identified as a
                \(\Cst\)\nb-algebra of a Fell bundle over the same
                groupoid. We shall apply Theorem~\ref{thm:KMS-state}
                and characterise the invariant traces and KMS states
                on the groupoid crossed product.
  
                \begin{definition}
                  [{\cite[Definition~3.46]{Goehle2009phd-thesis}}]
                  \label{def:gpd-crossed-prd}
                  A groupoid dynamical system is a triple
                  \((A, G, \alpha)\), where \(A\) is a
                  \(\textup{C}_0(G^{(0)})\)\nb-algebra and \(\alpha\)
                  is an action of \(G\) on \(A\), that is,
                  \(\alpha =\{\alpha_\gamma\}_{\gamma \in G}\) such
                  that
                  \begin{enumerate}[leftmargin=*]
                  \item for each
                    \(\gamma \in G, \alpha_\gamma \colon A(s(\gamma))
                    \to A(r(\gamma))\) is a \(^*\)\nb-isomorphism;
                  \item
                    \(\alpha_{\gamma \eta} = \alpha_\gamma \circ
                    \alpha_\eta\) for \((\gamma, \eta) \in G^{(2)}\);
                  \item \((\gamma, a) \mapsto \alpha_\gamma(a)\) is
                    continuous from
                    \(G\times_{s, \base, p}\mathcal{A} \to
                    \mathcal{A}\), where
                    \(p\colon \mathcal{A} \to \base \) is the
                    associated upper semicontinuous bundle of
                    \(\Cst\)-algebra corresponding to the
                    \(\Contz(\base)\)\nb-algebra \(A\). And
                    \(G\times_{s, \base, p}\mathcal{A}\) is the fibre
                    product of
                    \(\{(\gamma, a) \in G \times \A : s(\gamma) =
                    p(a)\}\).
                  \end{enumerate}
                \end{definition}
                \noindent Fix a groupoid dynamical system
                \((A,G,\alpha)\). Consider the pullback bundle
                \(q\colon r^*\A \to G\) along the range map \(r\),
                that is,
                \( r^*\A \defeq \big\{(\gamma, a) \in G \times \A :
                r(\gamma) = p(a)\big\}\) and \(q(\gamma, a)
                =\gamma\). The set \(r^*\A\) is given subspace
                topology of \(G \times \A\). Proposition~3.24
                of~\cite{Goehle2009phd-thesis} says that
                \(q\colon r^*\A \to G\) is an upper semicontinuous
                bundle of \(\Cst\)\nb-algebras.  Assume \(G\) is an
                \etale\ groupoid. Then
                by~\cite[Proposition~3.54]{Goehle2009phd-thesis}
                \(\Cc(G;r^*\mathcal{A})\) is a \(^*\)\nb-algebra with
                the following operations:
                \[
                  f*g(\gamma) \defeq \sum_{\eta \in
                    G^{r(\gamma)}}f(\eta)\alpha_\eta\big(g(\eta^{-1}\gamma)\big)~
                  \text{ and } f^*(\gamma) \defeq
                  \alpha_\gamma\big(f(\gamma^{-1})^*\big).
                \]

                \begin{proposition}[{\cite[Examples~2.1 and
                    2.8]{Muhly-Williams2008FellBundle-ME}}]
                  \label{rel:cross:pro-Fell:bundle}
                  Let \((A,G,\alpha)\) be a groupoid dynamical system.
                  \begin{enumerate}
                  \item Then \(r^*\mathcal{A}\) is a Fell bundle over
                    \(G\) and
                    \(A\rtimes_\alpha G \iso \Cst(G;r^*\mathcal{A})\).
                  \item
                    \(A(x)\rtimes_{\alpha|_{G^x_x}} G^x_x \iso
                    \Cst(G^x_x; r^*\mathcal{A}|_{G^x_x})\) for
                    \(x\in G^{(0)}\).
                  \end{enumerate}
                \end{proposition}
  
                \begin{proof}
                  \noindent (1): Let \((A,G,\alpha)\) be a groupoid
                  dynamical system. Then \(r^*\A\) is an upper
                  semicontinuous \(\Cst\)\nb-bundle over \(G\). The
                  multiplication and involution on \(r^*\A\) is
                  defined with the help of the action \(\alpha\). Let
                  \((r^*\A)^{(2)} = \big\{((\gamma, a),(\eta, b)) \in
                  r^*\A \times r^*\A : (\gamma, \eta) \in
                  G^{(2)}\big\}.\) Then the multiplication and
                  involution are defined by
                  \[
                    (\gamma, a)(\eta, b) \defeq \big( \gamma \eta
                    ,a\alpha_\gamma(b)\big) \text{ and } (\gamma, a)^*
                    \defeq
                    \big(\gamma^{-1},\alpha_{\gamma^{-1}}(a^*)\big).
                  \]
                  Those operations makes \(r^*\A\) is a Fell bundle
                  over \(G\) (see~\cite[Example
                  7]{Muhly2001Bundle-over-gpd}). Therefore, we can
                  think \(\Cc(G;r^*\A)\) as a dense
                  \(^*\)\nb-subalgebra of the crossed product
                  \(A\rtimes_\alpha G\), as well as, the Fell bundle
                  algebra \(\Cst(G;r^*\A)\). By
                  \cite[Proposition~3.131]{Goehle2009phd-thesis}
                  (or~\cite[Theorem
                  4.13]{Muhly-Williams2008FellBundle-ME}) we conclude
                  that the norm of \(A\rtimes_\alpha G\) and the norm
                  of \(\Cst(G;r^*\A)\) are same. Therefore,
                  \(A \rtimes_\alpha G \iso \Cst(G;r^*\A)\).
  	
                  \noindent (2): Let \(x\in G^{(0)}\) be a unit. For
                  \(\gamma \in G^x_x\), the map
                  \(\alpha_\gamma\colon A(x) \to A(x)\) is an
                  \(^*\)\nb-automorphism. Therefore, we can restrict
                  the action \(\alpha\) on the isotropy group
                  \(G^x_x\) to get a group dynamical system
                  \((A(x),G^x_x,\alpha|_{G^x_x})\). For any
                  \(\gamma \in G^x_x\) the fibre over \(\gamma\) is
                  \(A(r(\gamma))\) which is \(A(x)\). Therefore, as a
                  \(^*\)\nb-algebra \(\Cc(G^x_x;A(x))\) and
                  \(\Cc(G^x_x;r^*\A|_{G^x_x})\) are isomorphic. Then
                  by~\cite[Corollary~2.46]{Williams2007Crossed-product-Cst-Alg}
                  we have
                  \(A(x) \rtimes_{\alpha|_{G^x_x}} G^x_x \iso
                  \Cst(G^x_x;r^*\A|_{G^x_x})\).
                \end{proof}
  
                \noindent The following result characterises the KMS
                states on the groupoid crossed product when the acting
                groupoid is \emph{\etale}.
  
                \begin{corollary}[{Corollary of
                    Theorem~\ref{thm:KMS-state} and
                    Proposition~\ref{rel:cross:pro-Fell:bundle}}]
                  \label{kms-state-gpd:cros}
                  Let \((A, G, \alpha)\) be a groupoid dynamical
                  system where \(G\) is a locally compact, Hausdorff,
                  second countable, {\'e}tale groupoid. Let
                  \(\beta \in \R\). Let \(\tau\) be a real dynamics on
                  \(A\rtimes_{\alpha}G\) induced by a real-valued
                  \(1\)\nb-cocycle \(c\). Then there is a one-to-one
                  correspondence between \KMS\ states \(\phi\) on
                  \(A\rtimes_{\alpha} G\) and a pair \((\mu, [\Phi])\)
                  consisting of a probability measure \(\mu\) on
                  \(\base\) and a \(\mu\)\nb-measurable field
                  \(\Phi=\{\phi_x\}_{x \in G^{(0)}}\) where \(\phi_x\)
                  is a state on
                  \(A(x)\rtimes_{\alpha|_{G^x_x}} G^x_x\) such that
                  \(A(x)\) contained in centraliser \(\phi_x\)
                  satisfying:
                  \begin{enumerate}
                  \item \label{cross:prod:Cond-I} \(\mu\) is a
                    quasi-invariant measure with Radon--Nikodym
                    derivative \(\e^{-\beta c}\).
                  \item \label{cross:prod:Cond-II} There is a
                    \(\mu\)\nb-conull set~\(K\subseteq \base[G]\) such
                    that the field of states
                    \(\Phi = \{\phi_x\}_{x\in \base}\) satisfies the
                    following condition. Let \(x\in K\). For any
                    \(\gamma \in G^x_x\), \(\eta \in G_x\),
                    \(a\in A(x)\) and \(\xi \in A(r(\eta))\) the
                    following equation holds:
                    \[
                      \phi_{s(\eta)}\big(
                      a\alpha_{\gamma\eta^{-1}}(\xi^*\xi)\cdot
                      \delta_\gamma\big) =
                      \phi_{r(\eta)}\big(\xi\alpha_{\eta}(a)\alpha_{\eta
                        \gamma \eta^{-1}}(\xi^*)\cdot \delta_{\eta
                        \gamma \eta^{-1}}\big).
                    \]
                  \end{enumerate}
                \end{corollary}
  
  \begin{proof}
    Proposition~\ref{rel:cross:pro-Fell:bundle} says that \(r^*\A\) as
    a Fell bundle over \(G\) and
    \(A\rtimes_{\alpha} G \simeq \Cst(G;r^*\A)\) and
    \( A(x) \rtimes_{\alpha|_{G^x_x}} G^x_x \simeq
    \Cst(G^x_x;r^*\A|_{G^x_x}).\) Theorem~\ref{thm:KMS-state} gives us
    a one-to-one correspondence between a \KMS\ states \(\phi\) on
    \(A \rtimes_{\alpha} G\) and a pair \((\mu, [\Phi])\)
    satisfying~\eqref{Cond:I} and \eqref{Cond:II} of
    Theorem~\ref{thm:KMS-state}. Let
    \(\Phi = \{\phi_x\}_{x\in \base}\). Then the definition of the
    convolution and involution of groupoid crossed product and the
    \eqref{Cond:II} of Theorem~\ref{thm:KMS-state} forces that
    \(\Phi = \{\phi_x\}_{x\in \base}\) satisfies
    Condition~\eqref{cross:prod:Cond-II} of the current corollary.
  \end{proof}
  
  If we apply the above corollary for \(\beta =0\), then we get a
  result which characterises the invariant traces on the
  \(A\rtimes_{\alpha} G\).
  
  \begin{corollary}
    Let \((A, G, \alpha)\) be a groupoid dynamical system where \(G\)
    is a locally compact, Hausdorff, second countable, {\'e}tale
    groupoid. Let \(\tau\) be a real dynamics on
    \(A\rtimes_{\alpha}G\) induced by a \(1\)\nb-cocycle \(c\). Then
    there is an one-to-one correspondence between the
    \(\tau\)\nb-invariant trace \(\phi\) on \(A\rtimes_{\alpha} G\)
    and a pair \((\mu, [\Phi])\) consisting of a probability measure
    \(\mu\) and a \(\mu\)\nb-measurable field
    \( \Phi = \{\phi_x\}_{x \in G^{(0)}}\) where \(\phi_x\) is a state
    on \(A(x)\rtimes_{\alpha|_{G^x_x}} G^x_x\) with \(A(x)\) contained
    in the centraliser \(\phi_x\) satisfying:
    \begin{enumerate}
    \item \(\mu\) is an invariant measure.
    \item There is a \(\mu\)\nb-conull set~\(K\subseteq \base[G]\)
      such that the field of states \(\Phi = \{\phi_x\}_{x\in \base}\)
      satisfies the following condition. Let \(x\in K\). For any
      \(\gamma \in G^x_x\), \(\eta \in G_x\), \(a\in A(x)\) and
      \(\xi \in A(r(\eta))\) the following equation holds:
      \[
        \phi_{s(\eta)}\big( a\alpha_{\gamma\eta^{-1}}(\xi^*\xi)\cdot
        \delta_\gamma\big) =
        \phi_{r(\eta)}\big(\xi\alpha_{\eta}(a)\alpha_{\eta \gamma
          \eta^{-1}}(\xi^*)\cdot \delta_{\eta \gamma \eta^{-1}}\big).
      \]
  		
    \end{enumerate}
  \end{corollary}

  \subsection{KMS ststes on twisted groupoid crossed product}
  \label{sec:kms-state-gpd-twisted-cp}
  
  Renault~\cite{Renault1980Gpd-Cst-Alg} introduced and studied twisted
  groupoid algebra where the twist given by a continuous
  \(2\)\nb-cocycle. Then
  in~\cite{Renault1985Representations-of-crossed-product-of-gpd-Cst-Alg},
  he introduced twist over a groupoid \(G\) by the groupoid extension
  
  \begin{equation}
    \label{equ:gpd-extension}
    \base \to S \mathrel{\mathop{\rightarrow}^{i}} \Sigma \mathrel{\mathop{\rightarrow}^{j}} G \to \base
  \end{equation}
  
  \noindent where \(\Sigma\) is groupoid with unit space \(\base\),
  \(S\) is an abelian group bundle, \(i\) is injective and \(j\) is
  surjective groupoid homomorphism. We can think of \(S\) as a closed
  subgroupoid of \(\Sigma\). The twists coming from a \(2\)\nb-cocycle
  is a particular case of the twist given by groupoid
  extension. However, in general, not all twists induced from a
  \(2\)\nb-cocycle
  (see~\cite[Section~11.1]{Sims-Szabo-Williams2020Operator-alg-book}).
  Renault generalised Green's notion of the twisted group action
  (see~\cite{Green1978local-structure-of-twisted-covariant-alg}) to
  twisted groupoid action as follows:
  
  \begin{definition}[Green--Renault twisted dynamical system]
    Let \(G,\Sigma,\) and \(S\) as in above. A twisted groupoid
    dynamical system is an upper semicontinuous \(\Cst\)\nb-bundle
    \(\pi \colon \A \to \base (= \Sigma^{(0)})\) on which \(\Sigma\)
    acts by the isomorphisms \(\{\alpha_\sigma\}_{\sigma\in \Sigma}\)
    such that the restricted action of \(\alpha\) gives an action of
    \(S\) is given by a homomorphism \(\chi\) from \(S\) to
    \( \bigsqcup_{s\in S} M(A(r(s)))\) satisfying
    \( \alpha_s(a) = \chi(s)a \chi(s)^* ~\textup{ for } (s,a) \in
    S\times_{r,\base, \pi}\A\) and
    \( \chi(\sigma s\sigma^{-1}) = \alpha_\sigma(\chi(s))~ \textup{
      for } (\sigma,s)\in \Sigma\times_{s, \base, r}S\).
  \end{definition}
  There is an action of \(S\) on the bundle
  \(r^*\A = \big\{(a,\sigma) \in \A \times \Sigma : \pi(a) =
  r(\sigma)\big\}\) by
  \( (a,\sigma)\cdot s \defeq \big(a\chi(s)^*,s\sigma\big).\) We
  define \(\mathcal{B}\) to be the quotient \(r^*\A/S\) and the bundle
  map \(p\colon \mathcal{B} \to G\) by \(p([a,\sigma]) = j(\sigma)\).
  Then Lemma~2.7 of~\cite{Muhly-Williams2008FellBundle-ME}, says that
  \(p\colon \mathcal{B}\to G\) is a Fell bundle in the following way:
  let
  \( \mathcal{B}^{(2)} \defeq \big\{\big([a,\sigma],[b,\tau]\big) :
  \big(j(\sigma),j(\tau)\big) \in G^{(2)}\big\}.  \) Then for
  \([a, \sigma], [b, \tau] \in \mathcal{B},\) define
  \[
    [a,\sigma][b,\tau] \defeq [a\alpha_\sigma(b), \sigma\tau];\quad
    [a,\sigma]^* \defeq [\alpha^{-1}_\sigma(a^*), \sigma^{-1}].
  \]
  For a locally compact, Hausdorff, second countable groupoid \(G\),
  the \(\Cst\)\nb-algebra of this Fell bundle \(\mathcal{B}\) is
  called twisted crossed product \(\Cst\)\nb-algebra of
  \((G, \Sigma, \A, \alpha, \chi)\). It is denoted by
  \(\Cst(\mathcal{B})\). Our goal is to identify the
  \(\Cst\)\nb-algebra at isotropy of the Fell
  bundle~\(\mathcal{B}\). For any \(x\in \base\), the short exact
  sequence of groupoids~\eqref{equ:gpd-extension} become the short
  exact sequence of groups
  \[
    \{x\} \to S_x \mathrel{\mathop{\rightarrow}^{i}} \Sigma^x_x
    \mathrel{\mathop{\rightarrow}^{j}} G^x_x \to \{x\}
  \]
  where \(S_x\) is an abelian group. Note that \(S_x\) is a normal
  subgroup of \(\Sigma^x_x\). For any \(x\in \base\) the map
  \(\chi_x\colon S_x \to UM(A(x)) (= UM(\A_x))\) is a continuous
  homomorphism where \(\chi_x = \chi|_{S_x}\). Moreover for
  \(s\in S_x, a\in \A_x\) we have
  \((\alpha|_{\Sigma^x_x})_s (a) = \chi_x(s) a \chi(s)^*\) and for any
  \(\sigma \in \Sigma^x_x\) we have
  \( \chi_x(\sigma s \sigma^{-1}) =
  \overline{(\alpha|_{\Sigma^x_x})}_\sigma \chi_x(s)\).  Therefore,
  \((\A_x, \Sigma^x_x, \alpha|_{\Sigma^x_x}, \chi_x)\) is a twisted
  dynamical system in Green's sense
  (see~\cite[Section~7.4.1]{Williams2007Crossed-product-Cst-Alg}). Moreover,
  a similar argument as Proposition~\ref{rel:cross:pro-Fell:bundle}(2)
  says that: as a \(^*\)\nb-algebras
  \(\Cc(\Sigma^x_x, \A_x, \alpha|_{\Sigma^x_x}, \chi_x)\) and
  \(\Cc(G^x_x;\mathcal{B}|_{G^x_x})\) are isomorphic. Then
  by~\cite[Corollary~2.46]{Williams2007Crossed-product-Cst-Alg} we
  have
  \(A(x) \rtimes^{\chi_x}_{\alpha|_{G^x_x}} G^x_x \iso
  \Cst(G^x_x;r^*\A|_{G^x_x})\).  Since the \(\Cst\)\nb-algebras of
  this Fell bundle over the isotropy group at \(x\in \base\) is
  isomorphic to \(\A_x\rtimes^{\chi_x}_{\alpha|_{\Sigma^x_x}} G^x_x\),
  we apply Theorem~\ref{thm:KMS-state} and get an analogous result as
  (Corollary~\ref{kms-state-gpd:cros}) groupoid crossed product.
  
  \section{An integration-disintegration theorem for states on
    \(\Contz(X)\)-algebras}
  \label{sec:state-upp-semi-con-bundle}
 
  In this section, we shall establish an integration-disintegration
  theorem for the states (or traces) on the section algebra of an
  upper semicontinuous \(\Cst\)\nb-bundle. To prove such
  integration-disintegration theorem we need not appeal to the
  Muhly--Williams disintegration theorem for Fell bundle and our
  approch is quite simpler. We also remove the `centraliser related'
  hypothesis on the states. Hence, the main result of this section is
  not a particullar case of
  Theorem~\ref{prop:bij-betwen-states-and-fields-of-states}.  Recall
  the definition of measurable field of states from
  Definition~\ref{def:mbl-field-state}.
 
 \begin{proposition}\label{int-state-upp-sem}
   Let \(p\colon \A \to X \) be a upper semicontinuous, separable
   \(\Cst\)\nb-bundle over \(X\). Let \((\mu, \{\phi_x\}_{x\in X})\)
   be a pair consisting of a probability measure \(\mu\) on \(X\) and
   a \(\mu\)\nb-measurable field \(\{\phi_x\}_{x\in X}\) where
   \(\phi_x\) is a state (or trace) on~\(\A_x\). Then the pair
   \((\mu, \{\phi_x\}_{x\in X})\) defines a state (or trace)~\(\phi\)
   on the section algebra \(\mathrm{C}_0(X;\A)\) as follows:
   \[
     \phi(f)=\int_{X}\phi_x(f(x)) \mathrm{d}\mu(x)
   \]
   for \(f \in \Cc(X;\A)\).
 \end{proposition}
 
 \begin{proof}
   Define the linear functional~\(\phi\) on \(\Cc(X;\A)\) by
   \( \phi(f)=\int_{X}\phi_x(f(x)) \mathrm{d}\mu(x)\).  The positivity
   of the functional \(\phi\) follows from the fact that each
   \(\phi_x\) is positive and \(\mu\) is a positive measure on \(X\).
   Since each \(\phi_x\) has unit norm, for all \(f \in \Cc(X;\A)\),
   \[
     \abs{\phi(f)} \leq \int_X \abs{\phi_x(f(x))}\,\dd\mu(x) \leq
     \int_X \norm{f(x)}\,\dd\mu(x) \leq \norm{f}_{\infty}.
   \]
   Thus, \(\phi\) is a bounded functional on \(\Contc(X;\A)\) with
   \(\norm{\phi}\leq 1\). Extended \(\phi\) to \(\Contz(X;\A)\) as a
   bounded linear functional whose norm \(\leq 1\). We denote this
   extension by \(\phi\) itself. Our next claim is \(\norm{\phi} =1\):
   Let \((u_n)_{n\in \N}\) be an approximate identity of
   \(\Contz(X;\A)\) consisting of sections with compact support. Then
   \((u_n(x))_{n\in \N}\) is an approximate identity of \(\A_x\) (see
   Lemma~\ref{lem:fiberwisw-appr-iden}).  By~\cite[Theorem
   3.3.3]{Murphy-Book} we get,
   \begin{multline*}
     \norm{\phi}=\lim_{n}\phi(u_n)=\lim_{n}\int_{X}\phi_x(u_n(x))\mathrm{d}\mu(x)=\int_{X} \lim_{n} \phi_x(u_n(x))\mathrm{d}\mu(x)\\
     =\int_{X}\norm{\phi_x}\mathrm{d}\mu(x)=1.
   \end{multline*}
   The third equality above is because of the dominated convergence
   theorem. The fourth one follows
   from~Lemma~\ref{lem:fiberwisw-appr-iden} and~\cite[Theorem
   3.3.3]{Murphy-Book}.
 	
   \noindent Note that, if each \(\phi_x\) is a trace, then \(\phi\)
   is also a trace:
   \begin{equation*}
     \phi(fg)=\int_{X}\phi_x\big(f(x)g(x)\big)\mathrm{d}\mu(x)=\int_{X}\phi_x\big(g(x)f(x)\big)\mathrm{d}\mu(x)
     =\phi(gf)
   \end{equation*}
   for all \(f,g \in \Cc(X;\A)\).
 \end{proof}
 
 \begin{proposition}\label{dis-sta-upp-sem}
   Let \(p\colon \A \to X \) be a upper semicontinuous, separable
   \(\Cst\)\nb-bundle. Let \(\phi\) be a state (or trace) on
   \(\Contz(X;\A)\). Then there is a pair
   \((\mu, \{\phi_x\}_{x\in X})\) consisting of a probability measure
   \(\mu\) on \(X\) and a \(\mu\)\nb-measurable field
   \(\{\phi_x\}_{x\in X}\) where \(\phi_x\) is a state (or trace) on
   the fibre \(\A_x\). The relation between the given state (or trace)
   \(\phi\) and the pair \((\mu, \{\phi_x\}_{x\in X})\) is given by
   \begin{equation*}
     \phi(f)=\int_{X}\phi_x(f(x)) \mathrm{d}\mu(x)
   \end{equation*}
   for \(f \in \Cc(X;\A)\).
 \end{proposition}
 
 \begin{proof}
 
   Let \(\phi\) be a state on \(\Contz(X;\A)\) and let
   \((\pi_\phi, \Hilm_\phi, \xi_\phi)\) be the GNS triple
   of~\(\phi\). Therefore,
   \(\phi(f) = \binpro{\xi_{\phi}}{\pi_{\phi}(f)\xi_\phi}\) for
   \(f\in \Cc(X;\A)\).  We think \(\Contz(X;\A)\) as a
   \(\Contz(X)\)\nb-algebra where the action of \(\Contz(X)\) on
   \(\Contz(X;\A)\) given by \( f\cdot \xi (x) = f(x)\xi(x)\) for
   \(f \in \Contz(X)\) and \(\xi \in \Contz(X;\A)\). Propositions~3.99
   and~3.101 of~\cite{Goehle2009phd-thesis} say that there is a Borel
   Hilbert bundle \(X*\Hilm\) and a finite positive measure \(\nu\) on
   \(X\) such that \(\Hilm_{\phi}\) is unitary equivalent to
   \(\int_{X}^{\oplus} \Hilm_x\dd\mu(x)\) and \(\pi_{\phi}\) is also
   unitary equivalent to \(\int_{X}^{\oplus} \pi_x\dd\mu(x)\) where
   \(\pi_x:\A_x \to \Bound(\Hilm_x)\) is a representation. Identify
   the cyclic vector \(\xi_{\phi}\) for \(\pi_{\phi}\) to a vector
   field \((\xi_x)_{x\in X}\). Define the positive linear map
   \(\phi_x\) on \(\A_x\) by
   \[
     \phi_x(a) = \norm{\xi_x}^{-2}\binpro{\xi_x}{\pi_x(a)\xi_x}
   \]
   where \(a\in \A_x\). To show \(\norm{\phi_x} =1\): let
   \((a_n)_{n\in \N}\) be an approximate identity of \(\A_x\). Then
   \[
     \norm{\phi_x} = \lim_{n} \phi_x(a_n) = \lim_{n}
     \norm{\xi_x}^{-2}\binpro{\xi_x}{\pi_x(a_n)(\xi_x)} =
     \norm{\xi_x}^{-2} \binpro{\xi_x}{\xi_x} =1.
   \]
   The first equality of above computation is follows
   from~\cite[Theorem~3.33]{Murphy-Book} and third one follows from
   nondegeneracy of \(\pi_x\). Therefore, \(\phi_x\) is a state on
   \(\A_x\) for \(x\in X\).
 	
   \noindent Define \(\dd \mu = \norm{\xi_x}^2 \dd\nu\). For any
   \(f\in \Cc(X;\A)\), we have
   \begin{multline*}
     \phi(f) = \binpro{\xi_{\phi}}{\pi_{\phi} (f) \xi_{\phi}} = \int_{X} \binpro{\xi_x}{\pi_x(f(x))\xi_x} \dd\nu(x) = \int_{X} \norm{\xi_x}^2 \phi_x(f(x)) \dd \nu(x)\\
     = \int_{X} \phi_x(f(x)) \dd \mu(x).
   \end{multline*}
   To prove \(\mu\) is a probability measure, we take an approximate
   identity \((u_n)_{n\in\N}\) for \(\Cst(X;\A)\). From the last
   computation we have
   \( \phi(u_n) = \int_{X} \phi_x(u_n(x)) \dd \mu\).  Taking limit as
   \(n \to \infty \) in above and using~\cite[Theorem
   3.33]{Murphy-Book}, dominated convergence theorem and
   Lemma~\ref{lem:fiberwisw-appr-iden}, we have
   \[
     1 = \norm{\phi} = \lim_{n} \phi(u_n) = \lim_{n} \int_{X}
     \phi_x(u_n(x)) \dd \mu = \int_{X} \lim_{n} \phi_x(u_n(x)) \dd \mu
     = \int_{X} \dd \mu.
   \]
   Therefore, \(\mu\) is a probbility measure on \(X\).

   Suppose \(\phi\) is a trace on \(\Contz(X;\A)\). Let
   \(f,f_1\in \Cc(X;\A)\) and \(g\in \Cc(X)\). Define
   \(h = g\cdot f_1 \in \Cc(X;\A)\). Using the trace proparty of
   \(\phi\) we have \(\phi(f\cdot h) = \phi(h\cdot f)\). Which implies
   that
   \[
     \int_{X} \phi_x\big(f\cdot h(x)\big) \dd \mu(x) = \int_{X}
     \phi_x\big(h\cdot f(x)\big) \dd \mu(x).
   \]
   The defintion of the multiplication on \(\Contz(X;\A)\) gives us
   \[
     \int_{X} \phi_x\big(f(x) g(x)f_1(x)\big) \dd \mu(x) = \int_{X}
     \phi_x\big(g(x) f_1(x) f(x)\big) \dd \mu(x).
   \]
   Since \(g\) is a complex-valued function, the above equation became
   \[
     \int_{X} g(x) \phi_x(f(x)f_1(x)) \dd \mu(x) = \int_{X}
     g(x)\phi_x( f_1(x)f(x)) \dd \mu(x).
   \]
   But \(g\in \Cc(X)\) was arbitary,
   \( \phi_x(f(x)f_1(x)) = \phi_x( f_1(x)f(x)) \) for
   \(\mu\)\nb-almost everywhere \(x\in X\). Since \(p\colon \A \to X\)
   has enough sections \(\phi_x\) is a trace on \(\A_x\) for
   \(\mu\)\nb-almost everywhere \(x\in X\).
 	
 \end{proof}
 
 \begin{remarks}\label{remks:one-to-one-C(X)-alg}
   (1). Using the similar argument as in
   Proposition~\ref{prop:integration-of-states-2}, we conclude that if
   \((\mu, \{\phi_x\}_{x\in X})\) and \((\nu, \{\psi_x\}_{x\in X})\)
   are two disintegration of a state \(\phi\) on \(\Cst(X;\A)\) as in
   Proposition~\ref{dis-sta-upp-sem}, then \(\mu=\nu\) and
   \(\{\phi_x\}_{x\in X}\) and \(\{\psi_x\}_{x\in X}\) are eqivalent
   field of states.
 	
   \noindent (2). Along the same line of
   Lemma~\ref{lem:disint-states}, we can say that if
   \(\{\phi_x\}_{x\in X}\) and \(\{\psi_x\}_{x\in X}\) are two
   \(\mu\)\nb-equilvalent field of states, then the process of
   Proposition~\ref{int-state-upp-sem} induce the same state on
   \(\Cst(X;\A)\). We shall denote the equivalence class of
   \(\Phi = \{\phi_x\}_{x\in \base}\) by \([\Phi]\).
 	
   \noindent (3). Moreover, the above two process are inverses of each
   other.
 \end{remarks}
 
 \begin{theorem}\label{int-dis-state-upp-semi}
   Let \(p\colon \A \to X \) be a separable, upper semicontinuous
   \(\Cst\)\nb-bundle. Then there is an one-to-one correspondence
   between the state (or trace) \(\phi\) on \(\Contz(X;\A)\) and a
   pair \((\mu, [\Phi])\) consisting of a probability measure \(\mu\)
   on \(X\) and a \(\mu\)\nb-measurable field
   \( \Phi = \{\phi_x\}_{x\in X}\) where \(\phi_x\) is a state (or
   trace) on the fibre \(\A_x\). Moreover, the relation between a
   state (or trace) \(\phi\) and the associated pair \((\mu, [\Phi])\)
   is given by
   \begin{equation}\label{eq:rel-state:on:C(X)-alg}
     \phi(f)=\int_{X}\phi_x(f(x)) \mathrm{d}\mu(x),
   \end{equation}
   for \(f \in \Cc(X;\A)\).
 \end{theorem}
 \begin{proof}
   The proof follows from Propositions~\ref{int-state-upp-sem}
   and~\ref{dis-sta-upp-sem} and
   Remarks~\ref{remks:one-to-one-C(X)-alg}.
 \end{proof}
 
 \begin{remark}
   Theorem~\ref{int-dis-state-upp-semi} does not assume the condition
   that \(\Contz(X;\A)\) contained in centraliser of a state \(\phi\),
   which ensure that Theorem~\ref{int-dis-state-upp-semi} is not a
   special case of
   Theorem~\ref{prop:bij-betwen-states-and-fields-of-states}.
 \end{remark}

 \subsection{Group bundle}
 A locally compact, Hausdorff groupoid \(G\) is called a \emph{group
   bundle} if its range and source map are the same. In other words,
 \(G\) is the bundle \(G = \bigcup_{x\in \base} G^x_x\), where
 \(G^x_x\) is the isotropy at \(x\). The bundle map of the group
 bundle is \(p =r =s\). A locally compact, Hausdorff, second countable
 group bundle \(G\) has a Haar system \emph{iff} the bundle map \(p\)
 is open~(\cite[Theorem 6.9]{Williams2019A-Toolkit-Gpd-algebra}).
 \begin{example}
   Let \(G\) be a locally compact, Hausdorff groupoid. Then
   \( \text{Ib}(G) \defeq \{\gamma \in G : s(\gamma) = r(\gamma)\} =
   \bigcup_{x\in \base} G^x_x \) is a group bundle called the isotropy
   bundle of \(G\). Since the range and source maps are continuous,
   the isotropy bundle is a closed subgroupoid of \(G\). The isotropy
   bundle is the bundle of isotropy groups over the space of units
   \(\base\).
 \end{example}
 
 \noindent Williams~\cite{Williams2019A-Toolkit-Gpd-algebra}
 identifies the \(\Cst\)\nb-algebra of a group bundle \(G\) as a
 \(\Contz(\base)\)\nb-algebra.
 
 \begin{proposition}
   [{\cite[Proposition~5.38]{Williams2019A-Toolkit-Gpd-algebra}}]
   \label{prop:group:bundle-characterisation}
   Let \(G\) be a locally compact, Hausdorff, second countable group
   bundle with a Haar system. Then \(\Cst(G)\) is a
   \(\Contz(\base)\)\nb-algebra with fibre \(\Cst(G)(x)\) isomorphic
   to \(\Cst(G^x_x)\) for every \(x\in \base\).
 \end{proposition}
 
 \noindent In Proposition~\ref{prop:group:bundle-characterisation},
 the action of \(\Contz(\base)\) on \(\Cst(G)\) is
 \(\phi_{\Cst(G)} \colon \Contz(\base) \to ZM(\Cst(G))\) given by
 \( \phi_{\Cst(G)}(f)(g)(\gamma) \defeq f(p(\gamma)) g(\gamma), \)
 where \(f\in \Contz(\base), g\in \Cc(G)\) and \(p\) is the bundle
 map. Combining Theorem~\ref{int-dis-state-upp-semi} with the above
 proposition, we get the following immediate corollary.
 
 \begin{corollary}
   \label{coro:state-trace-on-gp-bundle}
   Let \(G\) be a locally compact, Hausdorff, second countable group
   bundle with a Haar system. Then there is a one-to-one
   correspondence between states (or traces) \(\phi\) on \(\Cst(G)\)
   and a pair \((\mu, [\Phi])\) consisting of a probability measure
   \(\mu\) on \(\base\) and a \(\mu\)\nb-measurable field
   \(\Phi = \{\phi_x\}_{x\in \base}\) where each \(\phi_x\) is a state
   (or trace) on \(\Cst(G^x_x)\). Moreover, the relation between
   \(\phi\) and the pair \((\mu, [\Phi])\) is given by
   \[
     \phi(f) = \int_{G^{(0)}} \phi_x(f|_{G^x_x}) \dd\mu(x)
   \]
   for \(f\in \Cc(G)\).
 \end{corollary}
 
 \subsection{Fell bundle over a group bundle}
 
 A locally compact, Hausdorff, second countable group bundle \(G\)
 acts on its unit space~\(\base\) trivially. Therefore, for any unit
 \(x\in \base\) the orbit \([x] = \{x\}\). A particular case of
 Corollary~10
 of~\cite{Sims-Williams2013Amenability-Fell-bundle-over-gpd} for group
 bundle gives us the following result:
 
 \begin{proposition}[{\cite[Corollary
     10]{Sims-Williams2013Amenability-Fell-bundle-over-gpd}}]
   \label{prop:Fell-bund-over-gp-bundle}
   Let \(G\) be a locally compact, Hausdorff, second countable group
   bundle with Haar system and \(p\colon \A \to G\) be a Fell bundle
   over \(G\). Then \(\Cst(G;\A)\) is a \(\Contz(\base)\)\nb-algebra
   with fibre over \(x\in \base\) being isomorphic to
   \(\Cst(G^x_x;\A(x))\).
 \end{proposition}
 The action of \(\Contz(\base)\) on \(\Cst(G;\A)\) is given by
 \( g\cdot f(\gamma) \defeq g(r(\gamma)) f(\gamma) \) where
 \(g\in \Contz(\base), f\in \Cst(G;\A)\) and \(\gamma\in G\).
 Therefore, We have an analogous result as
 Corollary~\ref{coro:state-trace-on-gp-bundle} for the Fell bundle
 over group bundle.
 
 \subsection{Crossed product \(\Cst\)-algebras}
 
 In~\cite{Gardella-Hirshberg-Santiago2021Rokhlin-dim-Crossed-prod}
 Gardella, Hirshberg and Santiago prove a structure theorem for
 certain crossed product \(\Cst\)\nb-algebra. We use their result to
 characterise the states on certain crossed product algebras.
 
 \begin{proposition}[{\cite[Proposition~4.5]{Gardella-Hirshberg-Santiago2021Rokhlin-dim-Crossed-prod}}]
   \label{prop:st-cros-prod-GHS}
   Let \(G\) be a compact group, let \(X\) be a locally compact space
   and let \(A\) be a \(\Cst\)\nb-algebra. Let \(G\) acts \(X\) and
   let \(\alpha \colon G\to \textup{Aut}(A)\) be an action by
   automorphism. Endow \(\Contz(X;A)\) with the diagonal
   \(G\)\nb-action \(\gamma\). Then \(\Contz(X;A)\rtimes_\gamma G\) is
   a continuous \(\Contz(X/G)\)\nb-algebra. Moreover, if
   \(\pi\colon X \to X/G\) denotes the quotient map, then the fibre
   over \(\pi(x)\) is canonically isomorphic to
   \( (A\rtimes_\alpha St(x)) \otimes \mathcal{K}(L^2(G/St(x))), \)
   where \(St(x)\) is the stabiliser at \(x\in X\).
 \end{proposition}
 
 \noindent The last result combined with
 Theorem~\ref{int-dis-state-upp-semi} gives the following corollary.
 
 \begin{corollary}
   Let \(G\) be a compact group, let \(X\) be a locally compact,
   Hausdorff, second countable space and let \(A\) be a
   \(\Cst\)\nb-algebra. Assume \(X\) and \(A\) satisfies the same
   hypothesis as in Proposition~\ref{prop:st-cros-prod-GHS}.  Then
   there is a one-to-one correspondence between state (or trace)
   \(\phi\) on \(\Contz(X;A)\rtimes_\gamma G\) and a pair
   \((\mu, [\Phi])\) consisting a probability measure \(\mu\) on
   \(X/G\) and a \(\mu\)\nb-measurable field
   \(\Phi = \{\phi_{\pi(x)}\}_{\pi(x) \in X/G}\) where
   \(\phi_{\pi(x)}\) is a state (or, trace) on
   \((A\rtimes_\alpha St(x)) \otimes \mathcal{K}(L^2(G/St(x)))\),
   where \(St(x)\) is the stabiliser at \(x\in X\). Moreover, the
   relation between \(\phi\) and \((\mu, [\Phi])\) is given by
   Equation~\eqref{eq:rel-state:on:C(X)-alg}.
 	
 \end{corollary}

 \section{KMS states on matrix algebras}
 \label{sec:kms-stat-text}

 In this section, we discuss a groupoid model of the matrix algebra
 \(\Mat_n(\mathrm{C}(X))\) for a compact topological space \(X\). We
 also study the quasi-invariant measures on the unit space of the
 underlying principal \etale\ groupoid of \(\Mat_n(\mathrm{C}(X))\).

 Let \(X\) be a compact, Hausdorff topological space and
 \(N=\{1,2,3,\cdots n\}\) for \(n\in \N\). Consider
 \(G= N\times X \times N\) with the product topology. Then \(G\) is a
 compact, Hausdorff topological groupoid with the following
 operations: \((h,x,m), (n,y,k)\in G \) are composable \emph{iff}
 \(x=y\) and \(m=n\); the composition is given by
 \((h,x,m)(m,x,k) \defeq (h,x,k)\); \((h,x,m)^{-1} \defeq (m,x, h)\).
 For \((h,x,m)\in G\) we have \(s(h,x,m) = (m,x,m)\),
 \(r(h,x,m) = (h,x,h)\). The space of units
 \(\base = \{(k,x,k): x\in X, k=1,2,\cdots, n\}\). Therefore,
 \(\base\) identified with \(N\times X\). Since the range map \(r\) is
 a local homeomorphism, the groupoid \(G\) is \etale. Given
 \(f\in \mathrm{C}(G)\) and \((h,x,k) \in G\), we write \(f(h,x,k)\)
 as \(f_{h,k}(x)\). Then
 \((f_{h,k}(x)) \in \mathrm{C}(X,\Mat_n(\C))\). Then we note that for
 \(f,g\in \mathrm{C}(G)\), the multiplication and involution agree
 those on \(\mathrm{C}(X,\Mat_n(\C))\). That is,
 \[
   f*g(h,x,k) \defeq \sum_{m \in N} f(h,x,m) g(m, x, k) ~\text{ and }
   f^*(h,x,k) \defeq \overline{f(k,x,h)}.
 \]
 Thus \(\mathrm{C}(G)\) is isomorphic to \(\Mat_n(\mathrm{C}(X))\) as
 a \(^*\)\nb-algebra. Uniqueness of \(\Cst\)\nb-norm implies
 \(\Cst(G) \iso \Mat_n(\mathrm{C}(X))\). For a unit \(u=(h,x,h)\), we
 have \(G^u_u = \{(h,x,h)\} = \{u\}\). Therefore, \(G\) is a principle
 groupoid. \cite[Proposition~II.5.4]{Renault1980Gpd-Cst-Alg} allows us
 to describe the \(\KMS\)-states on \(\Cst(G)\) in terms of
 quasi-invariant probability measure on \(G^{(0)} = N\times X\). Next,
 we investigate those quasi-invariant measures on \(N \times X\) whose
 Radon--Nikodym derivative \(\e^{-\beta c}\); for some real-valued
 \(1\)\nb-cocycle \(c\) on \(G\).

 Let \(c\) be a real-valued \(1\)\nb-cocycle on \(G\).  The unit space
 \(G^{(0)} = N\times X\) can be thought of as a disjoint union of
 \(n\) many copies of the compact topological space \(X\). Therefore,
 \(\base = \bigsqcup_{k=1}^n X_k\) where \(X_k =\{k\}\times X\).  Let
 \(\mu_1\) be a finite measure on \(X_1\). Using the real-valued
 \(1\)\nb-cocycle \(c\) we define a measure \(\mu_2\) on \(X_2\) by
 \( \dd \mu_2 (x) = \e^{-\beta c (2, x, 1) }\dd \mu_1(x); \) and, in
 general, the measure \(\mu_k\) on \(X_k\) as
 \( \dd \mu_k (x) = \e^{-\beta c (k, x, 1) }\dd \mu_1(x) \) for
 \(k\in N\). Since \(c\) is a cocycle,
 \(\dd \mu_h (x)= \e^{-\beta c (h, x, m) }\dd \mu_m(x)\) for
 \(h,m \in N\). Define a measure \(\mu \defeq \sum_{i=1}^{n} \mu_i\)
 on the unit space \(N\times X\). Then
 \begin{equation}\label{equ:def-prob-measure}
   \nu\defeq \frac{\mu}{\mu(\base)}
 \end{equation}
 is a probability measure on \(\base\). The measure \(\nu\) is a
 quasi-invariant measure with Radon--Nikodym derivative
 \(\e^{-\beta c}\). To see this we prove the following equality:
 \begin{equation}\label{equ:qus:inv-examp}
   \int_{N\times X} \sum_{(i,x,j) \in G^{(h,x)}} f(i,x,j) \dd \nu(h,x) =\int_{N\times X} \sum_{(i,x,j) \in G_{(h,x)}} f(i,x,j) \e^{-\beta c(i,x,j)} \dd \nu (h,x)
 \end{equation}
 for \(f\in \Cc(G)\).  The left hand side of the above equation is
 \begin{align*}
   \int_{N\times X} \biggr(\sum_{j=1}^{n} f(h,x,j) \biggr)\dd \nu(h,x) &= \frac{1}{\mu(\base)} \sum_{j=1}^{n} \sum_{h=1}^{n} \int_{X_h} f(h,x,j) \dd \mu_h(x) \\
                                                                       &= \frac{1}{\mu(\base)}\sum_{j=1}^{n} \sum_{h=1}^{n} \int_{X_j}\big( f(h,x,j) \e^{-\beta c(h,x,j)} \big) \dd \mu_j(x) \\
                                                                       &= \frac{1}{\mu(\base)} \sum_{h=1}^{n} \sum_{j=1}^{n} \int_{X_j} \big(f(h,x,j) \e^{-\beta c(h,x,j)} \big) \dd \mu_j(x).
 \end{align*}
 The last term equals the right hand side of
 Equation~\eqref{equ:qus:inv-examp}. Therefore, \(\nu\) is a
 quasi-invariant probability measure with Radon--Nikodym derivative
 \(\e^{-\beta c}\).

 Conversely, suppose that \(\nu\) be a quasi-invariant probability
 measure on \(N\times X\) with Radon--Nikodym derivative
 \(\e^{-\beta c}\) where \(c\) is a real-valued \(1\)\nb-cocycle of
 \(G\). Then the probability measure \(\nu\) satisfies the
 Equation~\eqref{equ:qus:inv-examp}. Let \(\nu =\sum_{i=1}^{n} \nu_i\)
 where \(\nu_i\) is a finite measure on \(X_i\) for \(i\in N\). Then
 the left hand side of the Equation~\eqref{equ:qus:inv-examp} is equal
 to
 \( \sum_{j=1}^{n} \sum_{h=1}^{n} \int_{X_h} f(h,x,j) \dd \nu_h(x) \)
 and the right hand side of the Equation~\eqref{equ:qus:inv-examp} is
 equal to
 \begin{multline*}
   \int_{N\times X} \sum_{(i,x,j) \in G_{(h,x)}} f(i,x,j) \e^{-\beta c(i,x,j)} \dd \nu (h,x) = \int_{N\times X} \sum_{i=1}^{n} f(i,x,h) \e^{-\beta c (i,x,h)} \dd\nu(h,x)\\
   = \sum_{i=1}^{n} \sum_{h=1}^{n} \int_{X_h} \big(f(i,x,h)\e^{\beta
     c(i,x,h)} \big)\dd \nu_h(x).
 \end{multline*}
 Therefore, equating the left and right hand sides of
 Equation~\eqref{equ:qus:inv-examp}, gives us
 \[
   \sum_{j=1}^{n} \sum_{h=1}^{n} \int_{X_h} f(h,x,j) \dd \nu_h(x) =
   \sum_{i=1}^{n} \sum_{h=1}^{n} \int_{X_h} \big(f(i,x,h)\e^{\beta
     c(i,x,h)} \big)\dd \nu_h(x)
 \]
 for any \(f\in \Cc(G)\). Choose \(f\in \Cc(G)\) which is supported in
 \(\{m\} \times X \times \{k\}\). Then the last equation becomes
 \[
   \int_{X_m}f(m,x,k) \dd \nu_m(x)= \int_{X_k}\big(f(m,x,k)\e^{-\beta
     c(m,x,k)} \big) \dd \nu_k(x).
 \]
 Since \(X_m, X_k\) both are copies of \(X\), the relation between
 \(\nu_m\) and \(\nu_k\) is given by
 \(\dd \nu_m (x) = \e^{-\beta c(m,x,k)} \dd \nu_k(x)\) where
 \(m,k\in N\). Therefore, we have proved the following result:

\begin{proposition}\label{thm:quasi-inv-masu-on-NXN}
  Let \(X\) be a compact, Hausdorff space and \(N=\{1,2,\cdots,n\}\)
  for \(n\in \N\). As earlier, define the groupoid
  \(G = N\times X \times N\). Let \(c\) be a real-valued
  \(1\)\nb-cocycle on \(G\). A finite measure \(\mu\) induces a
  quasi-invariant probability measure \(\nu\) on \(\base\) with
  Radon--Nikodym derivative \(\e^{-\beta c}\) as in
  Equation~\eqref{equ:def-prob-measure}.  Conversely, assume that
  \(\nu\) is a quasi-invariant probability measure on \(\base\) with
  Radon--Nikodym derivative \(\e^{-\beta c}\). Then \(\nu\) is
  obtained from a finite measure \(\mu\) on \(X\) by the last process.
\end{proposition}

\begin{remark}
  Our last result offers an answer for the Radon--Nikodym problem
  asked by
  Renault~\cite{Renault2001AF-equiv-rela-and-cocycle},~\cite{Renault2005The-R-N-problem-for-appx-proper-equiv-relation},
  for the groupoid \(G= N\times X \times N\).
\end{remark}

Renault~\cite[Proposition~II.5.4]{Renault1980Gpd-Cst-Alg} proved that
the KMS states of \(\Cst\)\nb-algebras of a \emph{principal \etale}
groupoid are completly characterised by the quasi-invariant
probability measures. Therefore, we have the next reslut:

\begin{corollary}[{Corollary of \cite[Proposition~II.5.4]{Renault1980Gpd-Cst-Alg}} and Proposition~\ref{thm:quasi-inv-masu-on-NXN}]
  Let \(X\) be a compact, Hausdorff space and \(N=\{1,2,\cdots,n\}\)
  for \(n\in \N\). As earlier, define the groupoid
  \(G = N\times X \times N\). Let \(c\) be a real-valued
  \(1\)\nb-cocycle on \(G\). Then
  \(\tau \colon \mathbb{R} \to \textup{Aut}(\Mat_n(\mathrm{C}(X)))\)
  is a dynamics given by
  \( (\tau_t(f))_{h ,k}(x) = \e^{\mathrm{i}tc(h,x,k)}f_{h,k}(x).\) Let
  \(\beta \in \R\). Then every \KMS\ states on
  \((\Mat_n(\mathrm{C}(X)), \mathbb{R}, \tau)\) are obtained from a
  finite measure on \(X\) as in
  Proposition~\ref{thm:quasi-inv-masu-on-NXN}.
\end{corollary}

Let \(A\) be a \(\Cst\)-algebra. Consider the trivial Fell bundle
\(\A = G \times A\). Here the Fell bundle operations are given by
\[
  \big((h,x,m),a\big)\big((m,x,k),b\big) = \big((h,x,k), ab\big)
  \textup{ and } \big((h,x,m),a\big)^*=\big((m,x,h),a^*\big)
\]
where \((h,x,m), (m,x,k) \in G\) and \(a, b \in A\). Since this Fell
bundle is a trivial bundle we have
\(\Cst(G;\A) \iso \Cst(G) \otimes A \iso \Mat_n(\mathrm{C}(X))\otimes
A\). Let \(c\) be a continuous real-valued \(1\)\nb-cocycle on
\(G\). For \(f \otimes a \in \Mat_n(\mathrm{C}(X))\otimes A \) define
\(\tau_t(f\otimes a) \defeq g\otimes a\) where
\(g_{h,k}(x) = \e^{\mathrm{i}tc(h,x,k)}f_{h,k}(x)\) for
\((h,x,k)\in G\) and \(t\in \mathbb{R}\). Then
\(\big(\Mat_n(\mathrm{C}(X))\otimes A, \mathbb{R}, \tau\big)\) is a
\(\Cst\)\nb-dynamical system.

\begin{corollary}\label{coro:kms-state-mncxa}
  Let \(\tau\) be a real dynamics on
  \(\Mat_n(\mathrm{C}(X))\otimes A\) defined above. Let
  \(\beta \in \mathbb{R}\). Let \(\phi\) be a \(\KMS\) state on
  \((\Mat_n(\mathrm{C}(X))\otimes A, \mathbb{R}, \tau)\). Then there
  is a pair \((\mu, [\Phi])\) consisting of a probability measure
  \(\mu\) on \(X\) and a \(\tilde{\mu}\)-measurable field
  \(\Phi = \{\phi_{(k,x)}\}_{(k,x)\in N\times X}\) where
  \(\tilde{\mu}\) is a measure on \(N \times X\) obtained from \(\mu\)
  as in Equation~\eqref{equ:def-prob-measure}. Moreover, for
  \(\tilde{\mu}\) almost everywhere on \(N\times X\)
  \[
    \phi_{(k,x)}\big(a\xi^*\xi \cdot \delta_{(k,x)}\big) =
    \phi_{(m,x)}\big(\xi a\xi^* \cdot \delta_{(m,x)}\big)\] for all
  \(m=1,2,\cdots,n\) and \(\xi,a \in A\).
	
  \noindent Conversely, every pair \((\mu, \Phi)\) as in above induces
  a \(\KMS\) states on
  \((\Mat_n(\mathrm{C}(X))\otimes A, \mathbb{R}, \tau)\).
\end{corollary}

\begin{proof}
  Let \(c\) be a continuous real-valued \(1\)\nb-cocycle on \(G\). We
  identify \(\Mat_n(\mathrm{C}(X))\otimes A\) as a \(\Cst\)-algebra of
  the trivial Fell bundle \(\A\) over \(G\) where each fibre is
  \(A\). Then the dynamics induced by \(c\) on \(\Cst(G;\A)\) is same
  as the dynamics \(\tau\). Let \(\phi\) be a \(\KMS\) state on
  \(\Mat_n(\mathrm{C}(X))\otimes A\). Then by
  Theorem~\ref{thm:KMS-state} there is a pair \((\nu, [\Phi])\) with
  \(\Phi = \{\phi_{(k,x)}\}_{(k,x)\in N\times X}\) satisfying
  Conditions~\eqref{Cond:I} and \eqref{Cond:II} of
  Theorem~\ref{thm:KMS-state}. Condition~~\eqref{Cond:I} says that
  \(\nu\) is a quasi-invariant measure on \(N\times X\) with
  Radon--Nikodym derivative \(\e^{-\beta
    c}\). Proposition~\ref{thm:quasi-inv-masu-on-NXN} ensures us
  \(\nu\) obtained from a probability measure \(\mu\) on \(X\). Since
  the groupoid \(G\) is principal for any unit \((k,x) \in \base\) we
  have \(\Cst(G^{(k,x)}_{(k,x)};\A((k,x))) \iso A\). Therefore, each
  \(\phi_{(k,x)}\) is a state on \(A\). Now Condition~\eqref{Cond:II}
  of Theorem~\ref{thm:KMS-state} gives us
  \[
    \phi_{(k,x)}\big(a\xi^*\xi \cdot \delta_{(k,x)}\big) =
    \phi_{r(m,x,k)} \big(\xi a \xi^* \cdot
    \delta_{(m,x,k)(k,x,k)(k,x,m)}\big)
  \]
  for all
  \((m,x,k) \in G_{(k,x)} = \{(m,x,k) : m = 1,2,3, \cdots, n\}\) and
  \(a,\xi \in A\). Then the last equation becomes
  \[
    \phi_{(k,x)}\big(a\xi^*\xi \cdot \delta_{(k,x)}\big) =
    \phi_{(m,x)} \big(\xi a \xi^* \cdot \delta_{(m,x)}\big)
  \]
  for all \(m =1,2, \cdots, n\) and \(a,\xi \in A\).
	
  The converse part directly follows from Theorem~\ref{thm:KMS-state}.
\end{proof}

\section{KMS states associated with \(G\)-spaces}
\label{sec:kms-state-G-space}

The \(G\)\nb-sapce is a particular groupoid dynamical system, and
hence we can associate a Fell bundle corresponding to a \(G\)\nb-space
\(X\). In this particular case,
Proposition~\ref{Prop:kms-state:G-space} gives a better understanding
of a KMS state on the groupoid crossed product \(\Contz(X)\rtimes
G\). Given a KMS state on \(\Contz(X)\rtimes G\), we disintegrate it
two times. First, by Theorem~\ref{thm:KMS-state} and then by
Neshveyev's theorem~\cite{Neshveyev2013KMS-states}.

Let \(X\) be a \(G\) space with \(m\colon X \to \base \) the
continuous momentum map. Then \(\Contz(\base)\) acts on \(\Contz(X)\)
by \( f\cdot g(x) \defeq f(m(x)) g(x)\) where
\(f\in \Contz(\base), g\in \Contz(X)\). Therefore, \(\Contz(X)\) is a
\(\Contz(\base)\)\nb-algebra with fibre over \(x \in \base\) can be
identify with \(\Contz(m^{-1}(x))\).  We also attach a groupoid
dynamical system by \((\Contz(X), G, \alpha)\) where
\(\alpha = \{\alpha_\gamma\}_{\gamma\in G}\) is given by
\(\alpha_\gamma \colon \Contz(m^{-1}(s(\gamma))) \to
\Contz(m^{-1}(r(\gamma)))\)
\[
  \alpha_\gamma(f)(x) \defeq f(x \cdot \gamma).
\]

\noindent Proposition~\ref{rel:cross:pro-Fell:bundle} says that there
is a Fell bundle associated with the \(G\)\nb-space \(X\). Suppose
\(p\colon \A \to G\) is the Fell bundle associated with the
\(G\)\nb-space \(X\).  And the \(\Cst\)\nb-algebra of this Fell bundle
is isomorphic to \( \Cst(X\rtimes G) \iso \Contz(X) \rtimes G\).  Then
Theorem~\ref{thm:KMS-state} gives us a one-to-one correspondence
between a \(\KMS\) state on \(\Contz(X) \rtimes G\) and a pair
\((\mu,[\Phi])\), where the pair \((\mu, [\Phi])\) satisfying
Condition~\eqref{Cond:I} and \eqref{Cond:II} of
Theorem~\ref{thm:KMS-state}. Assume \(\Phi = \{\phi_u\}_{u\in \base}\)
where \(\phi_u\) is a state on \(\Cst(G^u_u;\A(u))\).  The relation
between \(\phi\) and \(\Phi\) is given by
\begin{equation}\label{equ:rel:space:kms-state}
  \phi(f) = \int_{G^{(0)}} \phi_u(f|_{G^u_u})\dd\mu
\end{equation}
for \(f\in \Cc(G; \A)\).  Since \(G\) act on \(X\) there is an induced
action of the isotropy group \(G^u_u\) on \(m^{-1}_X(u)\). Let
\(m^{-1}_X(u) \rtimes G^u_u\) be the corresponding transformation
group. For \(x\in m^{-1}_X(u)\), let
\( St(x) = \big\{\gamma: (x, \gamma) \in m^{-1}_X(u) \times G^u_u
\textup{ and } x\gamma = x\big\} \subseteq G^u_u \) be the stabiliser
subgroup at \(x\).  Again \(\Cst(G^u_u; \A(u))\) is isomorphic to
\( \Contz(m_X^{-1}(u)) \rtimes G^u_u \iso \Cst(m_X^{-1}(u) \rtimes
G^u_u).  \) Condition~\eqref{Cond:II} of Theorem~\ref{thm:KMS-state},
gives us that \(\Contz(m^{-1}_X(u))\) contained in the centraliser of
\(\phi_u\) for each \(u\in G^{(0)}\). Then
by~\cite[Theorem~1.1]{Neshveyev2013KMS-states}, there is a one-to-one
correspondence between each \(\phi_u\) and pair
\((\nu_u, \{\tau_{u,x}\}_{x\in m^{-1}_X(u)})\) where \(\nu_u\) is a
probability measure on \(m^{-1}_X(u)\) and
\(\{\tau_{u,x}\}_{x\in m^{-1}_X(u)}\) is a \(\nu_u\)\nb-measurable
field of states on \(\Cst(St(x))\) satisfying
\begin{equation}\label{eq:space_state-int-disint}
  \phi_u (g) =\int_{m^{-1}_X(u)} \big(\sum_{\gamma \in St(x)} g(\gamma) \tau_{u,x}(\delta_\gamma) \big)\dd \nu_u(\gamma)
\end{equation}
for \(g\in \Cc(St(x))\) and \(u\in \base\). For \(f\in \Cc(G;\A)\) the
restriction \(f|_{G^u_u}\) can be thought as a function of
\(\Cc(G^u_u;\Contz(m^{-1}_X(u)))\). Therefore, from
Equation~\eqref{equ:rel:space:kms-state} and
Equation~\eqref{eq:space_state-int-disint} we have
\begin{equation}\label{equ:rel:kms-state-measure}
  \phi(f) = \int_{G^{(0)}} \int_{m^{-1}_X(u)}  \big(\sum_{\eta\in St(x)} f(\eta) \tau_{u,x}(\delta_\eta) \big) \dd\nu_u(\eta) \dd\mu
\end{equation}
for \(f \in \Cc(G;\A)\).  Therefore, we have proved the following
result:
\begin{proposition}\label{Prop:kms-state:G-space}
  Let \(G\) be a locally compact, Hausdorff, second countable,
  \etale~groupoid, and \(X\) be a locally compact, Hausdorff (right)
  \(G\)\nb-space. Let \(\beta \in \R\) and \(\sigma\) be the dynamics
  on \(\Contz(X)\rtimes G\) induced by a real-valued \(1\)\nb-cocycle
  \(c\) of \(X\rtimes G\). Then a \(\KMS\) state \(\phi\) on
  \((\Contz(X)\rtimes G, \R, \sigma)\) disintegrates to
  \(\big(\mu, \{\nu_u\}_{u\in \base}, \{\tau_{u,x}\}_{u\in \base, x\in
    m^{-1}_X(u)}\big)\) such that
  \begin{enumerate}
  \item\label{cond-1-GS} \(\mu\) is a quasi-invariant probability
    measure on \(\base\) with Radon--Nikodym derivative
    \(\e^{-\beta c}\);
  \item\label{cond-2-GS} \(\nu_u\) is a probability measure on
    \(m^{-1}_X(u)\) satisfying
    Equation~\eqref{equ:rel:kms-state-measure};
  \item\label{cond-3-GS} \(\{\tau_{u,x}\}_{x\in m^{-1}_X(u)}\) is a \(\nu_u\)\nb-measurable field, where \(\tau_{u,x}\) is a state on \(\Cst(St(x))\) \\
    and the relation between \(\phi_u\) and \(\tau_{u,x}\) given by
    Equation~\eqref{eq:space_state-int-disint}.
  \end{enumerate}
  Conversely, suppose
  \(\big(\mu, \{\nu_u\}_{u\in \base}, \{\tau_{u,x}\}_{u\in \base, x\in
    m^{-1}_X(u)}\big)\) integrates to a \(\KMS\) states on
  \((\Contz(X) \rtimes G, \R, \sigma)\) if the triplet satisfing
  Conditions~\eqref{cond-1-GS}--\eqref{cond-3-GS} of above and the
  integrated states \(\{\phi_u\}_{u\in \base}\) defined by
  Equation~\ref{eq:space_state-int-disint} satisfying
  Condition~\eqref{Cond:II} of Theorem~\ref{thm:KMS-state}.
\end{proposition}

\medskip

\paragraph{\itshape Acknowledgement:}
We are thankful to Jean Renault for discussing the latest development
about this subject and indicating some interesting problems.  We thank
Ralf Meyer for his insightful suggestions to improve
Condition~\eqref{Cond:II} of Theorem~\ref{thm:KMS-state}. This work
was supported by SERB's grant SRG/2020/001823 of the first author and
the CSIR grant~09/1020(0159)/2019-EMR-I of the second one. We are
thankful to the institutions. We are thankful to Bhaskaracharya
Pratishthan, Pune for the pleasant stay during the spring of~2022 when
a substantial work of the project was done.


\begin{bibdiv}
  \begin{biblist}
		
    \bib{Afsar-Sims2021KMS-states-on-Fell-bundle-Cst-alg}{article}{
      author={Afsar, Zahra}, author={Sims, Aidan}, title={K{MS} states
        on the {$C^*$}-algebras of {F}ell bundles over groupoids},
      date={2021}, ISSN={0305-0041}, journal={Math. Proc. Cambridge
        Philos. Soc.}, volume={170}, number={2}, pages={221\ndash
        246}, url={https://doi.org/10.1017/S0305004119000379},
      review={\MR{4222432}}, }
		
    \bib{Bratteli-Robinson1979Oper-alg-Quan-sta-mecha-part-1}{book}{
      author={Bratteli, Ola}, author={Robinson, Derek~W.},
      title={Operator algebras and quantum statistical
        mechanics. {V}ol. 1}, series={Texts and Monographs in
        Physics}, publisher={Springer-Verlag, New York-Heidelberg},
      date={1979}, ISBN={0-387-09187-4}, note={$C^{\ast} $- and
        $W^{\ast} $-algebras, algebras, symmetry groups, decomposition
        of states}, review={\MR{545651}}, }
		
    \bib{Bratteli-Robinson1981Oper-alg-Quan-sta-mech-part-2}{book}{
      author={Bratteli, Ola}, author={Robinson, Derek~W.},
      title={Operator algebras and quantum-statistical
        mechanics. {II}}, series={Texts and Monographs in Physics},
      publisher={Springer-Verlag, New York-Berlin}, date={1981},
      ISBN={0-387-10381-3}, note={Equilibrium states. Models in
        quantum-statistical mechanics}, review={\MR{611508}}, }
		
    \bib{Christensen-thesis}{thesis}{ author={Christensen, Johannes},
      title={{KMS} weights on groupoid
        {C}${^\ast}$\nobreakdash-algebras with an emphasis on graph
        {C}${^\ast}$\nobreakdash-algebras}, type={Ph.D. Thesis},
      date={2018}, }
		
    \bib{Christensen2018Symmetries-of-KMS-symplex}{article}{
      author={Christensen, Johannes}, title={Symmetries of the {KMS}
        simplex}, date={2018}, ISSN={0010-3616},
      journal={Comm. Math. Phys.}, volume={364}, number={1},
      pages={357\ndash 383},
      url={https://doi.org/10.1007/s00220-018-3250-5},
      review={\MR{3861301}}, }
		
    \bib{Christensen-Klaus2016Finite-digraphs-and-KMS-states}{article}{
      author={Christensen, Johannes}, author={Thomsen, Klaus},
      title={Finite digraphs and {KMS} states}, date={2016},
      ISSN={0022-247X}, journal={J. Math. Anal. Appl.}, volume={433},
      number={2}, pages={1626\ndash 1646},
      url={https://doi.org/10.1016/j.jmaa.2015.08.060},
      review={\MR{3398782}}, }
		
    \bib{Cuntz-Vershik2013Cst-Agl-asso-with-endomorphisms}{article}{
      author={Cuntz, Joachim}, author={Vershik, Anatoly},
      title={{$C^\ast$}-algebras associated with endomorphisms and
        polymorphsims of compact abelian groups}, date={2013},
      ISSN={0010-3616}, journal={Comm. Math. Phys.}, volume={321},
      number={1}, pages={157\ndash 179}, review={\MR{3089668}}, }
		
    \bib{Davidson1996Cst-by-Examples}{book}{ author={Davidson,
        Kenneth~R.}, title={{$C^*$}-{A}lgebras by example},
      series={Fields Institute Monographs}, publisher={American
        Mathematical Society, Providence, RI}, date={1996},
      volume={6}, ISBN={0-8218-0599-1}, review={\MR{1402012
          (97i:46095)}}, }
		
    \bib{Exel2008Inverse-semigroup-combinotorial-Cst-alg}{article}{
      author={Exel, Ruy}, title={Inverse semigroups and combinatorial
        {$C^\ast$}-algebras}, date={2008}, ISSN={1678-7544},
      journal={Bull. Braz. Math. Soc. (N.S.)}, volume={39},
      number={2}, pages={191\ndash 313},
      url={https://doi.org/10.1007/s00574-008-0080-7},
      review={\MR{2419901}}, }
		
    \bib{Fell1988Representation-of-Star-Alg-Banach-bundles}{book}{
      author={Fell, J. M.~G.}, author={S., Doran~R.},
      title={Representation ${^\ast}$\nobreakdash-algebras, locally
        compact groups, and {B}anach ${^\ast}$\nobreakdash-algebraic
        bundles, vol. 1 and 2 algebraic bundles}, series={Pure and
        Applied Mathematics}, publisher={Academic Press, New York},
      date={1988}, }
		
    \bib{Gardella-Hirshberg-Santiago2021Rokhlin-dim-Crossed-prod}{article}{
      author={Gardella, Eusebio}, author={Hirshberg, Ilan},
      author={Santiago, Luis}, title={Rokhlin dimension: duality,
        tracial properties, and crossed products}, date={2021},
      ISSN={0143-3857}, journal={Ergodic Theory Dynam. Systems},
      volume={41}, number={2}, pages={408\ndash 460},
      url={https://doi.org/10.1017/etds.2019.68},
      review={\MR{4177290}}, }
		
    \bib{Goehle2009phd-thesis}{book}{ author={Goehle, Geoff},
      title={Groupoid crossed products}, publisher={ProQuest LLC, Ann
        Arbor, MI}, date={2009}, ISBN={978-1109-70354-2},
      url={http://gateway.proquest.com/openurl?url_ver=Z39.88-2004&rft_val_fmt=info:ofi/fmt:kev:mtx:dissertation&res_dat=xri:pqdiss&rft_dat=xri:pqdiss:3397937},
      note={Thesis (Ph.D.)--Dartmouth College}, review={\MR{2941279}},
    }
		
    \bib{Green1978local-structure-of-twisted-covariant-alg}{article}{
      author={Green, Philip}, title={The local structure of twisted
        covariance algebras}, date={1978}, ISSN={0001-5962},
      journal={Acta Math.}, volume={140}, number={3-4},
      pages={191\ndash 250}, review={\MR{0493349 (58 \#12376)}}, }
		
    \bib{Hofmann1977Bundles-and-sheaves-equi-in-Cat-Ban}{inproceedings}{
      author={Hofmann, Karl~Heinrich}, title={Bundles and sheaves are
        equivalent in the category of {B}anach spaces}, date={1977},
      booktitle={{$K$}-theory and operator algebras ({P}roc. {C}onf.,
        {U}niv.  {G}eorgia, {A}thens, {G}a., 1975)}, pages={53\ndash
        69. Lecture Notes in Math., Vol. 575}, review={\MR{0487491}},
    }
		
    \bib{Holkar2017Construction-of-Corr}{article}{ author={Holkar,
        Rohit~Dilip}, title={Topological construction of
        {$C^*$}-correspondences for groupoid {$C^*$}-algebras},
      date={2017}, journal={Journal of {O}perator {T}heory},
      volume={77:1}, number={23-24}, pages={217\ndash 241}, }
		
    \bib{Kadison-Ringrose1997-Vol-II}{book}{ author={Kadison,
        Richard~V.}, author={Ringrose, John~R.}, title={Fundamentals
        of the theory of operator algebras. {V}ol. {II}},
      series={Graduate Studies in Mathematics}, publisher={American
        Mathematical Society, Providence, RI}, date={1997},
      volume={16}, ISBN={0-8218-0820-6},
      url={https://doi.org/10.1090/gsm/016}, note={Advanced theory,
        Corrected reprint of the 1986 original},
      review={\MR{1468230}}, }
		
    \bib{Kumjian1998Fell-bundles-over-gpd}{article}{ author={Kumjian,
        Alex}, title={Fell bundles over groupoids}, date={1998},
      ISSN={0002-9939}, journal={Proc. Amer. Math. Soc.},
      volume={126}, number={4}, pages={1115\ndash 1125},
      url={http://dx.doi.org/10.1090/S0002-9939-98-04240-3},
      review={\MR{1443836 (98i:46055)}}, }
		
    \bib{Kumjian-Pask_Sims2015twisted-higher-rank-graph-alg}{article}{
      author={Kumjian, Alex}, author={Pask, David}, author={Sims,
        Aidan}, title={On twisted higher-rank graph {$C^*$}-algebras},
      date={2015}, ISSN={0002-9947},
      journal={Trans. Amer. Math. Soc.}, volume={367}, number={7},
      pages={5177\ndash 5216},
      url={https://doi.org/10.1090/S0002-9947-2015-06209-6},
      review={\MR{3335414}}, }
		
    \bib{Lance1995Hilbert-modules}{book}{ author={Lance, E.~C.},
      title={Hilbert {$C^*$}-modules}, series={London Mathematical
        Society Lecture Note Series}, publisher={Cambridge University
        Press, Cambridge}, date={1995}, volume={210},
      ISBN={0-521-47910-X},
      url={http://dx.doi.org/10.1017/CBO9780511526206}, note={A
        toolkit for operator algebraists}, review={\MR{1325694
          (96k:46100)}}, }
		
    \bib{Muhly1999Coordinates}{inproceedings}{ author={Muhly, Paul},
      title={Coordinates in operator algebra}, date={1999},
      booktitle={Cmbs conference lecture notes (texas christian
        university, 1990)}, }
		
    \bib{Muhly2001Bundle-over-gpd}{incollection}{ author={Muhly,
        Paul~S.}, title={Bundles over groupoids}, date={2001},
      booktitle={Groupoids in analysis, geometry, and physics
        ({B}oulder, {CO}, 1999)}, series={Contemp. Math.},
      volume={282}, publisher={Amer. Math. Soc., Providence, RI},
      pages={67\ndash 82},
      url={https://doi.org/10.1090/conm/282/04679},
      review={\MR{1855243}}, }
		
    \bib{Muhly-Williams2008FellBundle-ME}{article}{ author={Muhly,
        Paul~S.}, author={Williams, Dana~P.}, title={Equivalence and
        disintegration theorems for {F}ell bundles and their
        {$C^*$}-algebras}, date={2008}, ISSN={0012-3862},
      journal={Dissertationes Math. (Rozprawy Mat.)}, volume={456},
      pages={1\ndash 57}, url={http://dx.doi.org/10.4064/dm456-0-1},
      review={\MR{2446021 (2010b:46146)}}, }
		
    \bib{Muhly-Williams2008Renaults-equivalence-thm-for-gpd-cros-prd}{book}{
      author={Muhly, Paul~S.}, author={Williams, Dana~P.},
      title={Renault's equivalence theorem for groupoid crossed
        products}, series={New York Journal of Mathematics. NYJM
        Monographs}, publisher={State University of New York,
        University at Albany, Albany, NY}, date={2008}, volume={3},
      review={\MR{2547343}}, }
		
    \bib{Munkress1975Topology-book}{book}{ author={Munkres, James~R.},
      title={Topology: a first course}, publisher={Prentice-Hall,
        Inc., Englewood Cliffs, N.J.}, date={1975},
      review={\MR{0464128 (57 \#4063)}}, }
		
    \bib{Murphy-Book}{book}{ author={Murphy, Gerard~J.},
      title={{$C^*$}-algebras and operator theory},
      publisher={Academic Press, Inc., Boston, MA}, date={1990},
      ISBN={0-12-511360-9}, review={\MR{1074574}}, }
		
    \bib{Neshveyev2010Traces-on-crossed-product}{article}{
      author={Neshveyev, S.}, title={Traces on crossed product},
      date={2010}, eprint={arXiv-1010.0600}, }
		
    \bib{Neshveyev2013KMS-states}{article}{ author={Neshveyev,
        Sergey}, title={K{MS} states on the {$C^\ast$}-algebras of
        non-principal groupoids}, date={2013}, ISSN={0379-4024},
      journal={J. Operator Theory}, volume={70}, number={2},
      pages={513\ndash 530},
      url={https://doi.org/10.7900/jot.2011sep20.1915},
      review={\MR{3138368}}, }
		
    \bib{Raeburn-Williams1998ME-book}{book}{ author={Raeburn, Iain},
      author={Williams, Dana~P.}, title={Morita equivalence and
        continuous-trace {$C^*$}-algebras}, series={Mathematical
        Surveys and Monographs}, publisher={American Mathematical
        Society, Providence, RI}, date={1998}, volume={60},
      ISBN={0-8218-0860-5}, url={http://dx.doi.org/10.1090/surv/060},
      review={\MR{1634408 (2000c:46108)}}, }
		
    \bib{Renault1980Gpd-Cst-Alg}{book}{ author={Renault, Jean},
      title={A groupoid approach to {$C^{\ast} $}-algebras},
      series={Lecture Notes in Mathematics}, publisher={Springer,
        Berlin}, date={1980}, volume={793}, ISBN={3-540-09977-8},
      review={\MR{584266 (82h:46075)}}, }
		
    \bib{Renault1985Representations-of-crossed-product-of-gpd-Cst-Alg}{article}{
      author={Renault, Jean}, title={Repr\'esentation des produits
        crois\'es d'alg\`ebres de groupo\"\i des}, date={1987},
      ISSN={0379-4024}, journal={J. Operator Theory}, volume={18},
      number={1}, pages={67\ndash 97}, review={\MR{912813
          (89g:46108)}}, }
		
    \bib{Renault2001AF-equiv-rela-and-cocycle}{incollection}{
      author={Renault, Jean}, title={A{F} equivalence relations and
        their cocycles}, date={2003}, booktitle={Operator algebras and
        mathematical physics ({C}onstan\c{t}a, 2001)},
      publisher={Theta, Bucharest}, pages={365\ndash 377},
      review={\MR{2018241}}, }
		
    \bib{Renault2005The-R-N-problem-for-appx-proper-equiv-relation}{article}{
      author={Renault, Jean}, title={The {R}adon-{N}ikod\'{y}m problem
        for appoximately proper equivalence relations}, date={2005},
      ISSN={0143-3857,1469-4417}, journal={Ergodic Theory
        Dynam. Systems}, volume={25}, number={5}, pages={1643\ndash
        1672}, url={https://doi.org/10.1017/S0143385705000131},
      review={\MR{2173437}}, }
		
    \bib{Rieffel1974Induced-rep}{article}{ author={Rieffel, Marc~A.},
      title={Induced representations of {$C^{\ast} $}-algebras},
      date={1974}, ISSN={0001-8708}, journal={Advances in Math.},
      volume={13}, pages={176\ndash 257}, review={\MR{0353003 (50
          \#5489)}}, }
		
    \bib{Sims-Szabo-Williams2020Operator-alg-book}{book}{
      author={Sims, Aidan}, author={Szab\'{o}, G\'{a}bor},
      author={Williams, Dana}, title={Operator algebras and dynamics:
        groupoids, crossed products, and {R}okhlin dimension},
      series={Advanced Courses in Mathematics. CRM Barcelona},
      publisher={Birkh\"{a}user/Springer, Cham}, date={[2020]
        \copyright 2020}, ISBN={978-3-030-39712-8; 978-3-030-39713-5},
      url={https://doi.org/10.1007/978-3-030-39713-5}, note={Lecture
        notes from the Advanced Course held at Centre de Recerca
        Matem\`atica (CRM) Barcelona, March 13--17, 2017, Edited by
        Francesc Perera}, review={\MR{4321941}}, }
		
    \bib{Sims-Williams2013Amenability-Fell-bundle-over-gpd}{article}{
      author={Sims, Aidan}, author={Williams, Dana~P.},
      title={Amenability for {F}ell bundles over groupoids},
      date={2013}, ISSN={0019-2082}, journal={Illinois J. Math.},
      volume={57}, number={2}, pages={429\ndash 444},
      url={http://projecteuclid.org/euclid.ijm/1408453589},
      review={\MR{3263040}}, }
		
    \bib{Williams2007Crossed-product-Cst-Alg}{book}{ author={Williams,
        Dana~P.}, title={Crossed products of {$C{^\ast}$}-algebras},
      series={Mathematical Surveys and Monographs},
      publisher={American Mathematical Society, Providence, RI},
      date={2007}, volume={134}, ISBN={978-0-8218-4242-3;
        0-8218-4242-0}, url={http://dx.doi.org/10.1090/surv/134},
      review={\MR{2288954 (2007m:46003)}}, }
		
    \bib{Williams2019A-Toolkit-Gpd-algebra}{book}{author={Williams,
        Dana~P.}, title={A tool kit for groupoid {$C^*$}-algebras},
      series={Mathematical Surveys and Monographs},
      publisher={American Mathematical Society, Providence, RI},
      date={2019}, volume={241}, ISBN={978-1-4704-5133-2},
      url={https://doi.org/10.1016/j.physletb.2019.06.021},
      review={\MR{3969970}}, }
		
  \end{biblist}
\end{bibdiv}
\end{document}